\documentclass[oneside,english]{amsart}
\usepackage[T1]{fontenc}
\usepackage[latin9]{inputenc}
\usepackage{babel}
\usepackage{amsthm}
\usepackage{amstext}
\usepackage{amssymb}
\usepackage{stmaryrd}
\usepackage{esint}
\usepackage[unicode=true,pdfusetitle,
 bookmarks=true,bookmarksnumbered=false,bookmarksopen=false,
 breaklinks=false,pdfborder={0 0 0},backref=false,colorlinks=false]
 {hyperref}

\makeatletter

\newcommand{\noun}[1]{\textsc{#1}}

\numberwithin{equation}{section}
\numberwithin{figure}{section}
\theoremstyle{plain}
\newtheorem{thm}{\protect\theoremname}[section]
  \theoremstyle{definition}
  \newtheorem{defn}[thm]{\protect\definitionname}
  \theoremstyle{remark}
  \newtheorem{rem}[thm]{\protect\remarkname}
  \theoremstyle{plain}
  \newtheorem{prop}[thm]{\protect\propositionname}
  \theoremstyle{plain}
  \newtheorem{cor}[thm]{\protect\corollaryname}
  \theoremstyle{plain}
  \newtheorem{question}[thm]{\protect\questionname}
  \theoremstyle{definition}
  \newtheorem{example}[thm]{\protect\examplename}
  \theoremstyle{plain}
  \newtheorem{lem}[thm]{\protect\lemmaname}
  \theoremstyle{remark}
  \newtheorem*{claim*}{\protect\claimname}
\newcommand{\strong}[1]{\textbf{#1}}

\usepackage{tikz-cd}

\DeclareMathOperator{\dom}{dom}
\DeclareMathOperator{\meas}{meas}

\newcommand{\name}{\mathsf{name}}


\makeatother

  \providecommand{\claimname}{Claim}
  \providecommand{\corollaryname}{Corollary}
  \providecommand{\definitionname}{Definition}
  \providecommand{\examplename}{Example}
  \providecommand{\lemmaname}{Lemma}
  \providecommand{\propositionname}{Proposition}
  \providecommand{\questionname}{Question}
  \providecommand{\remarkname}{Remark}
\providecommand{\theoremname}{Theorem}

\begin{document}

\title[Computable randomness and betting]{Computable randomness and betting for computable probability spaces}

\thanks{This work has been partially supported by NSF grant DMS1068829.}
\begin{abstract}
Unlike Martin-Löf randomness and Schnorr randomness, computable randomness
has not been defined, except for a few ad hoc cases, outside of Cantor
space. This paper offers such a definition (actually, several equivalent
definitions), and further, provides a general method for abstracting
``bit-wise'' definitions of randomness from Cantor space to arbitrary
computable probability spaces. This same method is also applied to
give machine characterizations of computable and Schnorr randomness
for computable probability spaces, extending the previously known
results. The paper contains a new type of randomness---endomorphism
randomness---which the author hopes will shed light on the open question
of whether Kol\-mo\-go\-rov-Love\-land randomness is equivalent
to Martin-Löf randomness. The last section contains ideas for future
research.
\end{abstract}

\maketitle
\tableofcontents{}


\section{Introduction\label{sec:Introduction}}

The subjects of measure theory and probability are filled with a number
of theorems stating that some property holds ``almost everywhere''
or ``almost surely.'' Informally, these theorems state that if one
starts with a random point, then the desired result is true. The field
of algorithmic randomness has been very successful in making this
notion formal: by restricting oneself to computable tests for non-randomness,
one can achieve a measure-one set of points that behave as desired.
The most prominent such notion of randomness is Martin-Löf randomness.
However, Schnorr \cite{Schnorr1971} gave an argument---which is now
known as Schnorr's Critique---that Martin-Löf randomness does not
have a sufficiently computable characterization. He offered two weaker-but-more-computable
alternatives: Schnorr randomness and computable randomness. All three
randomness notions are interesting and robust, and further each has
been closely linked to computable analysis, see for example V'yugin
\cite{Vcprimeyugin1998}; Gács, Hoyrup and Rojas \cite{Gacs2011};
Pathak, Rojas and Simpson \cite{Pathak:2014fk}; and Brattka, Miller
and Nies \cite{Brattka2011}. 

Computable randomness, however, is the only one of the three that
has not been defined for arbitrary computable probability spaces.
The usual definition is specifically for Cantor space (i.e.~the space
$2^{\omega}$ of infinite binary sequences), or by analogy, spaces
such as $3^{\omega}$. Namely, a sequence $x\in2^{\omega}$ is said
to be computably random (with respect to the fair-coin measure) if,
roughly speaking, one cannot win arbitrarily large amounts of money
using a computable betting strategy to gamble on the bits of $x$.
(See Section~\ref{sec:CR-on-Cantor} for a formal definition.) While
it is customary to say a real $x\in[0,1]$ is computably random if
its binary expansion is computably random with respect to $2^{\omega}$,
it was only recently shown \cite{Brattka2011} that this is the same
as saying that, for example, the ternary expansion of $x$ is computably
random with respect to $3^{\omega}$. In other words, computable randomness
is base invariant.

In this paper, I use a method for extending the ``bit-wise'' definitions
of randomness with respect to Cantor space to arbitrary computable
probability spaces. The method is based on previous methods given
by Gács \cite{Gacs2005} and later Hoyrup and Rojas \cite{Hoyrup2009}
of dividing a space into cells. However, to successfully extend a
randomness notion (such that the new definition agrees with the former
on $2^{\omega}$), one must show a property similar to base invariance.
I do this for computable randomness.

An outline of the paper is as follows. Section~\ref{sec:CR-on-Cantor}
presents computable randomness with respect to a computable measure $\mu$ on $2^{\omega}$.
In Schnorr's original definition \cite{Schnorr1971} $\mu$ was the
fair-coin measure. There is a prima facie obvious generalization to
all computable measures $\mu$, see for example Muchnik, Semenov and
Uspensky \cite{Muchnik:1998rt} and Bienvenu and Merkle \cite{Bienvenu2009}.
However, as I will show, a bit of care is needed in the pathological
case where $\mu$ has null open sets.

Section~\ref{sec:Background} gives background on computable analysis,
computable probability spaces, and algorithmic randomness.

Section~\ref{sec:Almost-everywhere-decidable} presents the concepts
of an almost-everywhere decidable set (due to Hoyrup and Rojas \cite{Hoyrup2009})
and an a.e.~decidable cell decomposition (which is similar to work
of Hoyrup and Rojas \cite{Hoyrup2009} and Gács \cite{Gacs2005}).
Recall that the topology of $2^{\omega}$ is generated by the collection
of basic open sets of the form $[\sigma]^{\prec}=\{x\in2^{\omega}\mid\sigma\prec x\}$
where $\sigma\prec x$ means $\sigma$ is an initial segment of $x$.
Further, any Borel measure $\mu$ on $2^{\omega}$ is determined by
the values $\mu([\sigma]^{\prec})$. The main idea of this paper is
that for a computable probability space $(\mathcal{X},\mu)$ one can
replace the basic open sets of $2^{\omega}$ (which are decidable)
with an indexed family of ``almost-everywhere decidable'' sets $\{A_{\sigma}\}_{\sigma\in2^{<\omega}}$
which behave in much that same way. I call each such indexed family
a cell decomposition of the space. This allows one to effortlessly
transfer a definition from Cantor space to any computable probability
space.

Section~\ref{sec:CR-on-prob-spaces} applies this method to computable
randomness, giving a variety of equivalent definitions based on martingales
and other tests. More importantly, I show this definition is invariant
under the choice of cell decomposition. Similar to the base-invariance
proof of Brattka, Miller and Nies \cite{Brattka2011}, my proof uses
computable analysis. However, their method does not apply here. (Their
proof uses differentiability and the fact that every atomless measure
on $[0,1]$ is naturally equivalent to a measure on $2^{\omega}$.
The situation is more complicated in the general case. One does not
have differentiability, and one must consider absolutely-continuous
measures instead of mere atomless ones.)

Section~\ref{sec:Kolmogorov-complexity-and-randomness} gives a machine
characterization of computable and Schnorr randomness with respect
to computable probability spaces. This combines the machine characterizations
of computable randomness and Schnorr randomness (respectively, Mihailovi\'{c}
\cite[Thereom 7.1.25]{Downey2010} and Downey, Griffiths, and LaForte
\cite{Downey2004a}) with the machine characterization of Martin-Löf
randomness with respect to arbitrary computable probability spaces
(Gács \cite{Gacs:1980fj} and Hoyrup and Rojas \cite{Hoyrup2009}).

Section~\ref{sec:CR-and-isomorphisms} shows a correspondence between
cell decompositions of a computable probability space $(\mathcal{X},\mu)$
and isomorphisms from $(\mathcal{X},\mu)$ to Cantor space. I also
show computable randomness is preserved by isomorphisms between computable
probability spaces, giving yet another characterization of computable
randomness. However, unlike other notions of randomness, computable
randomness is not preserved by mere morphisms (almost-everywhere computable,
measure-preserving maps).

Section~\ref{sec:Generalizing-randomness} gives three equivalent
methods to extend a randomness notion to all computable probability
measures. It also gives the conditions for when this new randomness
notion agrees with the original one.

Section~\ref{sec:Betting-strategies-and-KL} asks how the method
of this paper applies to Kol\-mo\-go\-rov-Love\-land randomness,
another notion of randomness defined by gambling. The result is that
the natural extension of Kol\-mo\-go\-rov-Love\-land randomness
to arbitrary computable probability measures is Martin-Löf randomness.
However, I do not answer the important open question as to whether
Kol\-mo\-go\-rov-Love\-land randomness and Martin-Löf randomness
are equivalent.

Section~\ref{sec:Endomorphism-randomness} explores a new notion
of randomness between Martin-Löf randomness and Kol\-mo\-go\-rov-Love\-land
randomness, possibly equal to both. It is called endomorphism randomness.

Last, in Section~\ref{sec:Further-directions}, I suggest ways to
generalize the method of this paper to a larger class of isomorphisms
and cell decompositions. I also suggest methods for extending computable
randomness to a larger class of probability spaces, including non-computable
probability spaces, computable topological spaces, and measures defined
by $\pi$-systems. Drawing on Section~\ref{sec:Kolmogorov-complexity-and-randomness},
I suggest a possible definition of K-triviality for computable probability
spaces. Finally, I ask what can be known about the interplay between
randomness, morphisms, and isomorphisms.

\section{Computable randomness with respect to $2^{\omega}$\label{sec:CR-on-Cantor}}

Before exploring computable randomness with respect to arbitrary computable
probability spaces, a useful intermediate step will be to consider
computable probability measures on Cantor space.

We fix notation: $2^{<\omega}$ is the space of finite binary strings;
$2^{\omega}$ is the space of infinite binary sequences; $\varepsilon$
is the empty string; $|\sigma|$ is the length of $\sigma\in2^{<\omega}$;
$\sigma\prec\tau$ and $\sigma\prec x$ mean $\sigma$ is a proper
initial segment of $\tau\in2^{<\omega}$ or $x\in2^{\omega}$; and
$[\sigma]^{\prec}=\{x\in2^{\omega}\mid\sigma\prec x\}$ is a \noun{basic open set}
or \noun{cylinder set}. Also for $\sigma\in2^{<\omega}$ (or $x\in2^{\omega}$),
$\sigma(n)$ is the $n$th digit of $\sigma$ (where $\sigma(0)$
is the ``$0$th'' digit) and $\sigma\upharpoonright n=\sigma(0)\cdots\sigma(n-1)$.

It is assumed that the reader is familiar with partial computable
functions of type $g\colon{\subseteq{}}\mathbb{N}\rightarrow\mathbb{N}$.
(Here, $\mathbb{N}$ can be replaced by an indexed countable set such
as $\mathbb{Q}$, $\mathbb{N}^{<\omega}$, $2^{<\omega}$, etc.) The
\noun{domain} of $g$, written $\dom g$, is the subset of $\mathbb{N}$
for which the algorithm for $g$ halts. For completeness, I give definitions
for partial computable functions of higher type.
\begin{defn}
\label{def:part-fun-seq-to-seq}A partial function $g\colon{\subseteq{}}\mathbb{N}^{\omega}\rightarrow\mathbb{N}^{\omega}$
is \noun{partial computable} if there is a partial computable function
$h\colon{\subseteq{}}\mathbb{N}^{<\omega}\rightarrow\mathbb{N}^{<\omega}$
such that for all $a\in\mathbb{N}^{\omega}$, the following two conditions
hold.
\begin{enumerate}
\item $a\in\dom{g}$ if and only if all of the following hold

\begin{enumerate}
\item $a\upharpoonright n\in\dom{h}$ for all $n$,
\item $h(a\upharpoonright0)\preceq h(a\upharpoonright1)\preceq\ldots$ ,
and 
\item $\lim_{n}|h(a\upharpoonright n)|=\infty$.
\end{enumerate}
\item If $a\in\dom{g}$, then $g(a)=\lim_{n}h(a\upharpoonright n)$.
\end{enumerate}
\end{defn}
Similarly, one can define \noun{partial computable functions} of
type $g\colon{\subseteq{}}\mathbb{N}\rightarrow\mathbb{N}^{\omega}$
and $g\colon{\subseteq{}}\mathbb{N}\rightarrow\mathbb{Q}^{\omega}$
along with their domains.
\begin{defn}
\label{def:part-fun-N-to-R}A partial function $f\colon{\subseteq{}}\mathbb{N}\rightarrow\mathbb{R}$
is \noun{partial computable} if there is a partial computable function
$g\colon{\subseteq{}}\mathbb{N}\rightarrow\mathbb{Q}^{\omega}$ such
that for all $m\in\mathbb{N}$, both of the following hold.
\begin{enumerate}
\item $m\in\dom{f}$ if and only $m\in\dom{g}$.
\item If $m\in\dom{f}$, we have that $g(m)$ is a \noun{Cauchy name}
for $f(m)$---that is, $g(m)=a\in\mathbb{Q}^{\omega}$ where for all
$k<\ell$, $|a(\ell)-a(k)|\leq2^{-k}$ and $\lim_{i}a(i)=f(m)$.
\end{enumerate}
\noindent I will often say $f(n)$ is \noun{undefined} or $f(n)=\mathit{undefined}$,
to denote that $n\notin\dom f$.
\end{defn}
Typically, a \noun{martingale} (on the fair-coin probability measure)
is defined as a function $M\colon2^{<\omega}\rightarrow[0,\infty)$
such that the following property holds for each $\sigma\in2^{<\omega}$:
$M(\sigma)=\frac{1}{2}(M(\sigma0)+M(\sigma1))$. Such a martingale
can be thought of as a betting strategy on coin tosses: the gambler
starts with the value $M(\varepsilon)$ as her capital and bets on
fair coin tosses. Assuming the string $\sigma$ represents the sequence
of coin tosses she has seen so far, $M(\sigma0)$ is the resulting
capital she has if the next coin toss comes up tails, and $M(\sigma1)$
if heads. A martingale $M$ is said to be \noun{computable} if $M(\sigma)$
is computable uniformly from $\sigma$.

A martingale $M$ is said to \noun{succeed} on a sequence $x\in2^{\omega}$
if $\limsup_{n\rightarrow\infty}M(x\upharpoonright n)=\infty$, i.e.~the
gambler wins arbitrary large amounts of money using the martingale
$M$ while betting on the sequence $x$ of coin tosses. By Kolmogorov's
theorem (see \cite[Theorem 6.3.3]{Downey2010}), such a martingale
can only succeed on a measure-zero set of points. A sequence $x\in2^{\omega}$
is said to be \noun{computably random} (with respect to the fair-coin
probability measure) if there does not exist a computable martingale
$M$ which succeeds on $x$.
\begin{defn}
A finite Borel measure $\mu$ on $2^{\omega}$ is a \noun{computable measure}
if the measure $\mu([\sigma]^{\prec})$ of each basic open set is
computable uniformly from $\sigma$. Further, if $\mu(2^{\omega})=1$,
then we say $\mu$ is a \noun{computable probability measure} (on
$2^{\omega}$) and $(2^{\omega},\mu)$ is a \noun{computable probability space}.
\end{defn}
In this paper, measure always means a finite Borel measure. When convenient,
I will drop the brackets and write $\mu(\sigma)$ instead. By the
Carathéodory extension theorem, one may uniquely represent a computable
measure as a computable function $\mu\colon2^{<\omega}\rightarrow[0,\infty)$
such that
\[
\mu(\sigma0)+\mu(\sigma1)=\mu(\sigma)
\]
for all $\sigma\in2^{<\omega}$. I will often confuse a computable
measure on $2^{\omega}$ with its representation on $2^{<\omega}$.

The \noun{fair-coin probability measure} (or the \noun{Lebesgue measure}
on $2^{\omega}$) is the measure $\lambda$ on $2^{\omega}$, defined
by 
\[
\lambda(\sigma)=2^{-|\sigma|}.
\]
(The Greek letter $\lambda$ will always be the fair-coin measure
on $2^{\omega}$, except in a few examples where it is the Lebesgue
measure on $[0,1]^{d}$ or the uniform measure on $3^{\omega}$.)

It is well-known and easy to see that the computable martingales with
respect to $\lambda$ are exactly those of the form $M(\sigma)=\nu(\sigma)/\lambda(\sigma)$
for some computable measure $\nu$. This approach can be used to generalizing
computable randomness to all computable probability measures $\mu$.
\begin{defn}
\label{def:comp-random-cantor} A sequence $x\in2^{\omega}$ is \noun{computably random}
with respect to a computable probability measure $\mu$ if both of
the following conditions occur.
\begin{enumerate}
\item $\mu(x\upharpoonright n)>0$ for all $n$.
\item $\limsup_{n}\nu(x\upharpoonright n)/\mu(x\upharpoonright n)<\infty$
for all computable measures $\nu$.
\end{enumerate}
\end{defn}
Now we can give a martingale characterization of computable randomness.
Basically, we would like our martingale $M$ to satisfy the fairness
condition
\[
M(\sigma0)\mu(\sigma0)+M(\sigma1)\mu(\sigma1)=M(\sigma)\mu(\sigma).
\]
However, we must take care when $\mu(\sigma)=0$ for some $\sigma\in2^{<\omega}$.
(Measures for which $\mu(\sigma)>0$ for all $\sigma\in2^{<\omega}$
are known as \noun{strictly-positive measures}.) For this we will
need partial computable functions of type $M\colon{\subseteq{}}2^{<\omega}\rightarrow\mathbb{R}$
(as in Definition~\ref{def:part-fun-N-to-R}). Recall that $M(\sigma)=\textit{undefined}$
means that $\sigma\notin\dom M$.

\begin{defn}
\label{def:Martingales-Cantor}If $\mu$ is a computable probability
measure on $2^{\omega}$, then a \noun{martingale} $M$ (with respect
to the measure $\mu$) is a partial function $M\colon{\subseteq{}}2^{<\omega}\rightarrow[0,\infty)$
such that the following two conditions hold:
\begin{enumerate}
\item (Impossibility condition) If $\sigma\notin\dom M$ then $\mu(\sigma)=0$.
\item (Fairness condition) For all $\sigma\in2^{<\omega}$ 
\[
M(\sigma0)\mu(\sigma0)+M(\sigma1)\mu(\sigma1)=M(\sigma)\mu(\sigma)
\]
where $\mathit{undefined}\cdot0=0$.
\end{enumerate}

\noindent We say $M$ is an \noun{almost-everywhere computable martingale}
if $M$ is partial computable (as in Definition~\ref{def:part-fun-N-to-R}).
We say $M$ \noun{succeeds} on a sequence $x\in2^{\omega}$ if $\limsup_{n\rightarrow\infty}M(x\upharpoonright n)=\infty$.

\end{defn}

\begin{rem}
By the impossibility condition and the convention $\mathit{undefined}\cdot0=0$,
the fairness condition is well-defined. Indeed $M(\sigma)\mu(\sigma)=0$
whenever $\sigma\notin\dom M$. Also, by the impossibility condition,
$M$ must be total whenever $\mu$ is strictly positive. Therefore,
for the fair-coin measure $\lambda$, or any other strictly positive
measure $\mu$, almost-everywhere computable martingales are the same
as computable martingales (where $M$ is a total function).

Nonetheless, Definition \ref{def:Martingales-Cantor} is technically
in disagreement with the computability theory literature. In computability
theory, a ``martingale'' must be total, while a ``partial martingale''
may not be. However, the martingales in Definition \ref{def:Martingales-Cantor}
deserve the name ``martingale'' for three reasons. First, they agree
in spirit with the standard computability theory definition of martingale---the
only time $M$ is not defined is when an impossible event has occurred.
Second, they agree with the definition of martingale used in probability
theory. Namely, let $M_{n}(x)=M(x\upharpoonright n)$. If $M$ is
a martingale in the sense of Definition \ref{def:Martingales-Cantor},
then $(M_{n})$ is a martingale in the sense of probability theory
(see for example Williams \cite{Williams1991}). In particular, in
probability theory, a martingale is a sequence $(M_{n})$ of random
variables satisfying certain properties. A random variable is only
defined up to almost-everywhere equality. Therefore, it does not matter
if $M_{n}$ is undefined on a measure zero set. (On the other hand,
if $M$ was a partial martingale that was undefined on more that a
positive measure set, then $M_{n}$ would not be a well-defined random
variable and $(M_{n})$ would not be a probability theoretic martingale.)
Third, martingales in the sense of Definition \ref{def:Martingales-Cantor}
are needed to define almost-everywhere computable martingales, and
a number of results in this paper require the use of almost-everywhere
computable martingales. For example, the next proposition shows a
correspondence between computable measures and almost-everywhere computable
martingales. This correspondence does not hold for the usual definition
of computable martingale where $M$ must be total (see Bienvenu and
Merkle \cite[Remark~8]{Bienvenu2009}). \end{rem}
\begin{prop}
\label{prop:mart-meas-correspondence}Let $\mu$ be a computable probability
measure on $2^{\omega}$. (Use the conventions $\mathit{undefined}\cdot0=0$
and $x/0=\mathit{undefined}$ for all $x\in\mathbb{R}$.)
\begin{enumerate}
\item If $M$ is an martingale on $\mu$ (in the sense of Definition \ref{def:Martingales-Cantor}),
then $\nu(\sigma)=M(\sigma)\mu(\sigma)$ defines a measure. If $M$
is almost-everywhere computable, then $\nu$ is computable uniformly
from $M$.
\item If $\nu$ is a measure, then $M(\sigma)=\nu(\sigma)/\mu(\sigma)$
defines a martingale (in the sense of Definition \ref{def:Martingales-Cantor}).
If $\nu$ is computable, then $M$ is almost-everywhere computable
uniformly from $\nu$.
\end{enumerate}
\end{prop}
\begin{proof}
(1) Let $M$ be a martingale. By the fairness condition, $\nu(\sigma)=M(\sigma)\mu(\sigma)$
defines a measure. To compute $\nu$ from $M$, it is enough to compute
$\nu(\sigma)$ by recursion on the length of $\sigma$ as follows.
Set $\nu(\varepsilon)=M(\varepsilon)$. (Note $\varepsilon\in\dom M$
since $\mu(\varepsilon)=1$.) To compute, say, $\nu(\sigma0)$ from
$\nu(\sigma)$, use the (overspecified) formula
\[
\nu(\sigma0)=\begin{cases}
M(\sigma0)\mu(\sigma0) & \text{if }\mu(\sigma0)>0\\
\nu(\sigma)-M(\sigma1)\mu(\sigma1) & \text{if }\mu(\sigma1)>0\\
0 & \text{otherwise}
\end{cases}\ .
\]
This is computable as follows: Wait until it becomes clear that either
$\mu(\sigma0)>0$ or $\mu(\sigma1)>0$, and then use the corresponding
formula. (If both $\mu(\sigma0)>0$ and $\mu(\sigma1)>0$ then both
formulas give the same result.) In the case that $\mu(\sigma)=\mu(\sigma0)=\mu(\sigma1)=0$,
the bounds $0\leq\nu(\sigma0)\leq\nu(\sigma)$ ``squeeze'' $\nu(\sigma0)$
to $0$.

(2) Let $\nu$ be a measure. Then $M$ can be computed from $\nu$
by waiting until $\mu(\sigma)>0$, else $M(\sigma)$ is never defined.
(Here we need the impossibility condition.)\end{proof}
\begin{cor}
Let $\mu$ be a computable measure on $2^{\omega}$. Then x is computably
random with respect to $\mu$ if and only if $\mu(x\upharpoonright n)>0$
for all $n$ and there is no almost-everywhere computable martingale
$M$ which succeeds on $x$.\end{cor}
\begin{rem}
Instead of using almost-everywhere computable martingales, one can
also characterize computable randomness via martingales taking values
in the extended real numbers, i.e.~$M\colon2^{<\omega}\rightarrow[0,\infty]$.
(Use the usual measure-theoretic convention that $\infty\cdot0=0$.)
Fix a computable probability measure $\mu$. Consider the martingale
$M_{0}$ defined by $M_{0}(\sigma)=\lambda(\sigma)/\mu(\sigma)$ where
$\lambda$ is the fair-coin measure. Since, $\lambda(\sigma)>0$ for
all $\sigma$, we have that $M_{0}$ is computable on the extended
real numbers (where $x/0=\infty$ if $x>0$). Notice $M_{0}(\sigma)=\infty$
if and only if $\mu(\sigma)=0$, hence one can ``forget'' the infinite
values to get a finite-valued almost-everywhere computable martingale
$M_{1}$ as in Definition~\ref{def:Martingales-Cantor}. For any
$x\in2^{\omega}$, if $M_{0}$ succeeds on $x$ then either $\mu(x\upharpoonright n)=0$
for some $n$ or $M_{1}(x)$ succeeds on $x$. In either case, $x$
is not computably random with respect to $\mu$. Conversely, if $x\in2^{\omega}$
is not computably random, either $M_{0}$ succeeds on $x$ or there
is some finite-valued almost-everywhere computable martingale $M$
as in Definition~\ref{def:Martingales-Cantor} which succeeds on
$x$. In the later case, $N=M+M_{0}$ is a martingale computable on
the extended real numbers which also succeeds on $x$. However, Proposition~\ref{prop:mart-meas-correspondence}
does not hold for total computable martingales, even if they take
infinite values (Bienvenu and Merkle \cite[Remark~8]{Bienvenu2009}).
Therefore, this paper will only use the finite-valued almost-everywhere
computable martingales in Definition~\ref{def:Martingales-Cantor}.
\end{rem}
I leave as an open question whether computable randomness with respect
to non-strictly positive measures can be characterized by martingales
without using partial functions or infinite values.
\begin{question}
Let $\mu$ be a computable probability measure on $2^{\omega}$, and
assume $x$ is not computably random with respect to $\mu$. Is there
necessary a computable martingale $M\colon2^{<\omega}\rightarrow[0,\infty)$
with respect to $\mu$ which is \emph{total}, \emph{finite-valued}
and succeeds on $x$?\footnote{After reading a preprint of this article, Turetsky \cite{Turetsky:2012kx}
negatively answered this question. He constructed a computable measure
$\mu$ and a sequence $x\in2^{\omega}$ which is not computably random
with respect to $\mu$, but such that no total, finite-valued computable
martingale succeeds on $x$. This shows that almost-everywhere computable
martingales (or martingales with infinite values) are necessary to
characterize computable randomness on such measures.}\end{question}
\begin{rem}
Martingales on $2^{\omega}$ were developed by Ville, and were later
used by Schnorr \cite{Schnorr1971} to define the notion now known
as computable randomness---with respect to the fair-coin measure.
(See Bienvenu, Shafer and Shen \cite{Bienvenu:2009uq} for a history.)
It is not clear to me who was the first to develop computable randomness
with respect to other measures. Schnorr and Fuchs \cite{Schnorr:1977fk}
used ratios of measures, similar to Definition~\ref{def:comp-random-cantor},
to define a related randomness notion, Schnorr randomness, with respect
to other measures. Muchnik, Semenov and Uspensky \cite{Muchnik:1998rt}
considered (total) computable martingales on a smaller class of computable
measures $\mu$, those for which $\sigma\mapsto\mu(\sigma)$ is a
computable map of type $2^{<\omega}\rightarrow\mathbb{Q}$. (For such
measures, $\mu(\sigma)=0$ is decidable. Therefore, there is a correspondence
between computable measures $\nu$ and total computable martingales
$M$. Namely, let $M(\sigma)=\nu(\sigma)/\mu(\sigma)$ when $\mu(\sigma)>0$
and $M(\sigma)=|\sigma|$ otherwise.)  Bienvenu and Merkle \cite{Bienvenu2009}
considered computable randomness with respect to computable measures
$\mu$, defined via (total) computable martingales. Most of their
results are for strictly positive  (``nowhere vanishing'' in their
paper) measures $\mu$. This section contributes a martingale characterization
of computable randomness which works for all computable measures $\mu$,
including those which are not strictly positive.
\end{rem}
See Downey and Hirschfelt \cite[Section 7.1]{Downey2010} and Nies
\cite[Chapter 7]{Nies2009} for more information on computable randomness
with respect to $(2^{\omega},\lambda)$.

\section{\label{sec:Background}Computable probability spaces and algorithmic
randomness}

In this section I give some background on computable analysis, computable
probability spaces, and algorithmic randomness.

\subsection{Computable analysis and computable probability spaces\label{sub:Comp-prob-spaces}}

Here I pre\-sent computable Polish spaces and computable probability
spaces. For a more detailed exposition of the same material see Hoyrup
and Rojas \cite{Hoyrup2009}. This paper assumes some familiarity
with basic computability theory and computable analysis, as in Pour
El and Richards \cite{Pour-El1989}, Weihrauch \cite{Weihrauch2000},
or Brattka et al.\ \cite{Brattka2008}.
\begin{defn}
A \noun{computable Polish space} (or \noun{computable metric space})
is a triple $(X,d,S)$ such that
\begin{enumerate}
\item $X$ is a complete metric space with metric $d\colon X\times X\rightarrow[0,\infty)$.
\item $S=\{a_{i}\}_{i\in\mathbb{N}}$ is a countable dense subset of $X$
(the \noun{simple points} of $\mathcal{X}$).
\item The distance $d(a,b)$ is computable uniformly from $a,b\in S$.\footnote{That is, $d(a_{i},a_{j})$ is computable uniformly from $i$ and $j$.
Equivalently, $d\colon S\times S\rightarrow\mathbb{R}$ is a computable
function. (Here the countable set $S$ is identified with the natural
numbers.)}
\end{enumerate}

\noindent A \noun{Cauchy name} for a point $x\in X$ is a sequence
$a\in S^{\omega}$ such that $d(a(\ell),a(k))\leq2^{-k}$ for all
$\ell\geq k$ and $x=\lim_{k}a(k)$, and $x$ is \noun{computable}
if it has a computable Cauchy name.

\end{defn}
The \noun{basic open balls} are the sets of the form $B(a,r)=\{x\in X\mid d(x,a)<r\}$
where $a\in S$ and $r>0$ is rational. The $\Sigma_{1}^{0}$ sets
(\noun{effectively open sets}) are those of the form $\bigcup_{(a,q)\in I}B(a,q)$
for a computably enumerable set $I\subseteq S\times\mathbb{Q}_{>0}$;
$\Pi_{1}^{0}$ sets (\noun{effectively closed sets}) are the complements
of $\Sigma_{1}^{0}$ sets; $\Sigma_{2}^{0}$ sets are computable unions
of $\Pi_{1}^{0}$ sets; and $\Pi_{2}^{0}$ sets are computable intersections
of $\Sigma_{1}^{0}$ sets. A function $f\colon\mathcal{X}\rightarrow\mathbb{R}$
is \noun{computable (-ly continuous)} if for each $\Sigma_{1}^{0}$
set $U$ in $\mathbb{R}$, the set $f^{-1}(U)$ is a $\Sigma_{1}^{0}$
subset of $\mathcal{X}$ (uniformly in $U$), or equivalently, there
is an algorithm which sends every Cauchy-name of $x$ to a Cauchy-name
of $f(x)$. A function $f\colon\mathcal{X}\rightarrow[0,\infty]$
is \noun{lower semicomputable} if it is the supremum of a computable
sequence of computable functions $f_{n}\colon\mathcal{X}\rightarrow[0,\infty)$.

A real $x$ is said to be \noun{lower (upper) semicomputable} if
$\{q\in\mathbb{Q}\mid q<x\}$ (respectively $\{q\in\mathbb{Q}\mid q>x\}$)
is a c.e.~set.
\begin{defn}
\label{def:comp_prob_meas}If $\mathcal{X}=(X,d,S)$ is a computable
Polish space, then a Borel measure $\mu$ is a \noun{computable measure}
on $\mathcal{X}$ if the value $\mu(X)$ is computable, and for each
$\Sigma_{1}^{0}$ set $U$, the value $\mu(U)$ is lower semicomputable
uniformly from $U$. A \noun{computable probability space} is a pair
$(\mathcal{X},\mu)$ where $\mathcal{X}$ is a computable Polish space,
$\mu$ is a computable measure on $\mathcal{X}$, and $\mu(\mathcal{X})=1$. 
\end{defn}
While this definition of computable probability space may seem ad
hoc, it turns out to be equivalent to a number of other definitions.
In particular, the computable probability measures on $\mathcal{X}$
are exactly the computable points in the space of probability measures
under the Prokhorov metric. Also, a probability space is computable
precisely if the integral operator is a computable operator on computable
functions $f\colon\mathcal{X}\rightarrow[0,1]$. See Hoyrup and Rojas
\cite{Hoyrup2009} and Schröder \cite{Schroder:2007kx} for details.

I will often confuse a metric space or a probability space with its
set of points, e.g.\ writing $x\in\mathcal{X}$ or $x\in(\mathcal{X},\mu)$
to mean that $x\in X$ where $\mathcal{X}=(X,d,S)$.

\subsection{Algorithmic randomness\label{sub:algorithmic-randomness}}

In this section I give background on algorithmic randomness. Namely,
I present three types of tests for Martin-Löf and Schnorr randomness.
In Section~\ref{sec:CR-on-prob-spaces}, I will generalize these
tests to computable randomness, building on the work of Merkle, Mihailovi\'{c}
and Slaman \cite{Merkle2006} (which is similar to that of Downey,
Griffiths and LaForte \cite{Downey2004a}). I also present Kurtz randomness.

Throughout this section, let $(\mathcal{X},\mu)$ be a computable
probability space.
\begin{defn}
A \noun{Martin-Löf test} (with respect to $(\mathcal{X},\mu)$) is
a computable sequence of $\Sigma_{1}^{0}$ sets $(U_{n})$ such that
$\mu(U_{n})\leq2^{-n}$ for all $n$. A \noun{Schnorr test} is a
Martin-Löf test such that $\mu(U_{n})$ is also computable uniformly
from $n$. We say $x$ is \noun{covered by} the test $(U_{n})$ if
$x\in\bigcap_{n}U_{n}$.
\end{defn}

\begin{defn}
We say $x\in\mathcal{X}$ is \noun{Martin-Löf random} (with respect
to $(\mathcal{X},\mu)$) if there is no Martin-Löf test which covers
$x$. We say $x$ is \noun{Schnorr random} if there is no Schnorr
test which covers $x$. We say $x$ is \noun{Kurtz random} (or \noun{weak random})
if $x$ is not in any null $\Pi_{1}^{0}$ set (or equivalently a null
$\Sigma_{2}^{0}$ set).
\end{defn}

It is easy to see that for all computable probability spaces 
\[
\text{Martin-Löf random}\ \rightarrow\ \text{Schnorr random}\ \rightarrow\ \text{Kurtz random}.
\]
It is also well-known (see \cite{Downey2010,Nies2009}) on $(2^{\omega},\lambda)$
that 
\begin{equation}
\text{Martin-Löf random}\ \rightarrow\ \text{Computably random}\ \rightarrow\ \text{Schnorr random}\ \rightarrow\ \text{Kurtz random.}\label{eq:ran-hierarchy}
\end{equation}
In the next section, after defining computable randomness with respect
to computable probability spaces, I will show (\ref{eq:ran-hierarchy})
holds for all computable probability spaces.

Next, I mention two other useful tests.
\begin{defn}
\label{def:solovay-integral}A \noun{Solovay test} is a computable
sequence $(U_{n})$ of $\Sigma_{1}^{0}$ sets such that $\sum_{n}\mu(U_{n})<\infty$.
We say $x$ is \noun{Solovay covered} by a Solovay test $(U_{n})$
if $x\in U_{n}$ for infinitely many $n$. An \noun{integral test}
is a lower semicomputable function $g\colon\mathcal{X}\rightarrow[0,\infty]$
such that $\int g\,d\mu<\infty$.\end{defn}
\begin{thm}
For $x\in\mathcal{X}$, the following are equivalent.
\begin{enumerate}
\item $x$ is Martin-Löf random.
\item $x$ is not Solovay covered by any Solovay test.
\item $g(x)<\infty$ for all integral tests.
\end{enumerate}
\end{thm}

\begin{thm}
\label{thm:schnorr-tests}For $x\in\mathcal{X}$, the following are
equivalent.
\begin{enumerate}
\item $x$ is Schnorr random.
\item $x$ is not Solovay covered by any Solovay test $(U_{n})$ such that
$\sum_{n}\mu(U_{n})$ is computable.
\item $g(x)<\infty$ for all integral tests $g$ such that $\int\!g\,d\mu=1$.
\end{enumerate}
\end{thm}
\begin{rem}
For a history of the tests for Schnorr and Martin-Löf randomness see
Downey and Hirschfelt \cite{Downey2010}. The integral test characterization
for Schnorr randomness is due to Miyabe \cite{Miyabe:2013uq} and
was also independently communicated to me by Hoyrup and Rojas.
\end{rem}
I will give Solovay and integral test characterizations of computable
randomness in Section~\ref{sec:CR-on-prob-spaces}.

There are also martingale characterizations of Martin-Löf and Schnorr
randomness with respect to $2^{\omega}$, but they will not be needed.

\section{\label{sec:Almost-everywhere-decidable}Almost-everywhere decidable
cell decompositions}

The main idea of this paper is that ``bit-wise'' definitions of
randomness with respect to $2^{\omega}$, such as computable randomness,
can be extended to arbitrary computable probability spaces by replacing
the basic open sets $[\sigma]^{\prec}$ on $2^{\omega}$ with an indexed
family $\{A_{\sigma}\}_{\sigma\in2^{<\omega}}$ of a.e.\ decidable
sets. This idea was introduced by Gács \cite{Gacs2005} and further
developed by Hoyrup and Rojas \cite{Hoyrup2009}. My method is based on
theirs, although the presentation and definitions differ on a few
key points.

Recall that a set $A\subseteq\mathcal{X}$ is \noun{decidable} if
both $A$ and its complement $\mathcal{X}\smallsetminus A$ are $\Sigma_{1}^{0}$
sets (equivalently $A$ is both $\Sigma_{1}^{0}$ and $\Pi_{1}^{0}$).
The intuitive idea is that from the code for any $x\in\mathcal{X}$,
one may effectively decide if $x$ is in $A$ or its complement. On
$2^{\omega}$, the cylinder sets $[\sigma]^{\prec}$ are decidable.
Unfortunately, a space such as $\mathcal{X}=[0,1]$ has no non-trivial
clopen sets, and therefore no non-trivial decidable sets. However,
using the idea that null measure sets can be ignored, we can use ``almost-everywhere
decidable sets'' instead.
\begin{defn}[Hoyrup and Rojas \cite{Hoyrup2009}]
\label{def:ae-decidable}Let $(\mathcal{X},\mu)$ be a computable
probability space. A pair $U,V\subseteq\mathcal{X}$ is a \noun{$\mu$-a.e.~decidable pair}
if
\begin{enumerate}
\item $U$ and $V$ are both $\Sigma_{1}^{0}$ sets,
\item $U\cap V=\varnothing$, and
\item $\mu(U\cup V)=1$.
\end{enumerate}

\noindent A set $A$ is a \noun{$\mu$-a.e.~decidable set} if there
is a $\mu$-a.e.~decidable pair $U,V$ such that $U\subseteq A\subseteq\mathcal{X}\smallsetminus V$.
The code for the $\mu$-a.e.~decidable set $A$ is the pair of codes
for the $\Sigma_{1}^{0}$ sets $U$ and $V$. 

\end{defn}
Hoyrup and Rojas \cite{Hoyrup2009} also required that $U\cup V$
be dense for technical reasons. We will relax this condition, since
it is not necessary. They also use the terminology ``almost decidable
set''.

Definition~\ref{def:ae-decidable} is an effectivization of \noun{$\mu$-continuity set},
i.e.~a set with $\mu$-null boundary. In Definition~\ref{def:ae-decidable},
the topological boundary of $A$ is a subset of the null set $\mathcal{X}\smallsetminus(U\cup V)$.
(Since, $U\cup V$ is not required to be dense, $\mathcal{X}\smallsetminus(U\cup V)$
may contain null open sets, but these will not present a problem.)

Not every $\Sigma_{1}^{0}$ set is a.e.~decidable; for example, take
a dense $\Sigma_{1}^{0}$ set with measure less than one. However,
any basic open ball $B(a,r)$ is a.e.~decidable provided that $\{x\mid d(a,x)=r\}$
has null measure. (Again, if we require the boundary to be nowhere
dense, the situation is more subtle. See the discussion in Hoyrup
and Rojas \cite{Hoyrup2009}.) Further, the closed ball $\overline{B}(a,r)$
is also a.e.~decidable with the same code. Any two a.e.~decidable
sets with the same code will be considered the same for our purposes.
Hence, I will occasionally say $x\in A$ (respectively $x\notin A)$,
when I mean $x\in U$ (respectively $x\notin V)$ for the corresponding
a.e.~decidable pair $(U,V)$.
\begin{rem}
\label{rem:a.e.-decidable-def-2}If $(U,V)$ is a pair of sets satisfying
conditions (1) and (3) in Definition~\ref{def:ae-decidable} and
satisfying $\mu(U\cap V)=0$ (in place of condition (2)), then $U$
is an a.e.~decidable set as follows. By the definition of a $\Sigma_{1}^{0}$
set, $U$ and $V$ are both unions of c.e.~listings of basic open
balls. Let $U'$ and $V'$ be the same unions, except without the
balls of measure $0$. Then $U'$ and $V'$ are still $\Sigma_{1}^{0}$,
$U'\subseteq U\subseteq\mathcal{X}\smallsetminus V'$, $U'=U$ a.e.,
$V'=V$ a.e., and $U'\cap V'=\varnothing$.
\end{rem}

\begin{defn}[Inspired by Hoyrup and Rojas \cite{Hoyrup2009}]
\label{def:a.e.-decidable-repr}Let $(\mathcal{X},\mu)$ be a computable
probability space, and let $\mathcal{A}=\{A_{i}\}_{i\in\mathbb{N}}$
be a computably indexed family of a.e.~decidable sets. Let the computably
indexed family $\mathcal{B}=\{B_{i}\}_{i\in\mathbb{N}}$ be the closure
of $\mathcal{A}$ under finite Boolean combinations. We say $\mathcal{A}$
is an \noun{a.e.~decidable generator} (or \noun{generator} for short)
of $(\mathcal{X},\mu)$ if given a $\Sigma_{1}^{0}$ set $U\subseteq\mathcal{X}$
one can find (effectively from the code of $U$) a c.e.~set $I$
of indices (possibly finite or empty) such that $U=\bigcup_{i\in I}B_{i}$
$\mu$-a.e.
\end{defn}
Such a family $\mathcal{A}$ is called a ``generator'' of $(\mathcal{X},\mu)$
since it generates the measure algebra of $\mu$.\footnote{Recall, that the measure algebra of $(\mathcal{X},\mu)$ is the set
of equivalence classes of Borel measurable sets under $\mu$-a.e.\ equivalence.
(This is an algebra since it contains (the equivalence class of) $\varnothing$
and is closed under the operations of countable union and complement.)
Closing $\mathcal{A}$ under the countable unions generates all (equivalence
classes of) open sets. Therefore, closing $\mathcal{A}$ under countable
unions and complements generates all (equivalence classes of) measurable
sets.} Hoyrup and Rojas \cite{Hoyrup2009} show that not only does such
a generator $\mathcal{A}$ exist for each $(\mathcal{X},\mu)$, but
it can be taken to be a basis of the topology, hence they call $\mathcal{A}$
a ``basis of almost decidable sets''. I will not require that $\mathcal{A}$
is a basis.
\begin{thm}[Hoyrup and Rojas \cite{Hoyrup2009}]
\label{thm:a.e.- decidable-repr}Let $(\mathcal{X},\mu)$ be a computable
probability space. There exists an a.e.~decidable generator $\mathcal{A}$
of $(\mathcal{X},\mu)$, which is computable uniformly from (the code
for) $(\mathcal{X},\mu)$.
\end{thm}
The main idea of the proof for Theorem~\ref{thm:a.e.- decidable-repr}
is to start with the collection of basic open balls centered at simple
points with rational radii. While, these may not have null boundary,
a basic diagonalization argument (similar to the proof of the Baire
category theory, see \cite{Brattka2001}) can be used to calculate
a set of radii approaching zero for each simple point such that the
resulting balls are all a.e.~decidable. Similar arguments have been
given by Bosserhoff \cite{Bosserhoff2008} and Gács \cite{Gacs2005}.
The technique is related to Bishop's theory of profiles \cite[Section 6.4]{Bishop1985}
and to ``derandomization'' arguments (see Freer and Roy \cite{Freer2011}
for example).

From a generator we can decompose $\mathcal{X}$ into a.e.~decidable
cells. This is the indexed family $\{A_{\sigma}\}_{\sigma\in2^{<\omega}}$
mentioned in the introduction.
\begin{defn}
\label{def:A-names}Let $\mathcal{A}=\{A_{i}\}$ be an a.e.~decidable
generator of $(\mathcal{X},\mu)$. Recall each $A_{i}$ is coded by
an a.e.~decidable pair $(U_{i},V_{i})$ where $U_{i}\subseteq A_{i}\subseteq\mathcal{X}\smallsetminus V_{i}$.
For $\sigma\in2^{\omega}$ of length $s$ define $[\sigma]_{\mathcal{A}}=A_{0}^{\sigma(0)}\cap A_{1}^{\sigma(1)}\cap\cdots\cap A_{s-1}^{\sigma(s-1)}$
where for each $i$, $A_{i}^{0}=U_{i}$ and $A_{i}^{1}=V_{i}$, and
$[\varepsilon]_{\mathcal{A}}=\mathcal{X}$. When possible, define
$x\upharpoonright_{\mathcal{A}}n$ as the unique $\sigma$ of length
$n$ such that $x\in[\sigma]_{\mathcal{A}}$. Also when possible,
define the \noun{$\mathcal{A}$-name} of $x$ as the binary sequence
$\name_{\mathcal{A}}(x)=\lim_{n\rightarrow\infty}x\upharpoonright_{\mathcal{A}}n$.
A point without an $\mathcal{A}$-name will be called an \noun{unrepresented point}.
Each $[\sigma]_{\mathcal{A}}$ will be called a \noun{cell}, and
the collection of $\{[\sigma]_{\mathcal{A}}\}_{\sigma\in2^{<\omega}}$
will be called an \noun{(a.e.\ decidable) cell decomposition} of
$(\mathcal{X},\mu)$.
\end{defn}
\begin{flushleft}
The choice of notation allows one to quickly translate between Cantor
space and the space $(\mathcal{X},\mu)$. Gács \cite{Gacs2005} and
others refer to the cell $[x\upharpoonright_{\mathcal{A}}n]_{\mathcal{A}}$
as the \noun{$n$-cell} of $x$ and write it as $\Gamma_{n}(x)$.
\par\end{flushleft}
\begin{rem}
There are two types of ``bad points'', unrepresented points and
points $x\in[\sigma]_{\mathcal{A}}$ where $\mu([\sigma]_{\mathcal{A}})=0$.
The set of ``bad points'' is a null $\Sigma_{2}^{0}$ set, so each
``bad point'' is not even Kurtz random! One may also go further,
and for each generator $\mathcal{A}$ compute another $\mathcal{A}'$
such that $[\sigma]_{\mathcal{A}}=[\sigma]_{\mathcal{A}'}$ a.e.,\ but
$\mu([\sigma]_{\mathcal{A}})=0$ if and only if $[\sigma]_{\mathcal{A}'}=\varnothing$.
Then all the ``bad points'' would be unrepresented points.\end{rem}
\begin{example}
Consider a computable measure $\mu$ on $2^{\omega}$. Let $A_{i}=\{x\in2^{\omega}\mid x(i)=1\}$
where $x(i)$ is the $i$th bit of $x$. Then $\mathcal{A}=\{A_{i}\}$
is a generator of $(2^{\omega},\mu)$. Further $[\sigma]_{\mathcal{A}}=[\sigma]^{\prec}$,
$x\upharpoonright_{\mathcal{A}}n=x\upharpoonright n$, and $\name_{\mathcal{A}}(x)=x$.
Call $\mathcal{A}$ the \noun{natural generator} of $(2^{\omega},\mu)$,
and $\{[\sigma]^{\prec}\}_{\sigma\in2^{<\omega}}$ the \noun{natural cell decomposition}.
\end{example}
In this next proposition, recall that a set $S\subseteq2^{<\omega}$
is \noun{prefix-free} if there is no pair $\tau,\sigma\in S$ such
that $\tau\prec\sigma$. 
\begin{prop}
\label{prop:decompose-open-set}Let $(\mathcal{X},\mu)$ be a computable
probability space with generator $\mathcal{A}$ and $\{[\sigma]_{\mathcal{A}}\}_{\sigma\in2^{<\omega}}$
the corresponding cell decomposition. Then for each $\Sigma_{1}^{0}$
set $U\subseteq\mathcal{X}$ there is a c.e.~set $\{\sigma_{i}\}$
(c.e.\ in the code for $U$) such that $U=\bigcup_{i}[\sigma_{i}]_{\mathcal{A}}$
a.e. Further, $\{\sigma_{i}\}$ can be assumed to be prefix-free and
such that $\mu([\sigma_{i}]_{\mathcal{A}})>0$ for all $i$.\end{prop}
\begin{proof}
Straight-forward from Definitions~\ref{def:ae-decidable} and \ref{def:a.e.-decidable-repr}.
\end{proof}
It is clear that a generator $\mathcal{A}$ is determined by its cell
decomposition $\{[\sigma]_{\mathcal{A}}\}_{\sigma\in2^{<\omega}}$,
namely 
\[
A_{i}=\bigcup_{\sigma\in2^{i}}[\sigma1]_{\mathcal{A}}\quad\text{a.e.}
\]
Hence we will often confuse a generator and its cell decomposition
writing both as $\mathcal{A}$. (I will still use the notation $A_{i}$
for the sets in the generator, and $[\sigma]_{\mathcal{A}}$ for the
sets in the cell decomposition.) Say that $\mathcal{A}=\mathcal{B}$
a.e.~if $A_{i}=B_{i}$ a.e.\ for all $i$ (or equivalently, $[\sigma]_{\mathcal{A}}=[\sigma]_{\mathcal{B}}$
a.e.\ for all $\sigma$). Further, this next proposition gives the
criterion for when an indexed family $\{B_{\sigma}\}_{\sigma\in2^{<\omega}}$
is a.e.~equal to an a.e.\ decidable cell decomposition.
\begin{prop}
\label{prop:net-to-representation}Let $(\mathcal{X},\mu)$ be a computable
probability space. Let $\{B_{\sigma}\}_{\sigma\in2^{<\omega}}$ be
a computably indexed family of $\Sigma_{1}^{0}$ sets. Then $\{B_{\sigma}\}$
is a.e.~equal to an a.e.~decidable cell decomposition corresponding
to some a.e.~decidable generator $\mathcal{A}$ (that is, $B_{\sigma}=[\sigma]_{\mathcal{A}}$
a.e.\ for all $\sigma\in2^{<\omega}$) if and only if $\{B_{\sigma}\}$
satisfies the following conditions,
\begin{enumerate}
\item $B_{\varepsilon}=\mathcal{X}$ a.e.,
\item for all $\sigma\in2^{\omega}$, $B_{\sigma0}\cap B_{\sigma1}=\varnothing$
a.e.\ and $B_{\sigma0}\cup B_{\sigma1}=B_{\sigma}$ a.e., and
\item for each $\Sigma_{1}^{0}$ set $U\subseteq\mathcal{X}$ there is a
c.e.~set $\{\sigma_{i}\}$ (c.e.\ in the code for $U$) such that
$U=\bigcup_{i}B_{\sigma_{i}}$ a.e.
\end{enumerate}
\end{prop}
\begin{proof}
It is straight-forward from Definitions~\ref{def:ae-decidable} and
\ref{def:A-names} as well as Proposition~\ref{prop:decompose-open-set}
that every a.e.~decidable cell decomposition $\{[\sigma]_{\mathcal{A}}\}_{\sigma\in2^{<\omega}}$
satisfies (1)--(3).

For the other direction, fix $\{B_{\sigma}\}_{\sigma\in2^{<\omega}}$
satisfying (1)--(3). Let $\mathcal{A}=\{A_{i}\}_{i\in\mathbb{N}}$
be given by 
\[
A_{i}=\bigcup_{\sigma\in2^{i}}B_{\sigma1}.
\]
Define $[\sigma]_{\mathcal{A}}$ as in Definition~\ref{def:A-names}.
Then $B_{\sigma}=[\sigma]_{\mathcal{A}}$ a.e.~by induction on the
length of $\sigma$ as follows (using (1) and (2)), 
\begin{align*}
[\varepsilon]_{\mathcal{A}} & =\mathcal{X}=B_{\varepsilon}\quad\text{a.e.}\\
{}[\sigma1]_{\mathcal{A}} & =[\sigma]_{\mathcal{A}}\cap A_{|\sigma|}=B_{\sigma}\cap\bigcup_{\tau=2^{|\sigma|}}B_{\tau1}=B_{\sigma1}\quad\text{a.e.}\\
{}[\sigma0]_{\mathcal{A}} & =[\sigma]_{\mathcal{A}}\cap A_{|\sigma|}^{c}=B_{\sigma}\cap\bigcup_{\tau=2^{|\sigma|}}B_{\tau0}=B_{\sigma0}\quad\text{a.e.}
\end{align*}
Now it remains to show that $\mathcal{A}$ is an a.e.~decidable generator.
By Remark~\ref{rem:a.e.-decidable-def-2} $A_{i}$ is a.e.~decidable
in $i$ by setting $U_{i}=A_{i}=\bigcup_{\tau\in2^{i}}B_{\tau1}$
and setting $V_{i}=\bigcup_{\tau\in2^{i}}B_{\tau0}$. Finally, $\mathcal{A}$
is a generator by (3).
\end{proof}
Each computable probability space $(\mathcal{X},\mu)$ can be represented
by a cell decomposition $\mathcal{A}$ along with the values $\mu([\sigma]_{\mathcal{A}})$.
Gács \cite{Gacs2005} and Hoyrup and Rojas \cite{Hoyrup2009} pick
a canonical cell decomposition for each $(\mathcal{X},\mu)$ based
on a generator $\mathcal{A}$ which is also a topological basis of
the space. In this paper, I will not fix a canonical cell decomposition,
instead showing that computable randomness can be defined independently
of the choice of cell decomposition (Theorem~\ref{thm:indepentant-of-representation})
and that there is a one-to-one correspondence between cell decompositions
of $(\mathcal{X},\mu)$ and isomorphisms between $(\mathcal{X},\mu)$
and probability measures on $2^{\omega}$ (Proposition~\ref{prop:isomorphism-to-representation}).

\section{Computable randomness with respect to computable probability spaces\label{sec:CR-on-prob-spaces}}

In this section I define computable randomness with respect to a computable
probability space. As a first step, I have already done this for spaces
$(2^{\omega},\mu)$. The second step will be to define computable
randomness with respect to a particular cell decomposition of the
space. Finally, the last step is Theorem~\ref{thm:indepentant-of-representation},
where I will show the definition is invariant under the choice of
cell decomposition.

There are two characterizations of computable randomness with respect
to $(2^{\omega},\lambda)$ using Martin-Löf tests that were obtained
independently. While they differ in formulation, they are essentially
the same. The first is due to Downey, Griffiths, and LaForte \cite{Downey2004a}.
However, I will use the other due to Merkle, Mihailovi\'{c}, and Slaman
\cite{Merkle2006}. 
\begin{defn}[Merkle et al.\ \cite{Merkle2006}]
\label{def:bounded_ML_test_cantor}On $(2^{\omega},\lambda)$ a Martin-Löf
test $(U_{n})$ is called a \noun{bounded Martin-Löf test} if there
is a computable measure $\nu\colon2^{<\omega}\rightarrow[0,\infty)$
such that for all $n\in\mathbb{N}$ and $\sigma\in2^{<\omega}$ 
\[
\lambda(U_{n}\cap[\sigma]^{\prec})\leq2^{-n}\nu(\sigma).
\]

\noindent We say that the test $(U_{n})$ is \noun{bounded by} the
measure $\nu$.
\end{defn}

\begin{thm}[Downey et al.\ \cite{Downey2004a} and Merkle et al.\ \cite{Merkle2006}]
\label{thm:bounded_ML_test_cantor}On $(2^{\omega},\lambda)$, a
sequence $x\in2^{\omega}$ is computably random if and only if $x$
is not covered by any bounded Martin-Löf test.
\end{thm}
The next theorem and definition give five equivalent types of tests
for computable randomness with respect to a cell decomposition $\mathcal{A}$.
(I also give a machine characterization of computable randomness in
Section~\ref{sec:Kolmogorov-complexity-and-randomness}.) The integral
test and Solovay test are new for computable randomness, although
they are implicit in the proof of Theorem~\ref{thm:bounded_ML_test_cantor}.
\begin{thm}
\label{thm:comp_rand_defs}Let $\mathcal{A}$ be a cell decomposition
of the computable probability space $(\mathcal{X},\mu)$. If $x\in\mathcal{X}$
is neither an unrepresented point nor in a null cell, then the following
are equivalent.
\begin{enumerate}
\item \label{enu:mart}(Martingale test) There is an a.e.\ computable martingale
$M\colon{\subseteq{}}2^{<\omega}\rightarrow[0,\infty)$ satisfying
for all $\sigma\in2^{<\omega}$ 
\begin{gather*}
\sigma\notin\dom M\quad\rightarrow\quad\mu([\sigma]_{\mathcal{A}})=0\\
M(\sigma0)\mu([\sigma0]_{\mathcal{A}})+M(\sigma1)\mu([\sigma1]_{\mathcal{A}})=M(\sigma)\mu([\sigma]_{\mathcal{A}})\quad(\textit{undefined}\cdot0=0)
\end{gather*}
such that $\limsup_{n\rightarrow\infty}M(x\upharpoonright_{\mathcal{A}}n)=\infty$.
(Say $M$ wins on $x$.)
\item \label{enu:mart_savings}(Martingale test with savings property, see
for example \cite[Proposition 2.3.8]{Downey2010}) There is an a.e.\ computable
martingale $N\colon{\subseteq{}}2^{<\omega}\rightarrow[0,\infty)$
satisfying for all $\sigma\in2^{<\omega}$ 
\begin{gather*}
\sigma\notin\dom N\quad\rightarrow\quad\mu([\sigma]_{\mathcal{A}})=0\\
N(\sigma0)\mu([\sigma0]_{\mathcal{A}})+N(\sigma1)\mu([\sigma1]_{\mathcal{A}})=N(\sigma)\mu([\sigma]_{\mathcal{A}})\quad(\textit{undefined}\cdot0=0)
\end{gather*}
and a partial-computable ``savings function'' $f\colon{\subseteq{}}2^{<\omega}\rightarrow[0,\infty)$
satisfying for all $\sigma,\tau\in2^{<\omega}$ 
\begin{gather*}
\sigma\notin\dom f\quad\rightarrow\quad\mu([\sigma]_{\mathcal{A}})=0\\
f(\sigma)\leq N(\sigma)\leq f(\sigma)+1\quad(\text{for }\sigma\in\dom f\cap\dom N)\\
\sigma\preceq\tau\quad\rightarrow f(\sigma)\leq f(\tau)\quad(\text{for }\sigma,\tau\in\dom f)
\end{gather*}
such that $\lim_{n\rightarrow\infty}N(x\upharpoonright_{\mathcal{A}}n)=\infty$.
(Say $(N,f)$ wins on $x$.)
\item \label{enu:integral}(Integral test) There is a computable measure
$\nu\colon2^{<\omega}\rightarrow[0,\infty)$ and a lower semicomputable
function $g\colon\mathcal{X}\rightarrow[0,\infty]$ satisfying for
all $\sigma\in2^{<\omega}$ 
\[
\int_{[\sigma]_{\mathcal{A}}}g\,d\mu\leq\nu(\sigma)
\]
such that $g(x)=\infty$. (Say $(g,\nu)$ wins on $x$.)
\item \label{enu:bounded}(Bounded Martin-Löf test) There is a computable
measure $\nu\colon2^{<\omega}\rightarrow[0,\infty)$ and a Martin-Löf
test $(U_{n})$ satisfying for all $n\in\mathbb{N}$ and $\sigma\in2^{<\omega}$
\[
\mu(U_{n}\cap[\sigma]_{\mathcal{A}})\leq2^{-n}\nu(\sigma).
\]
such that $(U_{n})$ covers $x$. (Say $((U_{n}),\nu)$ wins on $x$.)
\item \label{enu:vitali_cover}(Solovay test) There is a computable measure
$\nu\colon2^{<\omega}\rightarrow[0,\infty)$ and a Solovay test $(V_{n})$
satisfying for all $n\in\mathbb{N}$ and $\sigma\in2^{<\omega}$ 
\[
\sum_{n}\mu(V_{n}\cap[\sigma]_{\mathcal{A}})\leq\nu(\sigma)
\]
such that $(V_{n})$ Solovay covers $x$. (Say $((V_{n}),\nu)$ wins
on $x$.)
\end{enumerate}

For \emph{(\ref{enu:integral})} through \emph{(\ref{enu:vitali_cover})},
the measure $\nu$ may be assumed to be a probability measure and
satisfy the following absolute-continuity condition,
\begin{equation}
\nu(\sigma)\leq\int_{[\sigma]_{\mathcal{A}}}h\,d\mu\label{eq:abs-cont}
\end{equation}
for some integrable function $h$.

Further, the equivalences are computable. In other words, given two
types of tests, a test $T$ of the first type is computable from a
test $S$ of the second type such that $T$ wins on all points $x\in\mathcal{X}$
that $S$ wins on---assuming $x$ is neither an unrepresented point
nor in a null cell.

\end{thm}
\begin{defn}
\label{def:comp_rand}Let $\mathcal{A}$ be a cell decomposition of
the computable probability space $(\mathcal{X},\mu)$. Say $x\in X$
is \noun{computably random} (with respect to $\mathcal{A}$) if $x$
is neither an unrepresented point nor in a null cell, and $x$ does
not satisfy any of the equivalent conditions (\ref{enu:mart}) through
(\ref{enu:vitali_cover}) of Theorem~\ref{thm:comp_rand_defs}.
\end{defn}
Before proving the theorem, here is a technical lemma. It will be
needed to show (\ref{enu:mart}) implies (\ref{enu:mart_savings})
in the proof of Theorem~\emph{\ref{thm:comp_rand_defs}}.
\begin{lem}[Technical lemma]
Let $(a_{n})$ be a sequence of positive real numbers. Define $(b_{n})$
and $(c_{n})$ recursively as follows: $b_{0}=a_{0}$, $c_{0}=b_{0}-1$,
\[
b_{n+1}=c_{n}+\frac{a_{n+1}}{a_{n}}(b_{n}-c_{n})
\]
and $c_{n+1}=\max(c_{n},b_{n+1}-1)$. If $\limsup_{n}a_{n}=\infty$,
then $\lim_{n}b_{n}=\infty$.\end{lem}
\begin{proof}
Let $(n_{i})$ be the indices such that $c_{n_{i}}=b_{n_{i}}-1$ listed
in order. By induction on $n\in[n_{i},n_{i+1})$ we have $c_{n}=b_{n_{i}}-1$
and 
\[
b_{n+1}=b_{n_{i}}+\frac{a_{n+1}}{a_{n_{i}}}-1.
\]
Since $\limsup_{n}a_{n}=\infty$, there exists some $m>n_{i}$ such
$b_{m}-1\geq b_{n_{i}}-1=c_{n_{i}}$. The first such $m$ is $n_{i+1}$.
This is also the first $m$ such that $a_{m}\geq a_{n_{i}}$. Therefore
$(n_{i})$ is an infinite series, $a_{n_{i+1}}\geq a_{n_{i}}$, $\lim_{i}a_{n_{i}}=\infty$,
and 
\[
c_{n_{i+1}}=b_{n_{i+1}}-1=\left(b_{n_{i}}+\frac{a_{n_{i+1}}}{a_{n_{i}}}-1\right)-1=c_{n_{i}}+\left(\frac{a_{n_{i+1}}}{a_{n_{i}}}-1\right).
\]
We have that $c_{n_{i}}\geq\ln(a_{n_{i}})$ (natural logarithm) by
the identity $x-1\geq\ln(x)$ and by induction:
\begin{align*}
c_{n_{i+1}} & =c_{n_{i}}+\left(\frac{a_{n_{i+1}}}{a_{n_{i}}}-1\right)\\
 & \geq\ln(a_{n_{i}})+\ln\left(\frac{a_{n_{i+1}}}{a_{n_{i}}}\right)=\ln(a_{n_{i+1}}).
\end{align*}
Hence $\lim_{i}c_{n_{i}}\geq\lim_{i}\ln(a_{n_{i}})=\infty$. Since
$c_{n}$ is nondecreasing and $b_{n}\geq c_{n}$, we have $\lim_{n}b_{n}\geq\lim_{n}c_{n}=\infty$.
\end{proof}

\begin{proof}[Proof of Theorem~\emph{\ref{thm:comp_rand_defs}}]
(\ref{enu:mart}) implies (\ref{enu:mart_savings}): The idea is
to bet with the martingale $M$ as usual, except at each stage set
some of the winnings aside into a savings account $f(\sigma)$ and
bet only with the remaining capital. Formally, define $N$ and $f$
recursively as follows. (One may assume $M(\sigma)\ge1$ for all $\sigma$
by adding $1$ to $M(\sigma)$.) Start with $N(\varepsilon)=M(\varepsilon)$
and $f(\varepsilon)=N(\varepsilon)-1$. At $\sigma$, for $i=0,1$
let 
\[
N(\sigma i)=f(\sigma)+\frac{M(\sigma i)}{M(\sigma)}(N(\sigma)-f(\sigma))
\]
 and $f(\sigma i)=\max(f(\sigma),N(\sigma i)-1)$. (Do not define
$N(\sigma)$ or $f(\sigma)$ unless $\mu(\sigma)>0$, in which case
$M(\tau)$ must be defined for all $\tau\preceq\sigma$.) By the technical
lemma above, $\lim_{n}N(x\upharpoonright_{\mathcal{A}}n)=\infty$.

(\ref{enu:mart_savings}) implies (\ref{enu:integral}): Let $\nu(\sigma)=N(\sigma)\mu([\sigma]_{\mathcal{A}})$
and $g(y)=\sup\{f(y\upharpoonright_{\mathcal{A}}s)\mid\mu(y\upharpoonright_{\mathcal{A}}s)>0\}$
(where $\sup\varnothing=0$). Then $\int_{[\sigma]_{\mathcal{A}}}g\,d\mu\leq\nu(\sigma)\leq\int_{[\sigma]_{\mathcal{A}}}(g+1)\,d\mu$,
which also shows $\nu$ satisfies the absolute-continuity condition
of formula~(\ref{eq:abs-cont}). If $N(\varepsilon)$ is scaled to
be $1$, then $\nu$ is a probability measure.

(\ref{enu:integral}) implies (\ref{enu:mart}): Let $M(\sigma)=\nu(\sigma)/\mu([\sigma]_{\mathcal{A}})$.
Since $g$ is lower semicontinuous and $g(x)=y$, there is a decreasing
sequence $(r_{n})$ such that $g(y)\geq n$ for all $y\in B(x,r_{n})$
and $n$. Let $(k_{n})$ be a decreasing sequence such that $[x\upharpoonright_{\mathcal{A}}k_{n}]_{\mathcal{A}}\subseteq B(x,r_{n})$
for all $n$. Then $M(x\upharpoonright_{\mathcal{A}}k_{n})\geq\frac{\int_{[x\upharpoonright_{\mathcal{A}}k_{n}]_{\mathcal{A}}}g\,d\mu}{\mu([x\upharpoonright_{\mathcal{A}}k_{n}]_{\mathcal{A}})}\geq n$.
Therefore, $\limsup_{k\rightarrow\infty}M(x\upharpoonright_{\mathcal{A}}k)=\infty$.

(\ref{enu:integral}) implies (\ref{enu:bounded}): Let $U_{n}=\{x\mid g(x)>2^{n}\}$.
By Markov's inequality, $\mu(U_{n}\cap[\sigma]_{\mathcal{A}})\cdot2^{n}\leq\int_{[\sigma]_{\mathcal{A}}}g\,d\mu\leq\nu(\sigma)$.

(\ref{enu:bounded}) implies (\ref{enu:vitali_cover}): Let $V_{n}=U_{n}$.

(\ref{enu:vitali_cover}) implies (\ref{enu:integral}): Let $g=\sum_{n}\mathbf{1}_{V_{n}}$.
\end{proof}
In this next proposition, I show the standard randomness implications
(as in formula~(\ref{eq:ran-hierarchy})) still hold. 
\begin{prop}
\label{prop:computable-implies-schnorr}Let $(\mathcal{X},\mu)$ be
a computable probability space. If $x\in\mathcal{X}$ is Martin-Löf
random, then $x$ is computably random (with respect to every cell
decomposition $\mathcal{A}$). If $x\in\mathcal{X}$ is computably
random (with respect to a cell decomposition $\mathcal{A}$), then
$x$ is Schnorr random, and hence Kurtz random.\end{prop}
\begin{proof}
The statement on Martin-Löf randomness follows from the bounded Martin-Löf
test (Theorem~\ref{thm:comp_rand_defs}\,(\ref{enu:bounded})).

For the last statement, assume $x$ is not Schnorr random. If $x$
is an unrepresented point or in a null cell, then $x$ is not computably
random by Definition~\ref{def:comp_rand}. Else, there is some Solovay
test $(V_{n})$ where $\sum_{n}\mu(V_{n})$ is computable and $(V_{n})$
Solovay covers $x$. Define $\nu\colon2^{<\omega}\rightarrow[0,\infty)$
as $\nu(\sigma)=\sum_{n}\mu(V_{n}\cap[\sigma]_{\mathcal{A}})$. Then
clearly, $\mu(V_{n}\cap[\sigma]_{\mathcal{A}})\leq\nu(\sigma)$ for
all $n$ and $\sigma$. By the Solovay test (Theorem~\ref{thm:comp_rand_defs}~(\ref{enu:vitali_cover})),
it is enough to show that $\nu$ is a computable measure. It is straightforward
to verify that $\nu(\sigma0)+\nu(\sigma1)=\nu(\sigma)$. As for the
computability of $\nu$; notice $\nu(\sigma)$ is lower semicomputable
uniformly from $\sigma$, since $\mu$ is a computable probability
measure (see Definition~\ref{def:comp_prob_meas}). Then since $\nu(\varepsilon)=\sum_{n}\mu(V_{n})$
is computable, $\nu$ is a computable measure.\end{proof}
\begin{thm}
\label{thm:indepentant-of-representation}The definition of computable
randomness does not depend on the choice of cell decomposition.\end{thm}
\begin{proof}
Before giving the details, here is the main idea. It suffices to convert
a test with respect to one cell decomposition $\mathcal{A}$ to another
test which covers the same points, but is with respect to a different
cell decomposition $\mathcal{B}$. In order to do this, take a bounding
measure $\nu$ with respect to $\mathcal{A}$ (which is really a measure
on $2^{\omega}$) and transfer it to an actual measure $\pi$ on $\mathcal{X}$.
Then transfer $\pi$ back to a bounding measure $\kappa$ with respect
to $\mathcal{B}$. In order to guarantee that this will work, we will
assume $\nu$ satisfies the absolute-continuity condition~\ref{eq:abs-cont},
which ensures that $\pi$ exists and is absolutely continuous with
respect to $\mu$.

Now I give the details. Assume $x\in\mathcal{X}$ is not computably
random with respect to the cell decomposition $\mathcal{A}$ of the
space. Let $\mathcal{B}$ be another cell decomposition. If $x$ is
an unrepresented point or in a null cell, then $x$ is not a Kurtz
random with respect to $(\mathcal{X},\mu)$, and by Proposition~\ref{prop:computable-implies-schnorr},
$x$ is not computably random with respect to $\mathcal{B}$.

So assume $x$ is neither an unrepresented point nor in a null cell.
By condition (\ref{enu:bounded}) of Theorem~\ref{thm:comp_rand_defs}
there is some Martin-Löf test $(U_{n})$ bounded by a probability
measure $\nu$ such that $(U_{n})$ covers $x$. Further, $\nu$ can
be assumed to satisfy the absolute-continuity condition (\ref{eq:abs-cont}).
\begin{claim*}
If $\nu$ is a probability measure satisfying the absolute-continuity
condition (\ref{eq:abs-cont}) with respect to $\mu$, then $\pi([\sigma]_{\mathcal{A}})=\nu(\sigma)$
defines a probability measure $\pi$ absolutely continuous with respect
to $\mu$, i.e.\ every $\mu$-null set is a $\pi$-null set.\end{claim*}
\begin{proof}[Proof of claim.]
I apply the Carathéodory extension theorem. A semi-ring $\mathcal{R}$
is a family of sets which contains $\varnothing$, is closed under
intersections, and for each $A,B$ in $\mathcal{R}$, there are pairwise
disjoint sets $C_{1},\ldots,C_{n}$ in $\mathcal{R}$ such that $A\smallsetminus B=C_{1}\cup\ldots\cup C_{n}$.
For example, $\{\varnothing\}\cup\{[\sigma]^{\prec}\}_{\sigma\in2^{<\omega}}$
is a semi-ring on $2^{\omega}$. Similarly, the collection $\{\varnothing\}\cup\{[\sigma]_{\mathcal{A}}\}_{\sigma\in2^{<\omega}}$
is ``$\mu$-almost-everywhere'' a semi-ring on $\mathcal{X}$, in
that it contains $\varnothing$, is closed under intersections \emph{up to $\mu$-a.e.\ equivalence},
and for each $A,B$ in $\mathcal{R}$, there are pairwise disjoint
sets $C_{1},\ldots,C_{n}$ in $\mathcal{R}$ such that $A\smallsetminus B=C_{1}\cup\ldots\cup C_{n}$
\emph{$\mu$-a.e.} It is extended to a semi-ring by adding every $\mu$-null
set and every set which is $\mu$-a.e.\ equal to $[\sigma]_{\mathcal{A}}$
for some $\sigma$. (If $A=[\sigma]_{\mathcal{A}}$, then set $\pi(A)=\nu(\sigma)$.)
Denote this semi-ring as $\mathcal{R}$.

For $A\in\mathcal{R}$, define
\[
\pi(A)=\begin{cases}
\nu(\sigma) & \text{if }A=[\sigma]_{\mathcal{A}}\ \mu\text{-a.e.}\\
0 & \text{if }A=\varnothing\ \mu\text{-a.e.}
\end{cases}.
\]
(This is well-defined since if $A=[\sigma]_{\mathcal{A}}=[\tau]_{\mathcal{A}}$
$\mu$-a.e., then the symmetric difference $[\sigma]^{\prec}\triangle[\tau]^{\prec}$
is equal to a finite disjoint union of basic open sets $\bigcup_{i=0}^{k-1}[\rho_{i}]^{\prec}$.
By the absolute continuity condition (\ref{eq:abs-cont}), 
\[
\nu([\sigma]^{\prec}\triangle[\tau]^{\prec})=\sum_{i=0}^{k-1}\nu(\rho_{i})\leq\sum_{i=0}^{k-1}\int_{[\rho_{i}]_{\mathcal{A}}}\!h\,d\mu=\int_{[\sigma]_{\mathcal{A}}\triangle[\tau]_{\mathcal{A}}}\!h\,d\mu=0.
\]
Similarly if $A=[\sigma]_{\mathcal{A}}=\varnothing$ $\mu$-a.e.,
then $\nu(\sigma)=0$.) Now, it is enough to show $\pi$ is a pre-measure,
specifically that it satisfies countable additivity. Assume for some
pairwise disjoint family $\{A_{i}\}$ and some $B$, both in the semi-ring
$\mathcal{R}$, that $B=\bigcup_{i}A_{i}$. If $B$ is $\mu$-null,
then each $A_{i}$ is as well. By the definition of $\pi$ on $\mu$-null
sets, we have $\pi(B)=0=\sum_{i}\pi(A_{i})$. If $B$ is not $\mu$-null,
then $B=[\tau]_{\mathcal{A}}$ $\mu$-a.e.\ for some $\tau$ and
each $A_{i}$ of positive $\mu$-measure is $\mu$-a.e. equal to $[\sigma_{i}]_{\mathcal{A}}$
for some $\sigma_{i}\succeq\tau$. For each $k$, let $C_{k}=[\tau]^{\prec}\smallsetminus\bigcup_{i=0}^{k-1}[\sigma_{i}]^{\prec}$,
which is a finite union of basic open sets in $2^{\omega}$. Let $D_{k}$
be the same union as $C_{k}$ but replacing each $[\sigma]^{\prec}$
with $[\sigma]_{\mathcal{A}}$. Then by the absolute continuity condition,
\[
\pi(B)-\sum_{i=0}^{k-1}\pi(A_{i})=\nu(\tau)-\sum_{i=0}^{k-1}\nu(\sigma_{i})=\nu(C_{k})\leq\int_{D_{k}}h\,d\mu
\]
Since $[\tau]_{\mathcal{A}}=\bigcup_{i}[\sigma_{i}]_{\mathcal{A}}$
$\mu$-a.e., the right-hand-side goes to zero as $k\rightarrow\infty$.
So $\pi$ is a pre-measure and may be extended to a measure by the
Carathéodory extension theorem. 

Similarly by approximation, $\pi$ satisfies $\pi(A)\leq\int_{A}h\,d\mu$
for all Borel sets $A$ and hence is absolutely continuous with respect
to $\mu$. 

To see that $\pi$ is a computable probability measure on $\mathcal{X}$,
take a $\Sigma_{1}^{0}$ set $U$. By Proposition~\ref{prop:decompose-open-set},
there is a c.e.,\ prefix-free set $\{\sigma_{i}\}$ (c.e.\ in the
code for $U$) of finite strings such that $U=\bigcup_{i}[\sigma_{i}]_{\mathcal{A}}$
$\mu$-a.e.\ (and so $\pi$-a.e.\ by absolute continuity). As this
union is disjoint, $\pi(U)=\sum_{i}\pi([\sigma_{i}]_{\mathcal{A}})=\sum_{i}\nu(\sigma_{i})$
$\mu$-a.e. and so $\pi(U)$ is lower semicomputable uniformly in
$U$. Since $\pi(\mathcal{X})=1$, $\pi$ is a computable probability
measure. This proves the claim.
\end{proof}
Let $\pi$ be as in the claim. Since $\pi$ is absolutely continuous
with respect to $\mu$, any a.e.~decidable set of $\mu$ is an a.e.~decidable
set of $\pi$. In particular, the values $\pi([\tau]_{\mathcal{B}})$
are computable uniformly from $\tau$. Now transfer $\pi$ back to
a measure $\kappa\colon2^{<\omega}\rightarrow[0,\infty)$ using $\kappa(\sigma)=\pi([\sigma]_{\mathcal{B}})$.
This is a computable probability measure.

Last, we show the Martin-Löf test $(U_{n})$ is bounded by $\kappa$
with respect to the cell decomposition $\mathcal{B}$. To see this,
fix $\tau\in2^{<\omega}$ and take the c.e.,\ prefix-free set $\{\sigma_{i}\}$
of finite strings such that $[\tau]_{\mathcal{B}}=\bigcup_{i}[\sigma_{i}]_{\mathcal{A}}$
$\mu$-a.e.~(and so $\pi$-a.e.). Then $\kappa(\tau)=\sum_{i}\nu(\sigma_{i})$,
and for each $n$, 
\[
\mu(U_{n}\cap[\tau]_{\mathcal{B}})=\sum_{i}\mu(U_{n}\cap[\sigma_{j}]_{\mathcal{A}})\leq\sum_{i}2^{-n}\nu(\sigma_{i})=2^{-n}\kappa(\tau).\qedhere
\]

\end{proof}
Theorem~\ref{thm:comp_rand_defs} is just a sample of the many equivalent
definitions for computable randomness. I conjecture that the other
known characterizations of computable randomness, see for example
Downey and Hirschfelt \cite[Section\ 7.1]{Downey2010}, can be extended
to arbitrary computable Polish spaces using the techniques above.
As well, other test characterizations for Martin-Löf randomness can
be extended to computable randomness by ``bounding the test'' with
a computable measure or martingale. (See Section~\ref{sec:Kolmogorov-complexity-and-randomness}
for an example using machines.) Further, the proof of Theorem~\ref{thm:indepentant-of-representation}
shows that the bounding measure $\nu$ can be assumed to be a measure
on $\mathcal{X}$, instead of $2^{\omega}$, under the additional
condition that $\mathcal{A}$ is a cell decomposition for both $(\mathcal{X},\mu)$
and $(\mathcal{X},\nu)$. Similarly, we could modify the martingale
test to assume $M$ is a martingale on $(\mathcal{X},\mu)$ (in the
sense of probability theory) with an appropriate filtration.

Actually, the above ideas can be used to show any $L^{1}$-bounded
a.e.~computable martingale (in the sense of probability theory) converges
on computable randoms if the filtration converges to the Borel sigma-algebra
(or even a ``computable'' sigma-algebra) and the $L^{1}$-bound
is computable. This can be extended to (the Schnorr layerwise-computable
representatives of) $L^{1}$-computable martingales as well. The proof
is beyond the scope of this paper and will be published separately. 

However, I will use the above ideas to give an integral test characterization
of computable randomness which avoids cell decompositions or any other
representation of the measure $\mu$.
\begin{thm}
\label{thm:integral-test-CR}A point $x$ is computably random with
respect to $(\mathcal{X},\mu)$ if and only if $g(x)<\infty$ for
all integral tests $g$ (as in Definition~\ref{def:solovay-integral})
bounded by a computable measure $\pi$, that is $\int_{A}\!g\,d\mu\leq\pi(A)$
for all measurable sets $A$.\end{thm}
\begin{proof}
($\Rightarrow$) Assume $x$ is not computably random with respect
to some cell decomposition $\mathcal{A}$. If $x$ is an unrepresented
point of $\mathcal{A}$ or is in a null open set, then $x$ is not
Schnorr random. By Theorem~\ref{thm:schnorr-tests}, there is an
integral test $g\colon\mathcal{X}\rightarrow[0,\infty]$ such that
$\int\!g\,d\mu=1$ and $g(x)=\infty$. It is enough to show the probability
measure $\pi$ given by $\pi(A)=\int_{A}g\,d\mu$ is computable. For
all $\Sigma_{1}^{0}$ sets $U$, $\pi(U)=\int_{U}\!g\,d\mu$ is lower
semicomputable uniformly in $U$ since $g\cdot\mathbf{1}_{U}$ is
lower semicomputable uniformly in $U$ (Hoyrup and Rojas \cite[Proposition~4.3.1]{Hoyrup2009}).
Therefore, $\pi$ is computable.

If $x$ is neither an unrepresented point of $\mathcal{A}$ nor is
in a null open set, then by Theorem~\ref{thm:comp_rand_defs}, there
is an integral test $g\colon\mathcal{X}\rightarrow[0,\infty]$ and
an computable probability measure $\nu$ on $2^{\omega}$ such that
$\int_{[\sigma]_{\mathcal{A}}}g\,d\mu\leq\nu(\sigma)\leq\int_{[\sigma]_{\mathcal{A}}}h\,d\mu$
for some $\mu$-integrable function $h$. By the claim in the proof
of Theorem~\ref{thm:indepentant-of-representation}, $\pi([\sigma]_{\mathcal{A}})=\nu(\sigma)$
defines a computable measure on $\mathcal{X}$ absolutely continuous
with respect to $\mu$. Let $U\subseteq\mathcal{X}$ be $\Sigma_{1}^{0}$.
By Proposition~\ref{prop:decompose-open-set} there is a prefix-free
set of strings $I\subseteq2^{<\omega}$ such that $U=\bigcup_{\sigma\in I}[\sigma]_{\mathcal{A}}$\ $\mu$-a.e.
Therefore 
\[
\int_{U}\!g\,d\mu=\sum_{\sigma\in I}\int_{[\sigma]_{\mathcal{A}}}\!g\,d\mu\leq\sum_{\sigma\in I}\nu(\sigma)=\sum_{\sigma\in I}\pi([\sigma]_{\mathcal{A}})=\pi(U).
\]
By approximation, this inequality extends to all measurable sets $A$
in place of $U$. Therefore, $\pi$ bounds $g$.

($\Rightarrow$) Assume $g(x)=\infty$ for some integral test bounded
by a computable measure $\pi$. We may assume $\pi$ is a probability
measure. Consider the computable probability measure $\rho=(\mu+\pi)/2$,
and let $\mathcal{A}$ be a cell decomposition for $\rho$. Then $\mathcal{A}$
is also a cell decomposition for $\mu$ and $\pi$. Therefore, $\nu(\sigma)=\pi([\sigma]_{\mathcal{A}})$
defines a computable measure on $2^{\omega}$, such that $\int_{[\sigma]_{\mathcal{A}}}\!g(x)\,d\mu\leq\nu(\sigma)$.
By Theorem~\ref{thm:comp_rand_defs}\,(\ref{enu:integral}), $x$
is not computably random.
\end{proof}
In Section~\ref{sec:Further-directions}, I give ideas on how computable
randomness can be defined on an even broader class of spaces, and
also on non-computable probability spaces. I end this section by showing
that Definition~\ref{def:comp_rand} is consistent with the usual
definitions of computable randomness with respect to $2^{\omega}$,
$\Sigma^{\omega}$, and $[0,1]$.
\begin{example}
Consider a computable probability measure $\mu$ on $2^{\omega}$.
It is easy to see that computable randomness in the sense of Definition~\ref{def:comp_rand}
with respect to the natural cell decomposition is equivalent to computable
randomness with respect to $2^{\omega}$ as defined in Definition~\ref{def:comp-random-cantor}.
Since Definition~\ref{def:comp_rand} is invariant under the choice
of cell decomposition (Theorem~\ref{thm:indepentant-of-representation}),
the two definitions agree on $(2^{\omega},\mu)$.
\end{example}

\begin{example}
\label{ex:comp-rand-on-3-omega-1}Consider a computable probability
measure $\mu$ on $\Sigma^{\omega}$ where $\Sigma=\{a_{0},\ldots,a_{k-1}\}$
is a finite alphabet. It is natural to define a martingale $M\colon\Sigma^{<\omega}\rightarrow[0,\infty)$
as one satisfying the fairness condition 
\[
M(\sigma a_{0})\mu(\sigma a_{0})+\cdots+M(\sigma a_{k-1})\mu(\sigma a_{k-1})=M(\sigma)\mu(\sigma)
\]
for all $\sigma\in\Sigma^{<\omega}$ (along with the impossibility
condition from Definition~\ref{def:Martingales-Cantor}). Since there
was nothing special so far about a two-symbol alphabet, Theorem~\ref{thm:comp_rand_defs}
can be easily adapted to a $k$-symbol alphabet giving the following
result: An a.e.\ computable martingale of type $M\colon{\subseteq{}}\Sigma^{<\omega}\rightarrow[0,\infty)$
succeeds on a sequence $x\in\Sigma^{\omega}$ if and only $g(x)=\infty$
for some lower semicomputable function $g\colon\Sigma^{\omega}\rightarrow[0,\infty]$
bounded by a computable measure $\nu$ on $\Sigma^{\omega}$. Applying
Theorem~\ref{thm:integral-test-CR}, we have that a sequence $x\in\Sigma^{\omega}$
is computably random (as in Definition~\ref{def:comp_rand}) with
respect to $(\Sigma^{\omega},\mu)$ if and only if no a.e.\ computable
martingale of type $M\colon{\subseteq{}}\Sigma^{<\omega}\rightarrow[0,\infty)$
succeeds on $x$.
\end{example}

\begin{example}
\label{ex:Base-inv-1}Let $([0,1],\lambda)$ be the space $[0,1]$
with the Lebesgue measure. Let 
\[
A_{i}=\{x\in[0,1]\mid\text{the }i\text{th binary digit of }x\text{ is }1\}.
\]
Then $\mathcal{A}=(A_{i})$ is a generator of $([0,1],\lambda)$ and
$[\sigma]_{\mathcal{A}}=[0.\sigma,0.\sigma+2^{-|\sigma|})$ a.e. A
little thought reveals that $x\in([0,1],\lambda)$ is computably random
(in the sense of Definition~\ref{def:comp_rand}) if and only if
the binary expansion of $x$ is computably random with respect to
$(2^{\omega},\lambda)$. This is the standard definition of computable
randomness on $([0,1],\lambda)$. Further, using a base $b$ other
than binary gives a different generator, for example let 
\[
A_{bi+j}=\{x\in[0,1]\mid\text{the }i\text{th }b\text{-ary digit of }x\text{ is }j\}
\]
where $0\leq j<b$. Yet, the computably random points remain the same.
Hence computable randomness on $([0,1],\lambda)$ is base invariant
\cite{Brattka2011,Silveira:2011fk} . (The proof of Theorem~\ref{thm:indepentant-of-representation}
has similarities to the proof of Brattka, Miller and Nies \cite{Brattka2011},
but as mentioned in the introduction, there are also key differences.)
Also see Example~\ref{ex:Base-inv-2}.
\end{example}
More examples are given at the end of Section~\ref{sec:CR-and-isomorphisms}.

\section{Machine characterizations of computable and Schnorr randomness\label{sec:Kolmogorov-complexity-and-randomness}}

In this section I give machine characterizations of computable and
Schnorr randomness with respect to computable probability spaces.
This has already been done for Martin-Löf randomness.

Recall the following definition and fact. 
\begin{defn}
A machine $M$ is a partial-computable function $M\colon{\subseteq{}}2^{<\omega}\rightarrow2^{<\omega}$.
A machine is \noun{prefix-free} if $\dom M$ is prefix-free. For
a prefix-free machine $M$, let the \noun{prefix-free Kolmogorov complexity of $\sigma$ relative to $M$}
be
\[
K_{M}(\sigma)=\inf\left\{ |\tau|\;\middle|\;\tau\in2^{<\omega}\text{ and }M(\tau)=\sigma\right\} .
\]
(There is a non-prefix-free version of complexity as well.)\end{defn}
\begin{thm}[{Schnorr (see \cite[Theorem~6.2.3]{Downey2010})}]
A sequence $x\in(2^{\omega},\lambda)$ is Martin-Löf random if and
only if for all prefix-free machines $M$, 
\begin{equation}
\limsup_{n\rightarrow\infty}\left(n-K_{M}(x\upharpoonright n)\right)<\infty.\label{eq:Kolmogorov}
\end{equation}

\end{thm}
Schnorr's theorem has been extended to both Schnorr and computable
randomness. 
\begin{defn}
For a machine $M$ define the semimeasure $\meas_{M}\colon2^{<\omega}\rightarrow[0,\infty)$
as 
\[
\meas_{M}(\sigma)=\sum_{\begin{subarray}{c}
\tau\in\dom M\\
M(\tau)\succeq\sigma
\end{subarray}}2^{-|\tau|}.
\]
A machine $M$ is a \noun{computable-measure machine} if $\meas_{M}(\varepsilon)$
is computable. A machine $M$ is a \noun{bounded machine} if there
is some computable-measure $\nu$ such that $\meas_{M}(\sigma)\leq\nu(\sigma)$
for all $\sigma\in2^{<\omega}$. 

Downey, Griffiths, and LaForte \cite{Downey2004a} showed that $x\in(2^{\omega},\lambda)$
is Schnorr random precisely if formula~(\ref{eq:Kolmogorov}) holds
for all prefix-free, computable-measure machines. Mihailovi\'{c} (see
\cite[Thereom 7.1.25]{Downey2010}) showed that $x\in(2^{\omega},\lambda)$
is computably random precisely if formula~(\ref{eq:Kolmogorov})
holds for all prefix-free, bounded machines.
\end{defn}
Schnorr's theorem was extended to all computable probability measures
on Cantor space by Gács \strong{\cite{Gacs:1980fj}}. Namely, replace
formula~(\ref{eq:Kolmogorov}) with
\[
\limsup_{n\rightarrow\infty}\left(-\log_{2}\mu([x\upharpoonright n]^{\prec})-K_{M}(x\upharpoonright n)\right)<\infty.
\]
If $\mu([x\upharpoonright n])=0$ for any $n$ then we say this inequality
is false. Hoyrup and Rojas \cite{Hoyrup2009} extended this to any
computable probability space. Here, I do the same for Schnorr and
computable randomness. (I include Martin-Löf randomness for completeness.)
\begin{thm}
Let $(\mathcal{X},\mu)$ be a computable probability space with a.e.\ decidable cell decomposition $\mathcal{A}$ and let
$x\in\mathcal{X}$.
\begin{enumerate}
\item $x\in\mathcal{X}$ is Martin-Löf random precisely if
\begin{equation}
\limsup_{n\rightarrow\infty}\left(-\log_{2}\mu([x\upharpoonright_{\mathcal{A}}n]_{\mathcal{A}})-K_{M}(x\upharpoonright_{\mathcal{A}}n)\right)<\infty.\label{eq:Kol-comp-meas}
\end{equation}
holds for all prefix-free machines $M$. (Again, we say formula~\emph{(\ref{eq:Kol-comp-meas})}
is false if $\mu([x\upharpoonright n])=0$ for any $n$.)
\item $x\in\mathcal{X}$ is computably random precisely if formula~\emph{(\ref{eq:Kol-comp-meas})}
holds for all prefix-free, bounded machines $M$.
\item $x\in\mathcal{X}$ is Schnorr random precisely if formula~\emph{(\ref{eq:Kol-comp-meas})}
holds for all prefix-free, computable measure machines $M$.
\end{enumerate}

Further, \emph{(1)} through \emph{(3)} hold even if $M$ is not
assumed to be prefix-free, but only that $\meas_{M}(\varepsilon)\leq1$.

\end{thm}
\begin{proof}
Slightly modify the proofs of Theorems~6.2.3, 7.1.25, and 7.1.15
in Downey and Hirschfelt \cite{Downey2010}, respectively.
\end{proof}

\section{Computable randomness and isomorphisms\label{sec:CR-and-isomorphisms}}

In this section I give another piece of evidence that the definition
of computable randomness in this paper is robust, namely that the
computably random points are preserved under isomorphisms between
computable probability spaces. I also show a one-to-one correspondence
between cell decompositions of a computable measure space and isomorphisms
from that space to the Cantor space. However, first I need to define
partial computable functions and present some properties.

\subsection{Partial computable functions on computable Polish spaces}
\begin{defn}
\label{def:partial-comp-metric}Let $\mathcal{X}=(X,d_{\mathcal{X}},A)$
and $\mathcal{Y}=(Y,d_{\mathcal{Y}},B)$ be computable Polish spaces.
A partial function $f\colon{\subseteq{}}\mathcal{X}\rightarrow\mathcal{Y}$
is \noun{partial computable} if there is a partial computable function
$g\colon{\subseteq{}}A^{\omega}\rightarrow B^{\omega}$ (as in Definition~\ref{def:part-fun-seq-to-seq})
such that for all $a\in A^{\omega}$, the following hold.
\begin{enumerate}
\item $a\in\dom g$ if and only if $a$ is a Cauchy name for some $x\in\dom f$.
\item If $x\in\dom{f}$ and $a$ is a Cauchy name for $x$, then $g(a)$
is a Cauchy name for $f(x)$.
\end{enumerate}
\end{defn}
This next theorem gives a convenient characterization of partial computable
functions. It may be previously known, but I was not able to find
a satisfactory reference.\footnote{Adrian Maler communicated to me that a similar result is in an unpublished
paper of his.} The closest I could find was a result by Hemmerling \cite{Hemmerling:2002aa}
that the domain of a partial computable function is $\Pi_{2}^{0}$,
but his result assumes that the computable Polish space $\mathcal{X}$
``admits a generalized finite stratification''. This assumption
is not needed. (Also compare with Hoyrup and Rojas \cite[Theorem~3.3.1]{Hoyrup2009}.)
\begin{prop}
\label{prop:part-comp-char}Let $\mathcal{X}$ and $\mathcal{Y}$
be computable Polish spaces. Let $(V_{i})_{i\in\mathbb{N}}$ be an
enumeration of the $\Sigma_{1}^{0}$ subsets of $\mathcal{Y}$ in
a standard way.\footnote{For concreteness, let $(W_{e})_{e\in\mathbb{N}}$ be the standard
listing of c.e.~subsets of $\mathbb{N}$. If $\mathcal{Y}=(Y,d_{\mathcal{Y}},B)$,
use a computable pairing function to get a listing $(I_{e})_{e\in\mathbb{N}}$
of all c.e.~subsets of $B\times\mathbb{\mathbb{Q}}_{>0}$. Then set
$V_{e}=\bigcup_{(b,r)\in I_{e}}B(b,r)$.} Consider a partial function $f\colon D\subseteq\mathcal{X}\rightarrow\mathcal{Y}$.
Then $f$ is partial computable if and only if the following two conditions
hold.
\begin{enumerate}
\item The domain $D$ of $f$ is $\Pi_{2}^{0}$.
\item There is a sequence $(U_{i})_{i\in\mathbb{N}}$ such that $U_{i}\subseteq\mathcal{X}$
is $\Sigma_{1}^{0}$ in $i$ and for all $x\in D$ 
\[
x\in U_{i}\quad\leftrightarrow\quad f(x)\in V_{i}.
\]

\end{enumerate}
\end{prop}
\begin{proof}
Let $\mathcal{X}=(X,d_{\mathcal{X}},A)$ and $\mathcal{Y}=(Y,d_{\mathcal{Y}},B)$. 

For the below arguments, it helps to have a few definitions. Say that
$a\in A^{<\omega}$ is a \noun{partial $(\mathcal{X},\delta)$-Cauchy name}
if for all $k\leq n<|a|$, $d_{\mathcal{X}}(a(n),a(k))\leq2^{-k}(1+\delta)$,
and $a\in A^{<\omega}$ is a \noun{partial $(\mathcal{X},\delta^{-})$-Cauchy name}
if for all $k\leq n<|a|$, $d_{\mathcal{X}}(a(n),a(k))<2^{-k}(1+\delta)$.
(Notice the inequalities are different in each. Also, $\delta$ may
be positive, negative, or zero.) Similarly, define an \noun{$(\mathcal{X},\delta)$-Cauchy name}
and an \noun{$(\mathcal{X},\delta^{-})$-Cauchy name}. (An \noun{$(\mathcal{X},0)$-Cauchy name}
is the same as an $\mathcal{X}$-Cauchy name.) Say that $a\in A^{<\omega}$
is \noun{$(\mathcal{X},\delta)$-compatible with} $x$ if for all
$n<|a|$, $d_{\mathcal{X}}(a(n),x)\leq2^{-n}(1+\delta)$, and $a\in A^{<\omega}$
is \noun{$(\mathcal{X},\delta^{-})$-compatible with} $x$ if for
all $n<|a|$, $d_{\mathcal{X}}(a(n),x)<2^{-n}(1+\delta)$. 

$(\Rightarrow)$ Assume that $f\colon{\subseteq{}}\mathcal{X}\rightarrow\mathcal{Y}$
is partial computable as in Definition~\ref{def:partial-comp-metric}
and is given by a partial computable function $g\colon{\subseteq{}}A^{\omega}\rightarrow B^{\omega}$.
In turn, by Definition~\ref{def:part-fun-seq-to-seq}, $g$ is given
by a partial computable function $h\colon{\subseteq{}}A^{<\omega}\rightarrow B^{<\omega}$. 

(1) First we show that $\dom f$ is $\Pi_{2}^{0}$. By Definition~\ref{def:partial-comp-metric},
$x\in\dom f$ if and only if \emph{every} $\mathcal{X}$-Cauchy name
for $x$ is in $\dom g$. Also, by Definition~\ref{def:partial-comp-metric},
$x\in\dom f$ if and only if \emph{some} $\mathcal{X}$-Cauchy name
for $x$ is in $\dom g$. Because both quantifiers hold, we can replace
``$\mathcal{X}$-Cauchy name'' with ``$(\mathcal{X},0^{-})$-Cauchy
name'' in both characterizations. Then we have that $x\in\dom f$
if and only if the following condition holds.
\begin{itemize}
\item [($*$)] for every $n\in\mathbb{N}$, letting $\delta=2^{-n}$, every
partial $(\mathcal{X},-\delta)$-Cauchy name $a\in A^{<\omega}$ which
is $(\mathcal{X},-\delta)$-compatible with $x$ can be extended to
a partial $(\mathcal{X},-\delta/2^{-})$-Cauchy name $a'\in A^{<\omega}$
which is $(\mathcal{X},-\delta/2^{-})$-compatible with $x$ such
that $a'\in\dom h$ and $|h(a')|>n$.
\end{itemize}
\noindent (For if ($*$) holds, there is an $(\mathcal{X},0^{-})$-Cauchy
name for $x$ in the domain of $g$, and if ($*$) does not hold,
then there is an $(\mathcal{X},0^{-})$-Cauchy name for $x$ not in
the domain of $g$.) It remains to show that the set of $x$ satisfying
property ($*$) is $\Pi_{2}^{0}$. Towards that goal, for each $a\in A^{<\omega}$
and $n\in\mathbb{N}$ let the sets $W_{a,n}^{(0)}$, $W_{a,n}^{(1)}$
and $W_{a,n}^{(2)}$ be defined as follows, where $\delta=2^{-n}$. 
\begin{itemize}
\item $W_{a,n}^{(0)}=\mathcal{X}$ if $a$ is not a partial $(\mathcal{X},-\delta)$-Cauchy
name, else $W_{a,n}^{(0)}=\varnothing$. 
\item $W_{a,n}^{(1)}$ is the set of $x\in\mathcal{X}$ not $(\mathcal{X},-\delta)$-compatible
with $a$. 
\item $W_{a,n}^{(2)}$ is the set of $x\in\mathcal{X}$ such that there
exists $a'\in A^{<\omega}$ extending $a$ for which $a'$ is a partial
$(\mathcal{X},-\delta/2^{-})$-Cauchy name, $a'$ is $(\mathcal{X},-\delta/2^{-})$-compatible
with $x$, $a'\in\dom h$ and $|h(a')|>n$.
\end{itemize}
\noindent These sets are all $\Sigma_{1}^{0}$ in $a$ and $n$.
Moreover, $x$ satisfies ($*$) if and only if 
\[
x\in\bigcap_{n\in\mathbb{N}}\bigcap_{a\in A^{<\omega}}\left(W_{a,n}^{(0)}\cup W_{a,n}^{(1)}\cup W_{a,n}^{(2)}\right).
\]
Therefore $\dom f$ is $\Pi_{2}^{0}$.

(2) To show the second property, informally $U_{i}$ will be the set
of points $x\in\mathcal{X}$ that at some point in the algorithm for
$f$ it becomes apparent that if $x\in\dom f$ then $f(x)\in V_{i}$.

Formally, $U_{i}$ is defined as follows. Fix $i$ and set $\delta=2^{-i}$.
The set $V_{i}$ is given by a set of pairs $\{(b_{k}^{i},r_{k}^{i})\}\subseteq B\times\mathbb{Q}_{>0}$
which is c.e.\ in $i$ such that $V_{i}=\bigcup_{k}B(b_{k}^{i},r_{k}^{i})$.
Let $V''_{i}\subseteq B^{<\omega}$ be the set of all finite sequences
$b\in B^{<\omega}$ such that for some $n<|b|$ and some $k\in\mathbb{N}$,
\begin{equation}
d_{\mathcal{Y}}(b(n),b_{k}^{i})<r_{k}^{i}-2^{-n}(1+\delta).\label{eq:finite-cauchy-in-codomain}
\end{equation}
Let $V'_{i}\subseteq B^{\omega}$ be all sequences with an initial
segment in $V''_{i}$. Both $V'_{i}$ and $V''_{i}$ are $\Sigma_{1}^{0}$
in $i$.

Assume that $b\in B^{\omega}$ is a $\mathcal{Y}$-Cauchy name---and
hence a ($\mathcal{Y},\delta^{-})$-Cauchy name---for some $y\in\mathcal{Y}$.
 Then we claim that $b\in V'_{i}$ if and only if $y\in V_{i}$. To
see this, first assume $b\in V'_{i}$. Then there is some $k$ and
some $n$ such that (\ref{eq:finite-cauchy-in-codomain}) holds. Then
\[
d_{\mathcal{Y}}(y,b_{k}^{i})\leq d_{\mathcal{Y}}(y,b(n))+d_{\mathcal{Y}}(b(n),b_{k}^{i})<2^{-n}(1+\delta)+r_{k}^{i}-2^{-n}(1+\delta)=r_{k}^{i}
\]
which implies that $y\in B(b_{k}^{i},r_{k}^{i})\subseteq V_{i}$.
Conversely, assume $y\in V_{i}$. Then $y\in B(b_{k}^{i},r_{k}^{i})$
for some $k$. Let $n$ be large enough that $d_{\mathcal{Y}}(y,b_{k}^{i})<r_{k}^{i}-2\cdot2^{-n}(1+\delta)$.
Then 
\[
\delta_{\mathcal{Y}}(b(n),b_{k}^{i})\leq d_{\mathcal{Y}}(y,b(n))+d_{\mathcal{Y}}(y,b_{k}^{i})<2^{-n}(1+\delta)+r_{k}^{i}-2\cdot2^{-n}(1+\delta)=r_{k}^{i}-2^{-n}(1+\delta).
\]
Hence $b\in V_{i}'$. This proves the claim.

Let $U_{i}''\subseteq A^{<\omega}$ be the preimage $h^{-1}(V_{i}'')$
which is $\Sigma_{1}^{0}$ in $i$ since $h$ is a partial computable
function of type $h\colon{\subseteq{}}A^{<\omega}\rightarrow B^{<\omega}$.
Let $U'_{i}\subseteq A^{\omega}$ be all sequences with an initial
segment in $U''_{i}$. This is also $\Sigma_{1}^{0}$ in $i$. Finally,
let $U_{i}$ be the set of all points in $\mathcal{X}$ with an $(\mathcal{X},0^{-})$-Cauchy
name in $U'_{i}$. This is $\Sigma_{1}^{0}$ in $i$, since 
\[
U{}_{i}=\bigcup\left\{ B(a(0),2^{-0})\cap\ldots\cap B(a(n-1),2^{-(n-1)})\,\middle|\,a\in U''_{i}\ \text{and}\ n=|a|\right\} .
\]

Let $x$ be in $\dom f$. The goal now is to show that $x\in U_{i}$
if and only if $f(x)\in V_{i}$. First assume $x\in U_{i}$. Then
there is some $(\mathcal{X},0^{-})$-Cauchy name $a$ in $U'_{i}$
for $x$ and some initial segment $a\upharpoonright n$ in $U''_{i}$.
Then $h(a\upharpoonright n)\in V''_{i}$. Since $x\in\dom f$, both
$g(a)$ is a $\mathcal{Y}$-Cauchy name for $f(x)$ and $g(a)\in V'_{i}$
(since $g(a)$ extends $h(a\upharpoonright n)\in V''_{i}$). By the
above claim, $f(x)\in V_{i}$.

Conversely assume $f(x)\in V_{i}$. Let $a$ be some $(\mathcal{X},0^{-})$-Cauchy
name for $x$. (These always exist. Just remove the first element
of a Cauchy name.) Since $x\in\dom f$ then $g(a)$ is a Cauchy name
for $f(x)\in V_{i}$. By the above claim, $g(a)\in V'_{i}$. Since
$g(a)=\lim_{n}h(a\upharpoonright n)$, there is some $n$ large enough
that $h(a\upharpoonright n)\in V''_{i}$ and $a\upharpoonright n\in U''_{i}$.
This shows $a\in U'_{i}$. Since $a$ is an $(\mathcal{X},0^{-})$-Cauchy
name, $x\in U_{i}$.

$(\Leftarrow)$ Let $f\colon D\subseteq\mathcal{X}\rightarrow\mathcal{Y}$
be a partial function satisfying conditions (1) and (2). By condition
(1), $D=\bigcap_{n}W_{n}$ where $W_{n}\subseteq\mathcal{X}$ is $\Sigma_{1}^{0}$
in $n$. Let $U_{i}$ be as in condition (2). Let $(b,n)\mapsto i(b,n)$
be a computable function such that $V_{i(b,n)}=B(b,2^{-(n+1)})$.
One can compute $f$ as follows by constructing the corresponding
partial function $g\colon{\subseteq{}}A^{\omega}\rightarrow B^{\omega}$.

Given $a\in A^{\mathbb{\omega}}$, read initial segments of $a$ to
construct initial segments of $g(a)$. To find the $n$th item in
the Cauchy name $g(a)$, keep reading initial segments of $a$ until
it becomes apparent that $a$ is no longer a Cauchy name (in which
case, stop constructing $g(a)$) or it becomes apparent that for some
$k\in\mathbb{N}$ and $b\in B$, 
\[
B(a(k),2^{-k})\subseteq W_{n}\cap U_{i(b,n)}.
\]
 In this case, set the $n$th element of $g(a)$ to be $b$.\footnote{Formally, $W_{n}\cap U_{i(b,n)}$ is $\Sigma_{1}^{0}$ in $n$ and
$b$. The algorithm generates a c.e.~set of pairs $\{a_{j}^{n,b},r_{j}^{n,b}\}\subseteq A\times\mathbb{Q}_{>0}$
such that $W_{n}\cap U_{i(b,n)}=\bigcup_{j}B(a_{j}^{n,b},r_{j}^{n,b})$,
and the algorithm waits for $b$, $j$, and $k$ such that
\[
d_{\mathcal{X}}(a(k),a_{j}^{n,b})<r_{j}^{n,b}-2^{-k}.
\]
Such an event will always trigger if $x\in W_{n}\cap U_{i(b,n)}$
for some $b$.}

This algorithm will only produce an infinite sequence $g(a)$ if
$a$ is a Cauchy name for some $x\in\bigcap_{n}W_{n}=D$. Conversely,
assume $x\in D$. Then $x\in W_{n}$ for all $n\in\mathbb{N}$, and by
condition (2), for all $n\in\mathbb{N}$ there is some $b\in B$ such
that $x\in U_{i(b,n)}$. Hence $x\in W_{n}\cap U_{i(b,n)}$ for some such $b$. Therefore,
if $a$ is a Cauchy name for $x$, the algorithm will construct an
infinite sequence $g(a)$. Moreover, $g(a)$ is a Cauchy name for
$f(x)$, since if $b$ is the $n$th element of $g(a)$, then $f(x)\in V_{i(b,n)}=B(b,2^{-(n+1)})$
which implies $d_{\mathcal{X}}(b,f(x))<2^{-(n+1)}$.\end{proof}
\begin{rem}
The definition of partial computable function used in this paper (Definition~\ref{def:partial-comp-metric})
can be found in, for example, Hemmerling \cite{Hemmerling:2002aa}
and Miller \cite{Miller:2004ly}. Other authors, such as Weihrauch
\cite{Weihrauch2000} and Hoyrup and Rojas \cite{Hoyrup2009}, define
a partial computable function as a partial function $f\colon D\subseteq\mathcal{X}\rightarrow\mathcal{Y}$
such that item (2) of Proposition~\ref{prop:part-comp-char} holds.
In this case the domain $D$ may not be $\Pi_{2}^{0}$. It is important
for the results in this paper that the domain of a partial computable
function be $\Pi_{2}^{0}$. (Also in the less restrictive definition,
since any set $D$ could be a domain, there would be $2{}^{2^{\aleph_{0}}}$
many partial computable functions.)
\end{rem}

\begin{rem}
\label{rmk:Preimage}Let $f$, $(U_{i})$ and $(V_{i})$ be as in
Proposition~\ref{prop:part-comp-char}. Given a $\Sigma_{1}^{0}$
set $V=V_{i}\subseteq\mathcal{Y}$, then with a slight abuse of notation
define 
\[
f^{-1}(V)=U_{i}.
\]
Note, the exact value of $f^{-1}(V)$ depends on the codes of $V$
and $f$. Also, $f^{-1}(V)$ is not the true preimage of $U$, which
is $U_{i}\cap\dom f$. However, if $x\in\dom f$, then $x\in f^{-1}(V)$
if and only if $f(x)\in V$. Moreover, if $\dom f$ has measure one
(which it will in most of the paper), then $f^{-1}(V)$ is a.e.~equal
to the true preimage. 

Similarly, if $C\subseteq\mathcal{Y}$ is $\Pi_{1}^{0}$, define 
\[
f^{-1}(C):=\mathcal{X}\smallsetminus f^{-1}(C^{c}),
\]
and if $E\subseteq\mathcal{Y}$ is $\Sigma_{2}^{0}$, with $E=\bigcup_{i}C_{i}$
where $C_{i}$ is $\Pi_{1}^{0}$ in $i$, define 
\[
f^{-1}(E):=\bigcup_{i}f^{-1}(C_{i}).
\]
Again, $f^{-1}(E)$ is an slight abuse of notation. Its actual value
depends on the codes of $f$ and $E$. However, if $x\in\dom f$,
then $x\in f^{-1}(E)$ if and only if $f(x)\in E$. Also if $\dom f$
has measure one, then $f^{-1}(E)$ is a.e.~equal to the true preimage.
\end{rem}

\subsection{Almost-every computable maps, isomorphisms, and randomness}
\begin{defn}
Let $(\mathcal{X},\mu)$ and $(\mathcal{Y},\nu)$ be computable probability
spaces.
\begin{enumerate}
\item A partial map $T\colon(\mathcal{X},\mu)\rightarrow\mathcal{Y}$ is
said to be \noun{a.e.~computable} if it is partial computable with
a measure-one domain.
\item (Hoyrup and Rojas \cite{Hoyrup2009}) A partial map $T\colon(\mathcal{X},\mu)\rightarrow(\mathcal{Y},\nu)$
is said to be an \noun{a.e.\ computable morphism} (or \noun{morphism}
for short) if it is a.e.~computable and measure preserving, i.e.~$\mu(T^{-1}(A))=\nu(A)$
for all measurable $A\subseteq\mathcal{Y}$.
\item (Hoyrup and Rojas \cite{Hoyrup2009}) A pair of partial maps $T\colon(\mathcal{X},\mu)\rightarrow(\mathcal{Y},\nu)$
and $S\colon(\mathcal{Y},\nu)\rightarrow(\mathcal{X},\mu)$ are said
to be an \noun{a.e.\ computable isomorphism} (or \noun{isomorphism}
for short) if both maps are a.e.~computable morphisms such that $(S\circ T)(x)=x$
for $\mu$-a.e.~$x\in\mathcal{X}$ and $(T\circ S)(y)=y$ for $\nu$-a.e.~$y\in\mathcal{Y}$.
We also say $T\colon(\mathcal{X},\mu)\rightarrow(\mathcal{Y},\nu)$
is an isomorphism if such an $S$ exists.
\end{enumerate}
\end{defn}
Note that these definitions differ slightly from those of Hoyrup and
Rojas \cite{Hoyrup2009}, who implicitly require that the domain of
an almost-everywhere computable function also be dense. They also
call such functions ``almost computable.''

This next proposition says a.e.~computable maps are defined on Kurtz
randoms. Further, Kurtz randomness can be characterized by a.e.~computable
maps, and a.e.~computable maps are determined by their values on
Kurtz randoms. (For a different characterization of Kurtz randomness
using a.e.~computable functions, see Hertling and Wang \cite{Hertling:1997fk}.)
\begin{prop}
Let $(\mathcal{X},\mu)$ be a computable probability space and $\mathcal{Y}$
a computable Polish space. For $x\in\mathcal{X}$, $x$ is Kurtz random
if and only if it is in the domain of every a.e.~computable map $T\colon(\mathcal{X},\mu)\rightarrow\mathcal{Y}$.
Further, two a.e.\ computable maps are a.e.~equal if and only if
they agree on Kurtz randoms.\end{prop}
\begin{proof}
For the first part, if $x$ is Kurtz random, it avoids all null $\Sigma_{2}^{0}$
sets, and by Proposition~\ref{prop:part-comp-char} is in the domain
of every a.e.~computable map. Conversely, if $x$ is not Kurtz random,
it is in some null $\Sigma_{2}^{0}$ set $A$. Fix a computable $y_{0}\in\mathcal{Y}$ and let $T\colon(\mathcal{X},\mu)\rightarrow\mathcal{Y}$
be the partial map with domain $\mathcal{X}\smallsetminus A$ such
that $T(x)=y_{0}$ for $x\in\mathcal{X}\smallsetminus A$. By Proposition~\ref{prop:part-comp-char},
$T$ is a.e.~computable.

For the second part, let $T,S\colon(\mathcal{X},\mu)\rightarrow\mathcal{Y}$
be a.e.~computable maps that are a.e.~equal. The set 
\[
\{x\in\mathcal{X}\mid x\notin\dom T\cup\dom S\quad\text{or}\quad T(x)\neq S(x)\}
\]
 is a null $\Sigma_{2}^{0}$ set in $\mathcal{X}$. Conversely, if
$T(x)=S(x)$ for all Kurtz randoms $x$, then $T=S$ a.e.
\end{proof}

This next proposition shows that many common notions of randomness
are preserved by morphisms, and the set of randoms is preserved under
isomorphisms.
\begin{prop}
\label{prop:morphisms_preserve}If $T\colon(\mathcal{X},\mu)\rightarrow(\mathcal{Y},\nu)$
is a morphism and $x\in\mathcal{X}$ is Martin-Löf random, then $T(x)$
is Martin-Löf random. The same is true of Kurtz and Schnorr randomness.
Hence, if $T$ is an isomorphism, then $x$ is Martin-Löf (respectively
Kurtz, Schnorr) random if and only if $T(x)$ is.\end{prop}
\begin{proof}
Assume $T(x)$ is not Martin-Löf random with respect to $(\mathcal{Y},\nu)$.
Then there is a Martin-Löf test $(U_{n})$ for $(\mathcal{Y},\nu)$
which covers $T(x)$. Let $V_{n}=T^{-1}(U_{n})$ for each $n$, using
the convention of Remark~\ref{rmk:Preimage}. Then $(V_{n})$ is
a Martin-Löf test in $(\mathcal{X},\mu)$ which covers $x$. Hence
$x$ is not Martin-Löf random with respect to $(\mathcal{X},\mu)$.

The proofs for Kurtz and Schnorr randomness follow similarly. The
inverse image of a null $\Sigma_{2}^{0}$ set is still a null $\Sigma_{2}^{0}$
set, and the inverse image of a Schnorr test is still a Schnorr test.
\end{proof}
(Bienvenu and Porter have pointed out to me the following partial
converse to Proposition~\ref{prop:morphisms_preserve}, which was
first proved by Shen---see \cite{BienvenuSubmitted}. If $T\colon(\mathcal{X},\mu)\rightarrow(\mathcal{Y},\nu)$
is a morphism and $y$ is Martin-Löf random with respect to $(\mathcal{Y},\nu)$,
then there is some $x$ that is Martin-Löf random with respect to
$(\mathcal{X},\mu)$ such that $T(x)=y$.)

In Corollary~\ref{cor:comp-rand-morphisms}, we will see that computable
randomness is not preserved by morphisms. However, just looking at
the previous proof gives a clue as to why. There is another criterion
to the tests for computable randomness besides complexity and measure,
namely the cell decompositions of the space. The ``inverse image''
of a cell decomposition may not be a cell decomposition.

However, if $T$ is an isomorphism the situation is much better. Indeed,
these next three propositions show a correspondence between isomorphisms
and cell decompositions. Recall two cell decompositions $\mathcal{A}$
and $\mathcal{B}$ of a computable probability space $(\mathcal{X},\mu)$
are almost-everywhere equal (written $\mathcal{A}=\mathcal{B}$ a.e.)\ if
$[\sigma]_{\mathcal{A}}=[\sigma]_{\mathcal{B}}$ a.e.\ for all $\sigma\in2^{<\omega}$,
and two isomorphisms are almost-everywhere equal if they are pointwise
a.e.\ equal.
\begin{prop}[Isomorphisms to cell decompositions]
\label{prop:isomorphism-to-representation} If $T\colon(\mathcal{X},\mu)\rightarrow(\mathcal{Y},\nu)$
is an isomorphism and $\mathcal{B}$ is a cell decomposition of $(\mathcal{Y},\nu)$,
then there exists a cell decomposition $\mathcal{A}$, which we denote
$T^{-1}(\mathcal{B})$, such that $\name_{\mathcal{A}}(x)=\name_{\mathcal{B}}(T(x))$
for $\mu$-a.e.~$x$. Moreover, $\mathcal{A}$ is a.e.\ unique (that
is, if $\mathcal{A}'$ is a cell decomposition satisfying $\name_{\mathcal{A}'}(x)=\name_{\mathcal{B}}(T(x))$
for $\mu$-a.e.~$x$, then $\mathcal{A}'=\mathcal{A}$\ $\mu$-a.e.), and $\mathcal{A}$ is given by
$[\sigma]_{\mathcal{A}}=T^{-1}([\sigma]_{\mathcal{B}})$ a.e. In particular,
every isomorphism $T\colon(\mathcal{X},\mu)\rightarrow(2^{\omega},\nu)$
induces a cell decomposition $\mathcal{A}$ such that $\name_{\mathcal{A}}(x)=T(x)$
for $\mu$-a.e.~$x$.\end{prop}
\begin{proof}
We will show $[\sigma]_{\mathcal{A}}=T^{-1}([\sigma]_{\mathcal{B}})$
defines a cell decomposition $\mathcal{A}$. Using the convention
of Remark~\ref{rmk:Preimage}, $T^{-1}([\sigma]_{\mathcal{B}})$
is $\Sigma_{1}^{0}$ in $\sigma$. Clearly, $[\varepsilon]_{\mathcal{A}}=\mathcal{X}$\ $\mu$-a.e.,
$[\sigma0]_{\mathcal{A}}\cap[\sigma1]_{\mathcal{A}}=\varnothing$\ $\mu$-a.e.,
and $[\sigma0]_{\mathcal{A}}\cup[\sigma1]_{\mathcal{A}}=[\sigma]_{\mathcal{A}}$\ $\mu$-a.e.
Finally, take a $\Sigma_{1}^{0}$ set $U\subseteq\mathcal{X}$. By
Proposition~\ref{prop:net-to-representation}, it is enough to show
there is some c.e.\ set $\{\sigma_{i}\}$ (c.e.\ in the code for
$U$) such that $U=\bigcup_{i}[\sigma_{i}]_{\mathcal{A}}$ $\mu$-a.e.
Let $S$ be the inverse isomorphism to $T$. Then define $V=S^{-1}(U)$
using the convention of Remark~\ref{rmk:Preimage}. Then $V\subseteq\mathcal{Y}$
is $\Sigma_{1}^{0}$ (uniformly in the code for $U$) and $T^{-1}(V)=U$
$\mu$-a.e. By Proposition~\ref{prop:decompose-open-set} there is
some c.e.\ set $\{\sigma_{i}\}$ (c.e.\ in the code for $V$) such
that $V=\bigcup_{i}[\sigma_{i}]_{\mathcal{B}}$ $\nu$-a.e.\ and
therefore $U=T^{-1}(V)=\bigcup_{i}T^{-1}([\sigma_{i}]_{\mathcal{B}})=\bigcup_{i}[\sigma_{i}]_{\mathcal{A}}$
$\mu$-a.e. Therefore, $[\sigma]_{\mathcal{A}}=T^{-1}([\sigma]_{\mathcal{B}})$
defines a cell decomposition $\mathcal{A}$.

For $\mu$-a.e.\ $x$, $x\in\dom(T)\cap\dom(\name_{\mathcal{A}})$.
Then for all $n$, $x\in[x\upharpoonright_{\mathcal{A}}n]_{\mathcal{A}}=T^{-1}([x\upharpoonright_{\mathcal{A}}n]_{\mathcal{B}})$.
By Remark~\ref{rmk:Preimage}, $T(x)\in[x\upharpoonright_{\mathcal{A}}n]_{\mathcal{B}}$.
Therefore $\name_{\mathcal{B}}(T(x))=\name_{\mathcal{A}}(x)$. 

For $\mathcal{Y}=2^{\omega}$, let $\mathcal{B}$ be the natural cell
decomposition of $(2^{\omega},\nu)$, then $[\sigma]_{\mathcal{B}}=[\sigma]^{\prec}$
for all $\sigma\in2^{<\omega}$. Therefore for $\mu$-a.e.\ $x$,
$\name_{\mathcal{A}}(x)=\name_{\mathcal{B}}(T(x))=T(x)$.

To show the cell decomposition $\mathcal{A}$ is unique, assume $\mathcal{A}'$
is another cell decomposition such that for $\mu$-a.e.\ $x$, the
$\mathcal{A}$-name and $\mathcal{A}'$-name of $x$ are both the
$\mathcal{B}$ name of $T(x)$. Then $[\sigma]_{\mathcal{A}}=[\sigma]_{\mathcal{A}'}$
$\mu$-a.e. for all $\sigma\in2^{<\omega}$.\end{proof}
\begin{prop}[Cell decompositions to isomorphisms]
\label{prop:representation-to-isomorphism}Let $(\mathcal{X},\mu)$
be a computable probability space with cell decomposition $\mathcal{A}$.
There is a unique computable probability space $(2^{\omega},\mu_{\mathcal{A}})$
such that $\name_{\mathcal{A}}\colon(\mathcal{X},\mu)\rightarrow(2^{\omega},\mu_{\mathcal{A}})$
is an isomorphism. Namely, $\mu_{\mathcal{A}}(\sigma)=\mu([\sigma]_{\mathcal{A}})$.\end{prop}
\begin{proof}
If such a measure $\mu_{\mathcal{A}}$ exists, it must be unique.
Indeed, since $\name_{\mathcal{A}}$ is then measure-preserving, $\mu_{\mathcal{A}}$
must satisfy $\mu_{\mathcal{A}}(\sigma)=\mu(\name_{\mathcal{A}}^{-1}([\sigma]^{\prec}))=\mu([\sigma]_{\mathcal{A}})$,
which uniquely defines $\mu_{\mathcal{A}}$.

It remains to show the map $\name_{\mathcal{A}}\colon(\mathcal{X},\mu)\rightarrow(2^{\omega},\mu_{\mathcal{A}})$
which maps $x$ to its $\mathcal{A}$-name is an isomorphism. Clearly,
$\mu_{\mathcal{A}}$ is a computable measure since $\mu([\sigma]_{\mathcal{A}})$
is computable uniformly from $\sigma$. The map $\name_{\mathcal{A}}$
which takes $x$ to its $\mathcal{A}$-name is measure preserving
for cylinder sets and therefore for all sets by approximation. The
map from $x$ to $x\upharpoonright_{\mathcal{A}}n$ is a.e.~computable.
Indeed, wait for $x$ to show up in one of the sets $[\sigma]_{\mathcal{A}}$
where $|\sigma|=n$. Hence the map from $x$ to its $\mathcal{A}$-name
is also a.e.~computable. So $\name_{\mathcal{A}}$ is a morphism.
(As an extra verification, clearly $\dom(\name_{\mathcal{A}})$ is
a $\Pi_{2}^{0}$ measure-one set.) 

The inverse of $\name_{\mathcal{A}}$ will be the map $S$ from (a
measure-one set of) $\mathcal{A}$-names $y\in2^{\omega}$ to points
$x\in\mathcal{X}$ such that $\name_{\mathcal{A}}(x)=y$. The algorithm
for $S$ will be similar to the algorithm given by the proof of the
Baire category theorem (see \cite{Brattka2001}). Pick $y\in2^{\omega}$.
We compute $S(y)$ by a back-and-forth argument. Assume $\tau\prec y$.
Recall, $[\tau]_{\mathcal{A}}$ is $\Sigma_{1}^{0}$ in $\tau$. We
can enumerate a sequence of pairs $(a_{i},k_{i})$ where each $a_{i}$
is a simple point of $\mathcal{X}$ and each $k_{i}>|\tau|$ such
that $[\tau]_{\mathcal{A}}=\bigcup_{i}B(a_{i},2^{-k_{i}})$. Further,
by Proposition~\ref{prop:decompose-open-set}, we have that for each
$i$, there is a c.e.\ set $\{\sigma_{j}^{i}\}$ (c.e.\ in $i$)
such that $B(a_{i},2^{-k_{i}})=\bigcup_{j}[\sigma_{j}^{i}]_{\mathcal{A}}$
$\mu$-a.e. (We may assume $|\sigma_{j}^{i}|>|\tau|$ for all $i,j$.)
Given $y$, compute the Cauchy-name of $S(y)$ as follows. Start with
$\tau_{1}=y\upharpoonright1$. Then search for $\sigma_{j}^{i}\prec y$.
If we find one, let $b_{1}=a_{i}$ be the first approximation. Now
continue with $\tau_{2}=\sigma_{j}^{i}$, and so on. This gives a
Cauchy-name $(b_{n})$. The algorithm will fail if at some stage it
cannot find any $\sigma_{j}^{i}\prec y$. But then $y\in[\tau]^{\prec}\smallsetminus\bigcup_{i}\bigcup_{j}[\sigma_{i}^{j}]^{\prec}$,
which by the definition of $\mu_{\mathcal{A}}$, is a $\mu_{\mathcal{A}}$-measure-zero
set since $[\tau]_{\mathcal{A}}=\bigcup_{i}\bigcup_{j}[\sigma_{i}^{j}]_{\mathcal{A}}$
$\mu$-a.e. Hence $S$ is a.e.~computable.

By the back-and-forth algorithm, $\name_{\mathcal{A}}(S(y))=y$ for
all $y\in\dom(S)$. To show $S(\name_{\mathcal{A}}(x))=x$ a.e.,\ assume
$x\in\dom(\name_{\mathcal{A}})$. Consider the back-and-forth sequence
created by the algorithm: $[\tau_{n}]_{\mathcal{A}}\supseteq B(b_{n},2^{-k_{n}})\supseteq[\tau_{n+1}]_{\mathcal{A}}\supseteq\ldots$.
For all $n$, we have $\tau_{n}\prec\name_{\mathcal{A}}(x)$, then
$x\in[\tau_{n}]_{\mathcal{A}}$ for all $n$. So $x=\lim_{n\rightarrow\infty}b_{n}=S(\name_{\mathcal{A}}(x))$.
Since $S^{-1}([\sigma]_{\mathcal{A}})=S^{-1}(\name_{\mathcal{A}}^{-1}([\sigma]^{\prec}))=[\sigma]^{\prec}$
$\mu_{\mathcal{A}}$-a.e., $S$ is a measure-preserving map, and hence
a morphism. Therefore, $\name_{\mathcal{A}}$ is an isomorphism.
\end{proof}
These last two propositions show that there is a one-to-one correspondence
between cell decompositions $\mathcal{A}$ of a space $(\mathcal{X},\mu)$
and isomorphisms of the form $T\colon(\mathcal{X},\mu)\rightarrow(2^{\omega},\nu)$.
This next proposition shows a further one-to-one correspondence between
isomorphisms $T\colon(\mathcal{X},\mu)\rightarrow(\mathcal{Y},\nu)$
and $S\colon(2^{\omega},\mu_{\mathcal{A}})\rightarrow(2^{\omega},\nu_{\mathcal{B}})$.
\begin{prop}[Pairs of cell decompositions to isomorphisms]
\label{prop:representation-pairs-to-isomorphisms}Let $(\mathcal{X},\mu)$
and $(\mathcal{Y},\nu)$ be computable probability spaces with cell
decompositions $\mathcal{A}$ and $\mathcal{B}$. Let $\mu_{\mathcal{A}}$
be as in Proposition~\ref{prop:representation-to-isomorphism}, and
similarly for $\nu_{\mathcal{B}}$. Then for every isomorphism $T\colon(\mathcal{X},\mu)\rightarrow(\mathcal{Y},\nu)$
there is an a.e.\ unique isomorphism $S\colon(2^{\omega},\mu_{\mathcal{A}})\rightarrow(2^{\omega},\nu_{\mathcal{B}})$
and vice versa, such that $S$ maps $\name_{\mathcal{A}}(x)$ to $\name_{\mathcal{B}}(T(x))$
for $\mu$-a.e.~$x\in\mathcal{X}$. In other words the following
diagram commutes for $\mu$-a.e.~$x\in\mathcal{X}$. \[
\begin{tikzcd}[column sep=large] 
    (\mathcal{X},\mu)
      \arrow{r}{\name_{\mathcal{A}}} 
      \arrow{d}{T}
  & (2^\omega,\mu_{\mathcal{A}}) 
      \arrow{d}{S} \\
    (\mathcal{Y},\nu) 
      \arrow{r}{\name_{\mathcal{B}}}
  & (2^\omega,\nu_{\mathcal{B}})
\end{tikzcd}
\]Further we have the following.
\begin{enumerate}
\item If $(\mathcal{X},\mu)$ equals $(\mathcal{Y},\nu)$, then $T$ is
the identity isomorphism precisely when $S$ is the isomorphism which
maps $\name_{\mathcal{A}}(x)$ to $\name_{\mathcal{B}}(x)$. 
\item Conversely, $S$ is the identity isomorphism (and hence $\mu_{\mathcal{A}}$
equals $\nu_{\mathcal{B}}$) precisely when $\mathcal{A}=T^{-1}(\mathcal{B})$
(as in Proposition~\emph{\ref{prop:isomorphism-to-representation}}).
\end{enumerate}
\end{prop}
\begin{proof}
Given $T$, let $S=\name_{\mathcal{A}}^{-1}\circ T\circ\name_{\mathcal{B}}$,
and similarly to get $T$ from $S$. Then the diagram clearly commutes.
A.e.~uniqueness follows since the maps are isomorphisms. 

If $\mathcal{A}=T^{-1}(\mathcal{B})$ is induced by $T$, then by
Proposition~\ref{prop:isomorphism-to-representation}, $\name_{\mathcal{A}}(x)=\name_{\mathcal{B}}(T(x))$
which makes $S$ the identity map. But since $S$ is an isomorphism,
$\mu_{\mathcal{A}}$ and $\nu_{\mathcal{B}}$ must be the same measure. 

The rest follows from ``diagram chasing.''
\end{proof}
Now we can show computable randomness is preserved by isomorphisms.
\begin{thm}
\label{thm:isomorphisms_preserve}Isomorphisms preserve computable
randomness. Namely, given an isomorphism $T\colon(\mathcal{X},\mu)\rightarrow(\mathcal{Y},\nu)$,
then $x\in\mathcal{X}$ is computably random if and only if $T(x)$
is.\end{thm}
\begin{proof}
Assume $T(x)$ is not computably random. Fix an isomorphism $T\colon(\mathcal{X},\mu)\rightarrow(\mathcal{Y},\nu)$.
Let $\mathcal{B}$ be a cell decomposition of $(\mathcal{Y},\nu)$.
Take a bounded Martin-Löf test $(U_{n})$ on $(\mathcal{Y},\nu)$
with bounding measure $\kappa$ with respect to $\mathcal{\mathcal{B}}$
which covers $T(x)$. By Proposition~\ref{prop:isomorphism-to-representation}
there is a cell decomposition $\mathcal{A}=T^{-1}(\mathcal{B})$ on
$(\mathcal{X},\mu)$ such that $[\sigma]_{\mathcal{A}}=T^{-1}([\sigma]_{\mathcal{B}})$
for all $\sigma\in2^{<\omega}$. Define $V_{n}=T^{-1}(U_{n})$, using the convention of Remark~\ref{rmk:Preimage}.
Then $(V_{n})$ is a bounded Martin-Löf test on $(\mathcal{X},\mu)$
bounded by the same measure $\kappa$ with respect to $\mathcal{A}$.
Indeed,
\[
\mu(V_{n}\cap[\sigma]_{\mathcal{A}})=\nu(U_{n}\cap[\sigma]_{\mathcal{B}})\leq2^{-n}\kappa(\sigma).
\]
Also, $(V_{n})$ covers $x$, hence $x$ is not computably random.
\end{proof}
Using Theorem~\ref{thm:isomorphisms_preserve}, we can explore computable
randomness with respect to various spaces.
\begin{example}[Computably random vectors]
Let $([0,1]^{d},\lambda)$ be the cube $[0,1]^{d}$ with the Lebesgue
measure. The following is a natural isomorphism from $([0,1]^{d},\lambda)$
to $(2^{\omega},\lambda)$. First, represent $(x_{1},\ldots,x_{d})\in[0,1]^{d}$
by the binary sequence of each component; then interleave the binary
sequences. By Theorem~\ref{thm:isomorphisms_preserve}, $(x_{1},\ldots,x_{d})$
is computably random if and only if the sequence of interleaved binary
sequences is computably random. (This definition of computable randomness
with respect to $[0,1]^{d}$ was proposed by Brattka, Miller and Nies\ \cite{Brattka2011}.)
\end{example}

\begin{example}[Base invariance]
\label{ex:Base-inv-2}Let $\lambda_{3}$ be the uniform measure on
$3^{\omega}$. Consider the natural isomorphism $T_{2,3}\colon(2^{\omega},\lambda)\rightarrow(3^{\omega},\lambda_{3})$
which identifies the binary and ternary expansions of a real. This
is an a.e.~computable isomorphism, so $x\in[0,1]$ is computably
random if and only if $T_{2,3}(x)$ is computably random. We say a
randomness notion (defined on $(b^{\omega},\lambda_{b})$ for all
$b\geq2$, see Example~\ref{ex:comp-rand-on-3-omega-1}) is \noun{base invariant}
if this property holds for all base pairs $b_{1},b_{2}$.
\end{example}

\begin{example}[Computably random closed set]
Consider the space $\mathcal{F}(2^{\omega})$ of closed sets of $2^{\omega}$.
This space has a topology called the Fell topology. The subspace $\mathcal{F}(2^{\omega})\smallsetminus\{\varnothing\}$
can naturally be identified with trees on $\{0,1\}$ with no dead
branches. Barmpalias et al.\ \cite{Barmpalias2007} gave a natural
construction of these trees from ternary sequences in $3^{\omega}$.
Axon \cite{Axon2010} showed the corresponding map $T\colon3^{\omega}\rightarrow\mathcal{F}(2^{\omega})\smallsetminus\{\varnothing\}$
is a homeomorphism between $3^{\omega}$ and the Fell topology restricted
to $\mathcal{F}(2^{\omega})\smallsetminus\{\varnothing\}$. Hence
$\mathcal{F}(2^{\omega})\smallsetminus\{\varnothing\}$ can be represented
as a computable Polish space, and the probability measure on $\mathcal{F}(2^{\omega})\smallsetminus\{\varnothing\}$
induced by $T$ can be represented as a computable probability measure.
Since $T$ is an a.e.~computable isomorphism, the computably random
closed sets of this space are then the ones whose corresponding trees
are constructed from computably random ternary sequences in $3^{\omega}$.
\end{example}

\begin{example}[Computably random structures]
The last example can be extended to a number of random structures---infinite
random graphs, Markov processes, random walks, random matrices, Galton-Watson
processes, etc. The main idea is as follows. Assume $(\Omega,P)$
is a computable probability space (the sample space), $\mathcal{X}$
is the space of structures, and $T\colon(\Omega,P)\rightarrow\mathcal{X}$
is an a.e.\ computable map (a random structure). This induces a measure
$\mu$ on $\mathcal{X}$ (the distribution of $T$). If, moreover,
$T$ is an a.e.~computable isomorphism between $(\Omega,P)$ and
$(\mathcal{X},\mu)$, then the computably random structures of $(\mathcal{X},\mu)$
are exactly the objects constructed from computably random points
in $(\Omega,P)$.
\end{example}
In this next theorem, an \noun{atom} (or \noun{point-mass}) is a
point with positive measure. An \noun{atomless} probability space
is one without atoms.
\begin{prop}
\label{prop:atomless-isomorphic}If $(\mathcal{X},\mu)$ is an atomless
computable probability space, then there is a isomorphism $T\colon(\mathcal{X},\mu)\rightarrow(2^{\omega},\lambda)$.
Further, $T$ is computable in each code for $(\mathcal{X},\mu)$.\end{prop}
\begin{proof}
Given a code for $(\mathcal{X},\mu)$, compute an $\mu$-a.e.~decidable
cell decomposition $\mathcal{A}$ for $(\mathcal{X},\mu)$. By Proposition~\ref{prop:representation-to-isomorphism},
$\name_{\mathcal{A}}\colon(\mathcal{X},\mu)\rightarrow(2^{\omega},\mu_{\mathcal{A}})$
is an isomorphism. Since $\mu$ is atomless, so is $\mu_{\mathcal{A}}$.
Let $S\colon(2^{\omega},\mu_{\mathcal{A}})\rightarrow([0,1],\lambda)$
be the total map $x\mapsto\mu_{\mathcal{A}}(\{y\in2^{\omega}\mid y\leq x\})$
where $\leq$ is the usual lexicographical order on $2^{\omega}$.
Since $\mu_{\mathcal{A}}$ is atomless, $S$ is computable and measure-preserving.
Last, $S$ is an isomorphism since for $\lambda$-a.e.~$a\in[0,1]$,
there is exactly one $x\in2^{\omega}$ such that $S(x)=a$, and this
$x$ can be computed from $a$ using the monotonicity of $S$. The
result follows, given the isomorphism between $([0,1],\lambda)$ and
$(2^{\omega},\lambda)$.\end{proof}
\begin{cor}
\label{cor:atomless-measure}If $(\mathcal{X},\mu)$ is an atomless
computable probability space, then $x\in\mathcal{X}$ is computably
random if and only if $T(x)$ is computably random for any (and all)
isomorphisms $T\colon(\mathcal{X},\mu)\rightarrow(2^{\omega},\lambda)$.\end{cor}
\begin{proof}
Follows from Theorem~\ref{thm:isomorphisms_preserve} and Proposition~\ref{prop:atomless-isomorphic}.
\end{proof}

\begin{example}[Computably random Brownian motion]
Consider the space $C([0,1])$ of continuous functions from $[0,1]$
to $\mathbb{R}$ endowed with the Wiener probability measure $W$
(i.e.~the measure of Brownian motion). The space $C([0,1])$ with
the uniform norm is a computable Polish space (where the simple points
are the rational piecewise linear functions). The measure $W$ is
an atomless computable probability measure. Let $T\colon(2^{\omega},\lambda)\rightarrow(C([0,1]),W)$
be the isomorphism from Proposition~\ref{prop:atomless-isomorphic}.
(Kjos-Hanssen and Nerode \cite{Kjos-Hanssen:2009hc} construct a similar
map $\varphi$ directly for Brownian motion.) By Corollary~\ref{cor:atomless-measure},
the computably random Brownian motions (i.e.~the computably random
points of $(C([0,1]),W)$) are exactly the forward image of the computable
random sequences under the map $T$.\end{example}
\begin{cor}
\label{cor:comp-rand-A-name}Given a measure $(\mathcal{X},\mu)$
with cell decomposition $\mathcal{A}$, $x\in X$ is computably random
if and only if $\name_{\mathcal{A}}(x)$ is computably random with
respect to $(2^{\omega},\mu_{\mathcal{A}})$ where $\mu_{\mathcal{A}}(\sigma)=\mu([\sigma]_{\mathcal{A}})$.\end{cor}
\begin{proof}
Use Proposition~\ref{prop:representation-to-isomorphism} and Theorem~\ref{thm:isomorphisms_preserve}.
\end{proof}

\section{Generalizing randomness to computable probability spaces\label{sec:Generalizing-randomness}}

In this section, I explain the general method of this paper which
generalizes a randomness notion from $(2^{\omega},\lambda)$ to an
arbitrary computable measure space.

Imagine we have an arbitrary randomness notion called \noun{$\mathsf{X}$-randomness}
defined on $(2^{\omega},\lambda)$. (Here $\mathsf{X}$ is a place-holder
for a name like ``Schnorr'' or ``computable''; it has no relation
to being random relative to an oracle.) The definition of $\mathsf{X}$-random
should either explicitly or implicitly assume we are working in the
fair-coin measure. The method can be reduced to three steps.

\subsection*{Step 1: Generalize $\mathsf{X}$-randomness to computable probability
measures on $2^{\omega}$.}

This is self-explanatory, although not always trivial.

\subsection*{Step 2: Generalize $\mathsf{X}$-randomness to computable probability
spaces.}

There are three equivalent ways to do this for a computable probability
space $(\mathcal{X},\mu)$.
\begin{enumerate}
\item Replace all instances of $[\sigma]^{\prec}$ with $[\sigma]_{\mathcal{A}}$,
$x\upharpoonright n$ with $x\upharpoonright_{\mathcal{A}}n$, etc.\ in
the definition of $\mathsf{X}$-random from Step 1. Call this $\mathsf{X}^{*}$-random
with respect to $\mathcal{A}$. Then define $x\in\mathcal{X}$ to
be $\mathsf{X}^{*}$-random with respect to $(\mathcal{X},\mu)$ if
it is $\mathsf{X}^{*}$-random with respect to all cell decompositions
$\mathcal{A}$ (ignoring unrepresented points of $\mathcal{A}$ and
points in null cells ---which are not even Kurtz random). (Compare
with Definition~\ref{def:comp_rand}.)
\item Define $x\in\mathcal{X}$ to be $\mathsf{X}^{*}$-random with respect
to $(\mathcal{X},\mu)$ if for each cell decomposition $\mathcal{A}$,
$\name_{\mathcal{A}}(x)$ is $\mathsf{X}$-random with respect to
$(2^{\omega},\mu_{\mathcal{A}})$, where $\mu_{\mathcal{A}}$ is given
by $\mu_{\mathcal{A}}(\sigma)=\mu([\sigma]_{\mathcal{A}})$. (Compare
with Corollary~\ref{cor:comp-rand-A-name}.)
\item Define $x\in\mathcal{X}$ to be $\mathsf{X}^{*}$-random with respect
to $(\mathcal{X},\mu)$ if for all isomorphisms $T\colon(\mathcal{X},\mu)\rightarrow(2^{\omega},\nu)$
we have that $T(x)$ is $\mathsf{X}$-random with respect to $(2^{\omega},\nu)$.
(Compare with Theorem~\ref{thm:isomorphisms_preserve}.)
\end{enumerate}
The description of (1) is a bit vague, but when done correctly it
is the most useful definition. The definition given by (1) should
be equivalent to that given by (2) because (1) is essentially about
$\mathcal{A}$-names. To see that (2) and (3) give the same definition,
use Propositions~\ref{prop:isomorphism-to-representation} and \ref{prop:representation-to-isomorphism},
which show that isomorphisms to $2^{\omega}$ are maps to $\mathcal{A}$-names
and vice versa.

\subsection*{Step 3: Verify that the new definition is consistent with the original.}

It may be that on $(2^{\omega},\lambda)$ the class of $\mathsf{X}^{*}$-random
points is strictly smaller that the class of the original $\mathsf{X}$-random
points. There are three equivalent techniques to show that $\mathsf{X}^{*}$-randomness
with respect to $2^{\omega}$ is equivalent to $\mathsf{X}$-randomness.
The three techniques are related to the three definitions from Step
2.
\begin{enumerate}
\item Show the definition of $\mathsf{X}^{*}$-randomness is invariant under
the choice of cell decomposition. (Compare with Theorem~\ref{thm:indepentant-of-representation}.)
\item Show that for every two cell decompositions $\mathcal{A}$ and $\mathcal{B}$,
the $\mathcal{A}$-name of $x$ is $\mathsf{X}$-random with respect
to $(2^{\omega},\mu_{\mathcal{A}}$) if and only if the $\mathcal{B}$-name
is $\mathsf{X}$-random with respect to $(2^{\omega},\mu_{\mathcal{B}}$).
(Compare with Corollary~\ref{cor:comp-rand-A-name}.)
\item Show that $\mathsf{X}$-randomness is invariant under all isomorphisms
from $(2^{\omega},\mu)$ to $(2^{\omega},\nu)$. (Compare with Theorem~\ref{thm:isomorphisms_preserve}.)
\end{enumerate}
Again, these three approaches are equivalent. Assuming the definition
is stated correctly, (1) and (2) say the same thing. 

To see that (3) implies (2), assume $\mathsf{X}$-randomness is invariant
under isomorphisms on $2^{\omega}$. Consider two cell decompositions
$\mathcal{A}$ and $\mathcal{B}$ of the same space $(\mathcal{X},\mu)$.
By Proposition~\ref{prop:representation-pairs-to-isomorphisms}\,(1),
there is an isomorphism $S\colon(2^{\omega},\mu_{\mathcal{A}})\rightarrow(2^{\omega},\mu_{\mathcal{B}})$
which maps $\mathcal{A}$-names to $\mathcal{B}$-names, i.e.\ this
diagram commutes.\[
\begin{tikzcd}[column sep=large] 
    (\mathcal{X},\mu) 
      \arrow{r}{\name_{\mathcal{A}}}
      \arrow{dr}[swap, sloped, near end] 
        {\name_{\mathcal{B}}} 
  & (2^\omega,\mu_{\mathcal{A}}) 
      \arrow{d}{S} \\
    {}
  & (2^\omega,\mu_{\mathcal{B}})
\end{tikzcd}
\]Since $S$ preserves $\mathsf{X}$-randomness, $\name_{\mathcal{A}}(x)$
is $\mathsf{X}$-random with respect to $(2^{\omega},\mu_{\mathcal{A}})$
if and only if and only if $\name_{\mathcal{B}}(x)$ is $\mathsf{X}$-random
with respect to $(2^{\omega},\mu_{\mathcal{B}})$.

To see that (2) implies (3), assume that (2) holds. Consider an isomorphism
$S\colon(2^{\omega},\mu)\rightarrow(2^{\omega},\nu)$. Let $(\mathcal{X},\kappa)$
be any space isomorphic to $(2^{\omega},\mu)$. Then $(\mathcal{X},\kappa)$
is also isomorphic to $(2^{\omega},\nu)$. So there are isomorphisms
$T_{1}$ and $T_{2}$ such that this diagram commutes. \[
\begin{tikzcd} 
    (\mathcal{X},\kappa) 
      \arrow{r}{T_1} 
      \arrow{dr}[swap]{T_2} 
  & (2^\omega,\mu) \arrow{d}{S} \\
    {}
  & (2^\omega,\nu) \\
\end{tikzcd}
\]By Proposition~\ref{prop:isomorphism-to-representation} there are
two cell decompositions $\mathcal{A}$ and $\mathcal{B}$ on $(\mathcal{X},\kappa)$
such that $T_{1}=\name_{\mathcal{A}}$ and $(2^{\omega},\mu)=(2^{\omega},\kappa_{\mathcal{A}})$.
The same holds for $\mathcal{B}$ and $\nu$. Then we have this commutative
diagram. \[
\begin{tikzcd}[column sep=large]
    (\mathcal{X},\kappa) 
      \arrow{r}{\name_{\mathcal{A}}}
      \arrow{dr}[swap, sloped, near end] 
        {\name_{\mathcal{B}}}
  & (2^\omega,\kappa_{\mathcal{A}}) \arrow{d}{S} \\
    {}
  & (2^\omega,\kappa_{\mathcal{B}}) \\
\end{tikzcd}
\]Consider any $\mathsf{X}$-random $y\in(2^{\omega},\kappa_{\mathcal{A}})$.
It is the $\mathcal{A}$-name of some $x\in(\mathcal{X},\kappa)$,
in other words $y=\name_{\mathcal{A}}(x)$. By (2), we also have that
$\name_{\mathcal{B}}(x)$ is $\mathsf{X}$-random. So $S$ preserves
$\mathsf{X}$-randomness.

~

Notice that Step 3 implies that some randomness notions cannot be
generalized without making the set of randoms smaller. This is because
they are not invariant under isomorphisms between computable probability
measures on $2^{\omega}$. Yet, even when the $\mathsf{X}^{*}$-randoms
are a proper subclass of the $\mathsf{X}$-randoms, the $\mathsf{X}^{*}$
randoms are an interesting class of randomness. In particular we have
the following.
\begin{prop}
$\mathsf{X}^{*}$-randomness is invariant under isomorphisms.
\end{prop}
In some sense the $\mathsf{X}^{*}$-randoms are the largest such subclass
of the $\mathsf{X}$-randoms. (One must be careful how to say this,
since $\mathsf{X}$-randomness is only defined with respect to measures
$(2^{\omega},\mu)$.)
\begin{proof}
Let $T\colon(\mathcal{X},\mu)\rightarrow(\mathcal{Y},\nu)$ be an
isomorphism and let $x\in(\mathcal{X},\mu)$ be $\mathsf{X}^{*}$-random.
Let $\mathcal{B}$ be a arbitrary cell decomposition of $(\mathcal{Y},\nu)$.
Since $\mathcal{B}$ is arbitrary, it is enough to show that $\name_{B}(T(x))$
is $\mathsf{X}$-random with respect to $(2^{\omega},\nu_{\mathcal{B}})$.
By Proposition~\ref{prop:isomorphism-to-representation} and Proposition~\ref{prop:representation-pairs-to-isomorphisms}\,(2)
we have a cell decomposition $\mathcal{A}$ on $(\mathcal{X},\mu)$
such that $(2^{\omega},\mu_{\mathcal{A}})=(2^{\omega},\nu_{\mathcal{B}})$
and the following diagram commutes. \[
\begin{tikzcd}[column sep=large] 
    (\mathcal{X},\mu) 
      \arrow{r}{\name_{\mathcal{A}}} 
      \arrow{d}{T} 
  & (2^\omega,\mu_{\mathcal{A}}) 
      \arrow[equals, semithick ]{d}{} \\
    (\mathcal{Y},\nu) 
      \arrow{r}{\name_{\mathcal{B}}}
  & (2^\omega,\nu_{\mathcal{B}})
\end{tikzcd}
\]Since $x\in(\mathcal{X},\mu)$ is $\mathsf{X}^{*}$-random, $\name_{\mathcal{A}}(x)$
is $\mathsf{X}$-random with respect to $(2^{\omega},\mu_{\mathcal{A}})=(2^{\omega},\nu_{\mathcal{B}})$.
Since the diagram commutes, $\name_{B}(T(x))$ is also $\mathsf{X}$-random
with respect to $(2^{\omega},\nu_{\mathcal{B}})$. Since $\mathcal{B}$
is arbitrary, $x$ is $\mathsf{X}$-random.
\end{proof}
In the case that $(\mathcal{X},\mu)$ is an atomless computable probability
measure, we could instead define $x\in\mathcal{X}$ to be $\mathsf{X}^{\bigstar}$-random
if $T(x)$ is random for all isomorphisms $T\colon(\mathcal{X},\mu)\rightarrow(2^{\omega},\lambda)$.
We can then skip Step~1, and in Step~3 it is enough to check that
$\mathsf{X}$-randomness is invariant under automorphisms of $(2^{\omega},\lambda)$.
Similarly, $\mathsf{X}^{\bigstar}$-randomness would be invariant
under isomorphisms.

\section{Betting strategies and Kol\-mo\-go\-rov-Love\-land randomness\label{sec:Betting-strategies-and-KL}}

In the next two sections I consider how the method of Section~\ref{sec:Generalizing-randomness}
can be applied to Kol\-mo\-go\-rov-Love\-land randomness, which
is also defined through a betting strategy on the bits of the sequence.

Call a betting strategy on bits \noun{nonmonotonic} if the gambler
can decide at each stage which coin toss to bet on. For example, maybe
the gambler first bets on the $5$th bit. If it is $0$, then he bets
on the $3$rd bit; if it is $1$, he bets on the $8$th bit. (Here,
and throughout this paper we still assume the gambler cannot bet more
than what is in his capital, i.e.~he cannot take on debt.) A sequence
$x\in2^{\omega}$ is \noun{Kol\-mo\-go\-rov-Love\-land random} or
\noun{nonmonotonically random} (with respect to $(2^{\omega},\lambda)$)
if there is no computable nonmonotonic betting which succeeds on $x$.

Indeed, this gives a lot more freedom to the gambler and leads to
a strictly stronger notion than computable randomness. While it is
easy to show that every Martin-Löf random is Kol\-mo\-go\-rov-Love\-land
random, the converse is a difficult open question.
\begin{question}
\label{open:KL_equals_ML}Is Kol\-mo\-go\-rov-Love\-land randomness
the same as Martin-Löf randomness?
\end{question}
On one hand, there are a number of results that show Kol\-mo\-go\-rov-Love\-land
randomness is very similar to Martin-Löf randomness. On the other
hand, it is not even known if Kol\-mo\-go\-rov-Love\-land randomness
is base invariant, and it is commonly thought that Kol\-mo\-go\-rov-Love\-land
randomness is strictly weaker than Martin-Löf randomness. For the
most recent results on Kol\-mo\-go\-rov-Love\-land randomness
see the books by Downey and Hirschfeldt \cite[Section\ 7.5]{Downey2010}
and Nies \cite[Section\ 7.6]{Nies2009}, as well as the articles by
Bienvenu, Hölzl, Kräling and Merkle \cite{Bienvenu2010}, Kastermans
and Lempp \cite{Kastermans2010}, and Merkle, Miller, Nies, Reimann
and Stephan \cite{Merkle2006a}.

In this section I will ask what type of randomness one gets by applying
the method of Section~\ref{sec:Generalizing-randomness} to Kol\-mo\-go\-rov-Love\-land
randomness. The result is Martin-Löf randomness. However, this does
not prove that Kol\-mo\-go\-rov-Love\-land randomness is the same
as Martin-Löf randomness, since I leave as an open question whether
Kol\-mo\-go\-rov-Love\-land randomness (naturally extended to
all computable probability measures on $2^{\omega}$) is invariant
under isomorphisms. The presentation of this section follows the three-step
method of Section~\ref{sec:Generalizing-randomness}.

\subsection{\label{sub:Step-1-KL}Step 1: Generalize to other computable probability
measures $\mu$ on $2^{\omega}$}

Kol\-mo\-go\-rov-Love\-land randomness can be naturally extended
to computable probability measures on $2^{\omega}$. Namely, bet as
usual, but adjust the payoffs to be fair. For example, if the gambler
wagers $1$ unit of money to bet that $x(4)=1$ (i.e.\ the $4$th
bit is $1$) after seeing that $x(2)=1$ and $x(6)=0$, then if he
wins, the fair payoff is
\[
\frac{\mu\left(x(4)=0\mid x(2)=1,x(6)=0\right)}{\mu\left(x(4)=1\mid x(2)=1,x(6)=0\right)}
\]
where $\mu(A\mid B)=\mu(A\cap B)/\mu(B)$ represents the conditional
probability of $A$ given $B$. He also keeps the unit that he wagered.
If the gambler loses, he loses his unit of money. 

(Note, we could also allow the gambler to bet on a bit he has already
seen. Indeed, he will not win any money. This would, however, introduce
``partial randomness'' since the gambler could delay betting on
a new bit. Nonetheless, Merkle \cite{Merkle2003} showed that partial
Kol\-mo\-go\-rov-Love\-land randomness is the same as Kol\-mo\-go\-rov-Love\-land
randomness.)

As with computable randomness, we must address division by zero. The
gambler is not allowed to bet on a bit if it has probability zero
of occurring (conditioned on the information already known). Instead
we just declare the elements of such null cylinder sets to be not
random.

\subsection{Step 2: Generalize Kol\-mo\-go\-rov-Love\-land randomness to
computable probability measures}

Fix a computable probability space $(\mathcal{X},\mu)$ with generator
$\mathcal{A}=(A_{n})$. Following the second step of the method in
Section~\ref{sec:Generalizing-randomness}, the gambler bets on the
bits of the $\mathcal{A}$-name of $x$. A little thought reveals
that what the gambler is doing when she bets that the $n$th bit of
the $\mathcal{A}$-name is $1$ is betting that $x\in A_{n}$. For
any generator $\mathcal{A}$, if we add more a.e.~decidable sets
to $\mathcal{A}$, it is still a generator. Further, since we are
not necessary betting on all the sets in $\mathcal{A}$, we do not
even need to assume $\mathcal{A}$ is anything more than a computably
indexed family of a.e.~decidable sets. (This is the key difference
between our new notion of randomness and computable randomness.)

Hence, we may think of the betting strategy as follows. At each stage,
the gambler chooses some a.e.~decidable set $A$ and bets that $x\in A$
(or $x$ has property $A$). (Again, the gambler must know that $\mu(A)>0$
before betting on it.) Then if she wins, she gets a fair payoff, and
if she loses, she loses her stake. Based on her win/loss, she then
chooses a new set $B$ to bet on, and so on.\footnote{Assume that her strategy is ``total'', in that she already has picked
out a set $B$ to bet on if she wins on $A$ and a set $B'$ to bet
on if she losses on $A$, and so on. (See Remark~\ref{rem:betting-strategy}
for a formal definition.) Although, as already mentioned, she can
mimic a partial strategy by betting on a set she has already seen.
Moreover, she can combine $B$ and $B'$ into one set by betting on
$(B\cap A)\cup(B'\cap(\mathcal{X}\smallsetminus A))$.} Call such a strategy a \noun{computable betting strategy}. Call
the resulting randomness \noun{betting randomness}. (A more formal
definition is given in Remark~\ref{rem:betting-strategy}.)

I argue that betting randomness is the most general randomness notion
that can be described by a finitary fair-game betting scenario with
a ``computable betting strategy.'' Indeed, consider these three
basic properties of such a game:
\begin{enumerate}
\item The gambler must be able to determine (almost surely) some property
of $x$ that she is betting on, and this determination must be made
with only the information about $x$ that she has gained during the
game.
\item A bookmaker must be able to determine (almost surely) if this property
holds of $x$ or not.
\item If the gambler wins, the bookmaker must be able to determine (almost
surely) the fair payoff amount.
\end{enumerate}
The only way to satisfy (2) is if the property is a.e.~decidable.
Then (3) follows since a.e.~decidable sets have finite descriptions
and their measures are computable. To satisfy (1), the gambler must
be able to compute the a.e.~decidable set only knowing the results
of her previous bets. This is exactly the computable betting strategy
defined above.\footnote{In the three properties we did not consider the possibility of betting
on a collection of three or more pair-wise disjoint events simultaneously.
This is not an issue since one may use techniques similar to those
in Example~\ref{ex:comp-rand-on-3-omega-1}. There is also a more
general possibility of having a computable or a.e.~computable wager
function over the space $\mathcal{X}$. This can be made formal using
the martingales in probability theory, but it turns out that it does
not change the randomness characterized by such a strategy. By an
unpublished result of Ed Dean {[}personal communication{]}, any $L^{1}$-bounded
layerwise-computable martingale converges on Martin-Löf randomness
(which, as we will see, is equivalent to betting randomness).}

Now recall Schnorr's Critique that Martin-Löf randomness does not
have a ``computable-enough'' definition. The definition Schnorr
had in mind was a betting scenario. In particular, Schnorr gave a
martingale characterization of Martin-Löf randomness that is the same
as that of computable randomness, except the martingales are only
lower semicomputable \cite{Schnorr1971} (see also \cite{Downey2010,Nies2009}).
If Martin-Löf randomness equals Kol\-mo\-go\-rov-Love\-land randomness,
then some believe that this will give a negative answer to Schnorr's
Critique; namely, we will have found a computable betting strategy
that describes Martin-Löf randomness. While, there is some debate
as to what Schnorr meant by his critique (and whether he still agrees
with it), we think the following is a worthwhile question.\smallskip{}

\begin{quote}
\emph{Can Martin-Löf randomness be characterized using a finitary fair-game
betting scenario with a ``computable betting strategy''?}
\end{quote}
\smallskip{}
The answer turns out to be yes. As this next theorem shows, betting
randomness is equivalent to Martin-Löf randomness. Hitchcock and Lutz
\cite{Hitchcock2006} defined a generalization of martingales (as
in the type used to define computable randomness with respect to $2^{\omega}$)
called martingale processes. In the terminology of this paper, a \noun{martingale process}
is basically a computable betting strategy on $2^{\omega}$ with the
fair-coin measure which bets on decidable sets (i.e.~finite unions
of basic open sets). Merkle, Mihailovi\'{c} and Slaman \cite{Merkle2006}
showed that Martin-Löf randomness is equivalent to the randomness
characterized by martingale processes. The proof of this next theorem
is basically the Merkle et al.\ proof.\footnote{Downey and Hirschfelt \cite[footnote on p.~269]{Downey2010} also
remark that the Merkle et al.\ result gives a possible answer to
Schnorr's critique.}
\begin{thm}
\label{thm:Betting-random-is-ML}Betting randomness and Martin-Löf
randomness are the same.\end{thm}
\begin{proof}
Fix a computable probability space $(\mathcal{X},\mu)$. To show Martin-Löf
randomness implies betting randomness, we use a standard argument
which was employed by Hitchcock and Lutz \cite{Hitchcock2006} for
martingale processes. Assume $x\in\mathcal{X}$ is not betting random.
Namely, there is some computable betting strategy $\mathcal{B}$ which
succeeds on $x$. Without loss of generality, the starting capital
of $\mathcal{B}$ may be assumed to be $1$. Let $U_{n}=\{x\in\mathcal{X}\mid\mathcal{B}\text{ wins more than }2^{n}\text{ on }x\}$.
Each $U_{n}$ is $\Sigma_{1}^{0}$ in $n$, and by a standard result
in martingale theory $\mu(U_{n})\leq C2^{-n}$ where $C=1$ is the
starting capital.\footnote{This follows from Kolmogorov\textquoteright s inequality (proved by
Ville, see \cite[Theorem 6.3.3 and Lemma 6.3.15\,(ii)]{Downey2010})
which is a straight-forward application of Doob's submartingale inequality
(see for example \cite[Section 16.4]{Williams1991}).} Hence $(U_{n})$ is a Martin-Löf test which covers $x$, and $x$
is not Martin-Löf random.

For the converse, the argument is basically the Merkle, Mihailovi\'{c}
and Slaman \cite{Merkle2006} proof for martingale processes.

First, let us prove a fact. Assume a gambler starts with a capital
of $1$ and $U\subset\mathcal{X}$ is some $\Sigma_{1}^{0}$ set such
that $\mu(U)\leq1/2$. Then there is a computable way that the gambler
can bet on an unknown $x\in\mathcal{X}$ such that he doubles his
capital (to 2) if $x\in U$ (actually, some $\Sigma_{1}^{0}$ set
a.e.\ equal to $U$). The strategy is as follows. Choose a cell decomposition
$\mathcal{A}$ of $(\mathcal{X},\mu)$. Since $U$ is $\Sigma_{1}^{0}$,
by Proposition~\ref{prop:decompose-open-set} there is a c.e., prefix-free
set $\{\sigma_{i}\}$ (c.e.\ in the code for $U$) of finite strings
such that $U=\bigcup_{i}[\sigma_{i}]_{\mathcal{A}}$ a.e. We may assume
$\mu([\sigma_{i}]_{\mathcal{A}})>0$ for all $i$. To start, the gambler
bets on the set $[\sigma_{0}]_{\mathcal{A}}$ with a wager such that
if he wins, his capital is $2$. If he wins, he is done. If he loses,
then he bets on the set $[\sigma_{1}]_{\mathcal{A}}$, and so on.
Since the set $\{\sigma_{i}\}$ may be finite, the gambler may not
have a set to bet on at certain stages. This is not an issue, since
he may just bet on the whole space. This is functionally equivalent
to not betting at all since he wins no money. 

The only difficulty now is showing that his capital remains nonnegative.
Merkle et al.\ leave this an exercise for the reader; I give an intuitive
argument. It is well-known in probability theory that in a betting
strategy one can combine bets for the same effect. (Formally, this
is the martingale stopping theorem---see \cite{Williams1991}.) Hence
instead of separately betting on $[\sigma_{0}]_{\mathcal{A}},\ldots,[\sigma_{k}]_{\mathcal{A}}$
the gambler will have the same capital as if he just bet on the union
$[\sigma_{0}]_{\mathcal{A}}\cup\ldots\cup[\sigma_{k}]_{\mathcal{A}}$.
In the later case, the proper wager would be. 
\[
\frac{\mu([\sigma_{0}]_{\mathcal{A}}\cup\ldots\cup[\sigma_{k}]_{\mathcal{A}})}{\mu(\mathcal{X}\smallsetminus([\sigma_{0}]_{\mathcal{A}}\cup\ldots\cup[\sigma_{k}]_{\mathcal{A}}))}\leq1,
\]
The inequality follows from 
\[
\mu([\sigma_{0}]_{\mathcal{A}}\cup\ldots\cup[\sigma_{k}]_{\mathcal{A}})\leq1/2\leq\mu(\mathcal{X}\smallsetminus([\sigma_{0}]_{\mathcal{A}}\cup\ldots\cup[\sigma_{k}]_{\mathcal{A}})).
\]
Hence the gambler never wagers (and so loses) more than his starting
capital of $1$.

Now, assume $z\in\mathcal{X}$ is not Martin-Löf random. Let $(U_{k})$
be a Martin-Löf test which covers $z$. We may assume $(U_{k})$ is
decreasing. The betting strategy will be as follows. Let $x\in\mathcal{X}$
be the sequence we are betting on. Since $\mu(U_{1})<1/2$ we can
start with the computable betting strategy above which will reach
a capital of $2$ if $x\in U_{1}$. (Recall, we are not actually betting
on $U_{1}$, but the a.e.\ equal set $\bigcup_{i}[\sigma_{i}]_{\mathcal{A}}$.
This is not an issue, since the difference is a null $\Sigma_{2}^{0}$
set. If $x$ is in the difference, then $x$ is not computably random,
and so not betting random.)

Now, if the capital of $2$ is never reached then $x\notin U_{1}$
and $x$ is random. However, if the capital of $2$ is reached (in
a finite number of steps) then we know that $x\in[\sigma]_{\mathcal{A}}$
for some $\sigma=\sigma_{i}$ (and no other). Further, by the assumptions
in the above construction, $\mu([\sigma]_{\mathcal{A}})>2^{-k}$ for
some $k$. Then we can repeat the first step, but now we bet that
$x\in U_{k+1}$ and attempt to double our capital to $4$. Since $\mu(U_{k+1}\mid[\sigma]_{\mathcal{A}})\leq1/2$,
the capital will remain positive. 

Continuing this strategy for capitals of $8,16,32,\ldots$ we have
a computable betting strategy which succeeds on $z$.\end{proof}
\begin{rem}
\label{rem:universal-strategy}Since there is a universal Martin-Löf
test $(U_{k})$, there is a universal computable betting strategy.
(The null $\Sigma_{2}^{0}$ set of exceptions can be handled by being
more careful. Choose $\mathcal{A}$ to be basis for the topology,
and combine the null cells $[\sigma]_{\mathcal{A}}$ with non-null
cells $[\tau]_{\mathcal{A}}$.) However, note that this universal
strategy is very different from that of Kol\-mo\-go\-rov-Love\-land
randomness. This is the motivation for the next section.
\end{rem}
It is also possible to characterize Martin-Löf randomness by computable
randomness. First I give a more formal definition of computable betting
strategy.
\begin{rem}
\label{rem:betting-strategy}Represent a computable betting strategy
as follows. There is a computably indexed family of a.e.~decidable
sets $\{A_{\sigma}\}_{\sigma\in2^{<\omega}}$. These represent the
sets being bet on after the wins/losses characterized by $\sigma\in2^{<\omega}$.
From this, we have a computably indexed family $\{B_{\sigma}\}_{\sigma\in2^{<\omega}}$
defined recursively by $B_{\varepsilon}=\mathcal{X}$, $B_{\sigma1}=B_{\sigma}\cap A_{\sigma}$
and $B_{\sigma0}=B_{\sigma}\cap(\mathcal{X}\smallsetminus A_{\sigma})$.
This represents the known information after the wins/losses characterized
by $\sigma\in2^{<\omega}$. It is easy to see that $B_{\sigma0}\cap B_{\sigma1}=\varnothing$
and $B_{\sigma0}\cup B_{\sigma1}=B_{\sigma}$ a.e. Then a computable
betting strategy can be represented as a partial computable martingale
$M\colon{\subseteq{}}2^{<\omega}\rightarrow[0,\infty)$ such that
\[
M(\sigma0)\mu(B_{\sigma0})+M(\sigma1)\mu(B_{\sigma1})=M(\sigma)\mu(B_{\sigma})
\]
and if $\sigma\notin\dom M$ then $\mu(B_{\sigma0})=0$. Again, $M(\sigma)$
represents the capital after a state of $\sigma$ wins/losses. Say
the strategy \noun{succeeds} on $x$ if there is some strictly-increasing
chain $\sigma_{0}\prec\sigma_{1}\prec\sigma_{2}\prec\ldots$ from
$2^{<\omega}$ such that $\limsup_{n\rightarrow\infty}M(\sigma_{n})=\infty$
and $x\in B_{\sigma_{n}}$ for all $n$. Then $x\in\mathcal{X}$ is
betting random if there does not exist some $\{A_{\sigma}\}_{\sigma\in2^{<\omega}}$
and $M$ as above which succeed on $x$.\end{rem}
\begin{lem}
\label{lem:betting-to-mart-on-Cantor}Fix a computable probability
space $(\mathcal{X},\mu)$. For each computable betting strategy there
is a computable probability measure $\nu$ on $2^{\omega}$, a morphism
$T\colon(\mathcal{X},\mu)\rightarrow(2^{\omega},\nu)$, and an a.e.\ computable
martingale $M$ on $(2^{\omega},\nu)$ such that if this betting strategy
succeeds on $x$, then the martingale $M$ succeeds on $T(x)$, and
therefore $T(x)$ is not computably random with respect to $(2^{\omega},\nu)$.\end{lem}
\begin{proof}
Fix a computable betting strategy. Let $M\colon{\subseteq{}}2^{<\omega}\rightarrow[0,\infty)$
and $\{B_{\sigma}\}_{\sigma\in2^{<\omega}}$ be as in Remark~\ref{rem:betting-strategy}.
Then define $(2^{\omega},\nu)$ by $\nu(\sigma)=\mu(B_{\sigma})$.
Also, let $T(x)$ map $x$ to the $y\in2^{\omega}$ such that $x\in B_{y\upharpoonright n}$
for all $n$. Then $T$ is a morphism, $M$ also represents a martingale
on $(2^{\omega},\nu)$, and if the betting strategy succeeds on $x$
then $M$ succeeds on $T(x)$.
\end{proof}
We now have the following characterizations of Martin-Löf randomness.
\begin{cor}
\label{cor:char_of_ML_randomness}For a computable probability space
$(\mathcal{X},\mu)$, the following are equivalent for $x\in\mathcal{X}$.
\begin{enumerate}
\item $x$ is Martin-Löf random.
\item No computable betting strategy succeeds on $x$ (i.e.~$x$ is betting
random).
\item For all isomorphisms $T\colon(\mathcal{X},\mu)\rightarrow(2^{\omega},\nu)$,
$T(x)$ is ``Kol\-mo\-go\-rov-Love\-land random'' on $(2^{\omega},\nu)$
(i.e.\ the randomness from Section~\emph{\ref{sub:Step-1-KL}}).
\item For all morphisms $T\colon(\mathcal{X},\mu)\rightarrow(2^{\omega},\nu)$,
$T(x)$ is computably random with respect to $(2^{\omega},\nu)$.
\end{enumerate}
\end{cor}
\begin{proof}
The equivalence of (1) and (2) is Theorem~\ref{thm:Betting-random-is-ML}.
(1) implies both (3) and (4) since morphisms preserve Martin-Löf randomness
(Proposition~\ref{prop:morphisms_preserve}). 

(4) implies (2): Use Lemma~\ref{lem:betting-to-mart-on-Cantor}.
Assume $x$ is not betting random. Then there is some morphism $T$
such that $T(x)$ is not computably random.

(3) implies (2): Recall that the definition of betting randomness
came from applying the method of Section~\ref{sec:Generalizing-randomness}
to Kol\-mo\-go\-rov-Love\-land randomness. By method~(3) of Step~2
in Section~\ref{sec:Generalizing-randomness}, $x$ is betting random
if and only if (3) holds. (An alternate proof would be to modify Lemma~\ref{lem:betting-to-mart-on-Cantor}.)\end{proof}
\begin{cor}
\label{cor:comp-rand-morphisms}Computable randomness is not preserved
under morphisms. (See comments after Proposition~\emph{\ref{prop:morphisms_preserve}}.)\end{cor}
\begin{proof}
It is well-known that there is an $x\in2^{\omega}$ which is computably
random with respect to $(2^{\omega},\lambda)$ but not Martin-Löf
random (see \cite{Downey2010,Nies2009}). Then by Corollary~\ref{cor:char_of_ML_randomness},
there is some morphism $T$ such that $T(x)$ is not computably random.
\end{proof}
Corollary~\ref{cor:comp-rand-morphisms} was also proved by Bienvenu
and Porter \cite{BienvenuSubmitted}.

\subsection{Step 3: Is the new definition consistent with the former?}

To show that Martin-Löf randomness equals Kol\-mo\-go\-rov-Love\-land
randomness on $(2^{\omega},\lambda)$, it is sufficient to affirmatively
answer the following question.
\begin{question}
In Step 1 (Section~\emph{\ref{sub:Step-1-KL}}), Kol\-mo\-go\-rov-Love\-land
randomness was extended to other computable probability measures $\mu$
on $2^{\mathbb{\omega}}$. Is this extension of Kol\-mo\-go\-rov-Love\-land
randomness preserved under isomorphisms between computable probability
measures on $2^{\omega}$?
\end{question}
This is related to the question of whether Kol\-mo\-go\-rov-Love\-land
randomness with respect to $(2^{\omega},\lambda)$ is base invariant
(see Examples~\ref{ex:Base-inv-1} and \ref{ex:Base-inv-2}), which
is an open question.

\section{Endomorphism randomness\label{sec:Endomorphism-randomness}}

The generalization of Kol\-mo\-go\-rov-Love\-land randomness given
in the last section was, in some respects, not very satisfying. In
particular, the definition of Kol\-mo\-go\-rov-Love\-land randomness
with respect to $(2^{\omega},\lambda)$ assumes each event being bet
on is independent of all the previous events, and further has conditional
probability $1/2$. Therefore, at the ``end'' of the gambling session,
regardless of how much the gambler has won or lost, he knows what
$x$ is up to a measure-zero set (where $x$ is the sequence being
bet on). This is in contrast to the universal betting strategy given
in the proof of Theorem~\ref{thm:Betting-random-is-ML} (see Remark~\ref{rem:universal-strategy}),
which only narrows $x$ down to a positive measure set when $x$ is
Martin-Löf random.

In this section, I now give a new type of randomness which behaves
more like Kol\-mo\-go\-rov-Love\-land randomness. This randomness
notion can be defined using both morphisms and betting strategies.
\begin{defn}
Let $(\mathcal{X},\mu)$ be a computable probability space. An \noun{endomorphism}
on $(\mathcal{X},\mu)$ is a morphism from $(\mathcal{X},\mu)$ to
itself. Say $x\in\mathcal{X}$ is \noun{endomorphism random} if for
all endomorphisms $T\colon(\mathcal{X},\mu)\rightarrow(\mathcal{X},\mu)$,
we have that $T(x)$ is computably random.
\end{defn}
Notice the above definition is the same as that given in Corollary~\ref{cor:char_of_ML_randomness}\,(4),
except that ``morphism'' is replaced with ``endomorphism''.

If the space is atomless, we have an alternate characterization.
\begin{prop}
\label{prop:Endo-random-by-morphisms}Let $(\mathcal{X},\mu)$ be
a computable probability space with no atoms. Then $x\in\mathcal{X}$
is endomorphism random if and only if for all morphisms $T\colon(\mathcal{X},\mu)\rightarrow(2^{\omega},\lambda)$,
$T(x)$ is computably random.\end{prop}
\begin{proof}
Use that there is an isomorphism from $(\mathcal{X},\mu)$ to $(2^{\omega},\lambda)$
(Theorem~\ref{prop:atomless-isomorphic}) and that isomorphisms preserve
computable randomness (Theorem~\ref{thm:isomorphisms_preserve}).
\end{proof}
Also, we can define endomorphism randomness using computable betting
strategies as in the previous section. 
\begin{defn}
Let $(\mathcal{X},\mu)$ be an atomless computable probability space.
Consider a computable betting strategy $\mathcal{B}$. Let $\{A_{\sigma}\}_{\sigma\in2^{<\omega}},\{B_{\sigma}\}_{\sigma\in2^{<\omega}}$
be as in Remark~\ref{rem:betting-strategy}. Call the betting strategy
$\mathcal{B}$ \noun{balanced} if it only bets on events with conditional
probability $\frac{1}{2}$, conditioned on $B_{\sigma}$ (the information
known by the gambler at after the wins/losses given by $\sigma$).
In other words, $\mu(A_{\sigma}\mid B_{\sigma})=1/2$. Call the betting
strategy $\mathcal{B}$ \noun{exhaustive} if $\mu(B_{\sigma_{n}})\rightarrow0$
for any strictly increasing chain $\sigma_{0}\prec\sigma_{1}\prec\ldots$.
In other words the measure of the information known about $x$ approaches
$0$.\end{defn}
\begin{thm}
\label{thm:Endo_char}Let $(\mathcal{X},\mu)$ be an atomless computable
probability space and $x\in\mathcal{X}$. The following are equivalent.
\begin{enumerate}
\item $x$ is endomorphism random.
\item There does not exist a balanced computable betting strategy which
succeeds on $x$.
\item There does not exist an exhaustive computable betting strategy which
succeeds on $x$.
\end{enumerate}
\end{thm}
\begin{proof}
(3) implies (2) since balanced betting strategies are exhaustive.
For (2) implies (1), assume $x$ is not endomorphism random. Then
there is some morphism $T\colon(\mathcal{X},\mu)\rightarrow(2^{\omega},\lambda)$
such that $T(x)$ is not computably random. Hence there is a computable
martingale $M$ which succeeds on $T(x)$. We can also assume this
martingale is rational valued (that is $M$ is a computable function
of type $M\colon2^{<\omega}\rightarrow\mathbb{Q}$ where $\mathbb{Q}$
is identified with $\mathbb{N}$), so it is clear what bit is being
bet on. This martingale on $(2^{\omega},\lambda)$ can be pulled back
to a computable betting strategy on $(\mathcal{X},\mu)$ (use the
proof of Lemma~\ref{lem:betting-to-mart-on-Cantor}, except in reverse).
This betting strategy is balanced since $M$ is a balanced ``dyadic''
martingale.

For (1) implies (3), assume there is some computable, exhaustive betting
strategy which succeeds on $x$. Then from this strategy we can construct
a morphism $S\colon(\mathcal{X},\mu)\rightarrow([0,1],\lambda)$ recursively
as follows. Each $B_{\sigma}$ will be mapped to an open interval
$(a,b)$ of length $\mu(B_{\sigma})$. First, map $S(B_{\varepsilon})=(0,1)$.
For the recursion step, assume $S(B_{\sigma})=(a,b)$ of length $\mu(B_{\sigma})$.
Set $S(B_{\sigma0})=(a,a+\mu(B_{\sigma0}))$ and $S(B_{\sigma1})=(a+\mu(B_{\sigma0}),b)$.
This function $S$ is well-defined and computable since the betting
strategy is exhaustive. Also, $S$ is clearly measure-preserving,
so it is a morphism. Then using the usual isomorphism from $([0,1],\lambda)$
to $(2^{\omega},\lambda)$, we can assume $S$ is a morphism to $(2^{\omega},\lambda)$.
Moreover, the set of images $S(B_{\sigma})$ describes a cell decomposition
$\mathcal{A}$ of $(2^{\omega},\lambda)$, and the betting strategy
can be pushed forward to give a martingale on $(2^{\omega},\lambda)$
with respect to $\mathcal{A}$ (similar to the proof of Lemma~\ref{lem:betting-to-mart-on-Cantor}).
\end{proof}
Now we can relate endomorphism randomness to Kol\-mo\-go\-rov-Love\-land
randomness.
\begin{cor}
\label{cor:ER_KL}On $(2^{\omega},\lambda)$, every endomorphism random
is Kol\-mo\-go\-rov-Love\-land random.\end{cor}
\begin{proof}
Every nonmonotonic, computable betting strategy on bits is a balanced
betting strategy. Hence every endomorphism random is Kol\-mo\-go\-rov-Love\-land
random.\end{proof}
\begin{cor}
Let $(\mathcal{X},\mu)$ be a computable probability space with no
atoms. Then $x\in\mathcal{X}$ is endomorphism random with respect
to $(\mathcal{X},\mu)$ if and only if for all morphisms $T\colon(\mathcal{X},\mu)\rightarrow(2^{\omega},\lambda)$,
$T(x)$ is Kol\-mo\-go\-rov-Love\-land random with respect to
$(2^{\omega},\lambda)$.\end{cor}
\begin{proof}
If $x$ is endomorphism random with respect to $(\mathcal{X},\mu)$,
then so is $T(x)$ with respect to $(2^{\omega},\lambda)$. By Corollary~\ref{cor:ER_KL},
$T(x)$ is Kol\-mo\-go\-rov-Love\-land random with respect to
$(2^{\omega},\lambda)$. If $T(x)$ is Kol\-mo\-go\-rov-Love\-land
random with respect to $(2^{\omega},\lambda)$ for all morphisms $T\colon(\mathcal{X},\mu)\rightarrow(2^{\omega},\lambda)$,
then $T(x)$ is computably random with respect to $(2^{\omega},\lambda)$
for all such $T$. Therefore, $x$ is endomorphism random with respect
to $(\mathcal{X},\mu)$.\end{proof}
\begin{cor}
\label{cor:ER_KL_CR}Computable randomness is not preserved by endomorphisms.\end{cor}
\begin{proof}
It is well-known that there exists a computable random $x\in(2^{\omega},\lambda)$
which is not Kol\-mo\-go\-rov-Love\-land random (see \cite{Downey2010,Nies2009}).
Then $x$ is not endomorphism random with respect to $(2^{\omega},\lambda)$.
By definition, $T(x)$ is not computably random for some morphism
$T\colon(2^{\omega},\lambda)\rightarrow(2^{\omega},\lambda)$.
\end{proof}
Also, clearly each betting random (i.e.~each Martin-Löf random) is
an endomorphism random.

I will add one more randomness notion. Say $x\in2^{\omega}$ is \noun{automorphism random}
(with respect to $(2^{\omega},\lambda)$) if for all automorphisms
$T\colon(2^{\omega},\lambda)\rightarrow(2^{\omega},\lambda)$, $T(x)$
is Kol\-mo\-go\-rov-Love\-land random. It is clear that for $(2^{\omega},\lambda)$
we have
\begin{multline}
\text{Martin-Löf random}\ \rightarrow\ \text{Endomorphism random}\\
\rightarrow\ \text{Automorphism random}\ \rightarrow\ \text{Kolmogorov-Loveland random}.\label{eq:ran-hierarchy-2}
\end{multline}
We now have a more refined version of Question~\ref{open:KL_equals_ML}.
\begin{question}
Do any of the implications in formula~\emph{(\ref{eq:ran-hierarchy-2})}
reverse?\footnote{Recently, and independently of my work, Tomislav Petrovic has claimed
that there are two balanced betting strategies on $(2^{\omega},\lambda)$
such that if a real $x$ is not Martin-Löf random, then at least one
of the two strategies succeeds on $x$. In particular, Petrovic's
result, which is in preparation, would imply that endomorphism randomness
equals Martin-Löf randomness. Further, via the proof of Theorem~\ref{thm:Endo_char},
this result would extend to every atomless computable probability
space.}
\end{question}

\section{Further directions\label{sec:Further-directions}}

Throughout this paper I was working with a.e.~computable objects:
a.e.~decidable sets, a.e.~decidable cell decompositions, a.e.~computable
morphisms, and Kurtz randomness---which as I showed, can be defined
by a.e.~computability. Recall a.e.~decidable sets are only sets
of $\mu$-continuity, and a.e.~computable morphisms are only a.e.~continuous
maps.

The ``next level'' is to consider the computable Polish spaces of
measurable sets and measurable maps. The a.e.~decidable sets and
a.e.~computable maps are dense in these spaces. Hence, in the definitions,
one may replace a.e.~decidable sets, a.e.~decidable cell decompositions,
a.e.~computable morphisms, and Kurtz randomness with effectively
measurable sets, decompositions into effectively measurable cells,
effectively measurable measure-preserving maps, and Schnorr randomness.
(This is closely related to the work of Pathak, Simpson and Rojas
\cite{Pathak:2014fk}; Miyabe \cite{Miyabe:2013uq}; Hoyrup and Rojas
{[}personal communication{]}; and the author on ``Schnorr layerwise-computability''
and convergence for Schnorr randomness.) Indeed, the results of this
paper remain true, even with those changes. However, some proofs change
and I will give the results in a later paper.

An even more general extension would be to ignore the metric space
structure all together. Any standard probability space space can be
described uniquely by the measures of an intersection-closed class
of sets, or a $\pi$-system, which generates the measure algebra of
the measure. From this, one can obtain a cell decomposition. In the
case of a computable probability space (Definition~\ref{def:comp_prob_meas}),
each a.e.~decidable generator closed under intersections is a $\pi$-system.
The definition of computable randomness with respect to such a general
space would be the analog of the definition in this paper. 

In particular, this would allow one to define computable randomness
with respect to effective topological spaces with measure \cite{Hertling2003}.
In this case the $\pi$-system is the topological basis. This also
allows one to define Schnorr, Martin-Löf, and weak-2 randomness as
well, namely replace, say, $\Sigma_{1}^{0}$ sets in the definition
with effective unions of sets in the $\pi$-system. This agrees with
most definitions of, say, Martin-Löf randomness in the literature.\footnote{Some authors \cite{Axon2010,Hertling2003} define Martin-Löf randomness
via open covers, even for non-regular topological spaces. Such definitions
have unnatural properties \cite{Miyabe:2014fr}, where as my definition
is more natural. All these methods agree for spaces with an effective
regularity condition.}

Using $\pi$-systems also allows one to define ``abstract'' measure
spaces without points. The computable randoms then become ``abstract
points'' given by generic ultrafilters on the Boolean algebra of
measurable sets \emph{a la} Solovay forcing.

Another possible generalization is to non-computable probability spaces
(on computable Polish spaces). This has been done by Levin \cite{Levin:1976uq}
and extended by others (see \cite{Gacs2005,Bienvenu2011}) for Martin-Löf
randomness in a natural way using \noun{uniform tests} which are
total computable functions from measures to tests. Possibly a similar
approach would work for computable randomness. For example, for $2^{\omega}$,
a uniform test for computable randomness would be a total computable
map $\mu\mapsto\nu$ where $\nu$ is the bounding measure for $\mu$.
This map is enough to define a uniform martingale test for each $\mu$
given by $\nu(\sigma)/\mu(\sigma)$. (I showed in Section~\ref{sec:CR-on-Cantor}
that this martingale is uniformly computable.) Uniform tests for Schnorr
and computable randomness have been used by Miyabe \cite{Miyabe2011}. 

Also, what other applications for a.e.~decidable sets are there in
effective probability theory? The method of Section~\ref{sec:Generalizing-randomness}
basically allows one to treat every computable probability space as
the Cantor space. It is already known that the indicator functions
of a.e.~decidable sets can be used to define $L^{1}$-computable
functions \cite{Miyabe:2013uq}.

However, when it comes to defining classes of points, the method of
Section~\ref{sec:Generalizing-randomness} is specifically for defining
\emph{random} points since such a definition must be a subclass of
the Kurtz randoms. Under certain circumstances, however, one may be
able to use related methods to generalize other definitions. For example,
is the following a generalization of K-triviality to arbitrary computable
probability spaces? Let $K=K_{M}$ where $M$ is a universal prefix-free
machine. Recall, a sequence $x\in2^{\omega}$ is \noun{K-trivial}
(on $(2^{\omega},\lambda)$) if there is some $b$ such that 
\[
\forall n\ K(x\upharpoonright n)\leq K(n)+b
\]
where $K(n)=K(0^{n})$ and $0^{n}$ is the string of $0$'s of length
$n$. Taking a clue from Section~\ref{sec:Kolmogorov-complexity-and-randomness},
call a point $x\in(\mathcal{X},\mu)$ \noun{K-trivial} if there is
some cell decomposition $\mathcal{A}$ and some $b$ such that for
all $n$, 
\[
K(x\upharpoonright_{\mathcal{A}}n)<K(-\log\mu([x\upharpoonright_{\mathcal{A}}n]_{\mathcal{A}}))+b.
\]
(Here we assume $K(\infty)=\infty$.) Does the $\mathcal{A}$-name
or Cauchy-name of $x$ satisfy the other nice degree theoretic properties
of K-triviality, such as being low-for-$(\mathcal{X},\mu)$-random?
(Here I say a Turing degree $\mathbf{d}$ is low-for-$(\mathcal{X},\mu)$-random
if when used as an oracle, $\mathbf{d}$ does not change the class
of Martin-Löf randoms in $(\mathcal{X},\mu)$. Say a point $x\in(\mathcal{X},\mu)$
is low-for-$(\mathcal{X},\mu)$-random if its Turing degree is.) 

If it is a robust definition, how does it relate to the definition
of Melnikov and Nies \cite{Melnikov:2013vn} generalizing K-triviality
to computable Polish spaces (as opposed to probability spaces)? I
conjecture that their definition is equivalent to being low-for-$(\mathcal{X},\mu)$-random
for every computable probability measure $\mu$ of $\mathcal{X}$.

Last, isomorphisms and morphisms offer a useful tool to classify randomness
notions. One may ask what randomness notions (defined for all computable
probability measures on $2^{\omega}$) are invariant under morphisms
or isomorphisms? By Proposition~\ref{prop:morphisms_preserve}, Martin-Löf,
Schnorr, and Kurtz randomness are invariant under morphisms. (This
can easily be extended to $n$-randomness, weak $n$-randomness, and
difference randomness. See \cite{Downey2010,Nies2009} for definitions.)
However, by Corollary~\ref{cor:char_of_ML_randomness}\,(4), there
is no randomness notion between Martin-Löf randomness and computable
randomness that is invariant under morphisms. Is there such a randomness
notion between Schnorr randomness and Martin-Löf randomness? Further,
by Theorem~\ref{thm:isomorphisms_preserve} computable randomness
is invariant under isomorphisms. André Nies pointed out to me that
this is not true of partial computable randomness since it it not
invariant under permutations \cite{Bienvenu2010}. In general what
can be said of full-measure, isomorphism-invariant sets of a computable
probability space $(\mathcal{X},\mu)$? The notions of randomness
connected to computable analysis will most likely be the ones that
are invariant under isomorphisms.\footnote{There is at least one exception to this rule. Avigad \cite{Avigad:2013kx}
discovered that the randomness notion, called UD randomness, characterized
by a theorem of Weyl is incomparable with Kurtz randomness; and therefore,
it is not even preserved by automorphisms.}


\subsection{Acknowledgments}

I would like to thank Mathieu Hoyrup and Cristóbal Rojas for suggesting
to me the use of isomorphisms to define computable randomness outside
Cantor space. I would also like to thank Jeremy Avigad, Laurent Bienvenu,
Peter Gács, Mathieu Hoyrup and Bjørn Kjos-Hanssen for helpful comments
on earlier drafts, as well as Joseph Miller and André Nies for helpful
comments on this research. Last, I would like to thank the reviewer
for their thorough and helpful review.

\bibliographystyle{plain}
\bibliography{comp_rand_paper}

\def\cprime{$'$}
\begin{thebibliography}{10}

\bibitem{Avigad:2013kx}
Jeremy Avigad.
\newblock Uniform distribution and algorithmic randomness.
\newblock {\em J. Symbolic Logic}, 78(1):334--344, 2013.

\bibitem{Axon2010}
Logan~M. Axon.
\newblock {\em Algorithmically random closed sets and probability}.
\newblock ProQuest LLC, Ann Arbor, MI, 2010.
\newblock Thesis (Ph.D.)--University of Notre Dame.

\bibitem{Barmpalias2007}
George Barmpalias, Paul Brodhead, Douglas Cenzer, Seyyed Dashti, and Rebecca
  Weber.
\newblock Algorithmic randomness of closed sets.
\newblock {\em J. Logic Comput.}, 17(6):1041--1062, 2007.

\bibitem{Bienvenu2011}
Laurent Bienvenu, Peter G{\'a}cs, Mathieu Hoyrup, Crist{\'o}bal Rojas, and
  A.~Shen.
\newblock Algorithmic tests and randomness with respect to a class of measures.
\newblock {\em Proceedings of the Steklov Institute of Mathematics},
  274:41--102, 2011.

\bibitem{Bienvenu2010}
Laurent Bienvenu, Rupert H{\"o}lzl, Thorsten Kr{\"a}ling, and Wolfgang Merkle.
\newblock Separations of non-monotonic randomness notions.
\newblock {\em J. Logic and Computation}, 2010.

\bibitem{Bienvenu2009}
Laurent Bienvenu and Wolfgang Merkle.
\newblock Constructive equivalence relations on computable probability
  measures.
\newblock {\em Ann. Pure Appl. Logic}, 160(3):238--254, 2009.

\bibitem{BienvenuSubmitted}
Laurent Bienvenu and Christopher Porter.
\newblock Strong reductions in effective randomness.
\newblock {\em Theoret. Comput. Sci.}, 459:55--68, 2012.

\bibitem{Bienvenu:2009uq}
Laurent Bienvenu, Glenn Shafer, and Alexander Shen.
\newblock On the history of martingales in the study of randomness.
\newblock {\em J. {\'E}lectron. Hist. Probab. Stat.}, 5(1):40, 2009.

\bibitem{Bishop1985}
Errett Bishop and Douglas Bridges.
\newblock {\em Constructive analysis}.
\newblock 279. Springer-Verlag, 1985.

\bibitem{Bosserhoff2008}
Volker Bosserhoff.
\newblock Notions of probabilistic computability on represented spaces.
\newblock {\em J. Universal Computer Science}, 14(6):956--995, 2008.

\bibitem{Brattka2001}
Vasco Brattka.
\newblock Computable versions of {B}aire's category theorem.
\newblock In Jiri Sgall, Ales Pultr, and Petr Kolman, editors, {\em
  Mathematical foundations of computer science, 2001 ({M}ari\'ansk\'e
  {L}\'azn\u e)}, volume 2136 of {\em Lecture Notes in Comput. Sci.}, pages
  224--235. Springer, Berlin, 2001.

\bibitem{Brattka2008}
Vasco Brattka, Peter Hertling, and Klaus Weihrauch.
\newblock A tutorial on computable analysis.
\newblock In {\em New computational paradigms}, pages 425--491. Springer, New
  York, 2008.

\bibitem{Brattka2011}
Vasco Brattka, Joseph~S. Miller, and Andr{\'e} Nies.
\newblock Randomness and differentiability.
\newblock {\em Trans. Amer. Math. Soc.}
\newblock To appear.

\bibitem{Downey2004a}
Rodney~G Downey, Evan Griffiths, and Geoffrey LaForte.
\newblock On {S}chnorr and computable randomness, martingales, and machines.
\newblock {\em Math. Logic Quart.}, 50(6):613--627, 2004.

\bibitem{Downey2010}
Rodney~G. Downey and Denis~R. Hirschfeldt.
\newblock {\em Algorithmic randomness and complexity}.
\newblock Theory and Applications of Computability. Springer, New York, 2010.

\bibitem{Freer2011}
Cameron~E. Freer and Daniel~M. Roy.
\newblock Computable de {F}inetti measures.
\newblock {\em Annals of Pure and Applied Logic}, 2011.

\bibitem{Gacs:1980fj}
P{\'e}ter G{\'a}cs.
\newblock Exact expressions for some randomness tests.
\newblock {\em Z. Math. Logik Grundlag. Math.}, 26(5):385--394, 1980.

\bibitem{Gacs2005}
Peter G{\'a}cs.
\newblock Uniform test of algorithmic randomness over a general space.
\newblock {\em Theoret. Comput. Sci.}, 341(1-3):91--137, 2005.

\bibitem{Gacs2011}
Peter G{\'a}cs, Mathieu Hoyrup, and Crist{\'o}bal Rojas.
\newblock Randomness on computable probability spaces---a dynamical point of
  view.
\newblock {\em Theory Comput. Syst.}, 48(3):465--485, 2011.

\bibitem{Hemmerling:2002aa}
Armin Hemmerling.
\newblock Effective metric spaces and representations of the reals.
\newblock {\em Theoret. Comput. Sci.}, 284(2):347--372, 2002.
\newblock Computability and complexity in analysis (Castle Dagstuhl, 1999).

\bibitem{Hertling:1997fk}
Peter Hertling and Yongge Wang.
\newblock Invariance properties of random sequences.
\newblock {\em J.UCS}, 3(11):1241--1249 (electronic), 1997.

\bibitem{Hertling2003}
Peter Hertling and Klaus Weihrauch.
\newblock Random elements in effective topological spaces with measure.
\newblock {\em Inform. and Comput.}, 181(1):32--56, 2003.

\bibitem{Hitchcock2006}
John~M. Hitchcock and Jack~H. Lutz.
\newblock Why computational complexity requires stricter martingales.
\newblock {\em Theory Comput. Syst.}, 39(2):277--296, 2006.

\bibitem{Hoyrup2009}
Mathieu Hoyrup and Crist{\'o}bal Rojas.
\newblock Computability of probability measures and {M}artin-{L}\"of randomness
  over metric spaces.
\newblock {\em Inform. and Comput.}, 207(7):830--847, 2009.

\bibitem{Kastermans2010}
Bart Kastermans and Steffen Lempp.
\newblock Comparing notions of randomness.
\newblock {\em Theoret. Comput. Sci.}, 411(3):602--616, 2010.

\bibitem{Kjos-Hanssen:2009hc}
Bj{\o}rn Kjos-Hanssen and Anil Nerode.
\newblock Effective dimension of points visited by {B}rownian motion.
\newblock {\em Theoret. Comput. Sci.}, 410(4-5):347--354, 2009.

\bibitem{Levin:1976uq}
L.~A. Levin.
\newblock Uniform tests for randomness.
\newblock {\em Dokl. Akad. Nauk SSSR}, 227(1):33--35, 1976.

\bibitem{Melnikov:2013vn}
Alexander Melnikov and Andr{{\'e}} Nies.
\newblock {$K$}-triviality in computable metric spaces.
\newblock {\em Proc. Amer. Math. Soc.}, 141(8):2885--2899, 2013.

\bibitem{Merkle2003}
Wolfgang Merkle.
\newblock The {K}olmogorov-{L}oveland stochastic sequences are not closed under
  selecting subsequences.
\newblock {\em J. Symbolic Logic}, 68(4):1362--1376, 2003.

\bibitem{Merkle2006}
Wolfgang Merkle, Nenad Mihailovi{\'c}, and Theodore~A. Slaman.
\newblock Some results on effective randomness.
\newblock {\em Theory Comput. Syst.}, 39(5):707--721, 2006.

\bibitem{Merkle2006a}
Wolfgang Merkle, Joseph~S. Miller, Andr{\'e} Nies, Jan Reimann, and Frank
  Stephan.
\newblock Kolmogorov-{L}oveland randomness and stochasticity.
\newblock {\em Ann. Pure Appl. Logic}, 138(1-3):183--210, 2006.

\bibitem{Miller:2004ly}
Joseph~S. Miller.
\newblock Degrees of unsolvability of continuous functions.
\newblock {\em J. Symbolic Logic}, 69(2):555--584, 2004.

\bibitem{Miyabe2011}
Kenshi Miyabe.
\newblock Truth-table {S}chnorr randomness and truth-table reducible
  randomness.
\newblock {\em MLQ Math. Log. Q.}, 57(3):323--338, 2011.

\bibitem{Miyabe:2013uq}
Kenshi Miyabe.
\newblock {$L^1$}-computability, layerwise computability and {S}olovay
  reducibility.
\newblock {\em Computability}, 2(1):15--29, 2013.

\bibitem{Miyabe:2014fr}
Kenshi Miyabe.
\newblock Algorithmic randomness over general spaces.
\newblock {\em Mathematical Logic Quarterly}, 60(3):184--204, 2014.

\bibitem{Muchnik:1998rt}
Andrei~A. Muchnik, Alexei~L. Semenov, and Vladimir~A. Uspensky.
\newblock Mathematical metaphysics of randomness.
\newblock {\em Theoret. Comput. Sci.}, 207(2):263--317, 1998.

\bibitem{Nies2009}
Andr{\'e} Nies.
\newblock {\em Computability and randomness}, volume~51 of {\em Oxford Logic
  Guides}.
\newblock Oxford University Press, Oxford, 2009.

\bibitem{Pathak:2014fk}
Noopur Pathak, Crist{{\'o}}bal Rojas, and Stephen~G. Simpson.
\newblock Schnorr randomness and the {L}ebesgue differentiation theorem.
\newblock {\em Proc. Amer. Math. Soc.}, 142(1):335--349, 2014.

\bibitem{Pour-El1989}
Marian~B. Pour-El and J.~Ian Richards.
\newblock {\em Computability in analysis and physics}.
\newblock Perspectives in Mathematical Logic. Springer-Verlag, Berlin, 1989.

\bibitem{Schnorr:1977fk}
C.-P. Schnorr and P.~Fuchs.
\newblock General random sequences and learnable sequences.
\newblock {\em J. Symbolic Logic}, 42(3):329--340, 1977.

\bibitem{Schnorr1971}
Claus-Peter Schnorr.
\newblock A unified approach to the definition of random sequences.
\newblock {\em Math. Systems Theory}, 5:246--258, 1971.

\bibitem{Schroder:2007kx}
Matthias Schr{\"o}der.
\newblock Admissible representations for probability measures.
\newblock {\em MLQ Math. Log. Q.}, 53(4-5):431--445, 2007.

\bibitem{Silveira:2011fk}
Javier~Gonzalo Silveira.
\newblock Invariancia por cambio de base de la aleatoriedad computable y la
  aleatoriedad con recursos acotados.
\newblock Master's thesis, University of Buenos Aires, 2011.

\bibitem{Turetsky:2012kx}
Daniel Turetsky.
\newblock Partial martingales for computable measures.
\newblock In Andr{\'e} Nies, editor, {\em Logic Blog}, 2014.
\newblock Available at
  \url{https://dl.dropboxusercontent.com/u/370127/Blog/Blog2014.pdf}.

\bibitem{Vcprimeyugin1998}
V.~V. V'yugin.
\newblock Ergodic theorems for individual random sequences.
\newblock {\em Theoret. Comput. Sci.}, 207(2):343--361, 1998.

\bibitem{Weihrauch2000}
Klaus Weihrauch.
\newblock {\em Computable analysis: An introduction}.
\newblock Texts in Theoretical Computer Science. An EATCS Series.
  Springer-Verlag, Berlin, 2000.

\bibitem{Williams1991}
David Williams.
\newblock {\em Probability with martingales}.
\newblock Cambridge Mathematical Textbooks. Cambridge University Press,
  Cambridge, 1991.

\end{thebibliography}

\end{document}